\documentclass{article}
\usepackage{amssymb}
\usepackage{verbatim}
\usepackage{array}
\usepackage{latexsym}
\usepackage{enumerate}
\usepackage{amsmath}
\usepackage{amsfonts}
\usepackage{amsthm}
\usepackage{color}
\usepackage[english]{babel}

\newtheorem{theorem}{Theorem}[section]
\newtheorem{proposition}[theorem]{Proposition}
\newtheorem{lemma}[theorem]{Lemma}

\newtheorem{claim}[theorem]{Claim}
\newtheorem{result}[theorem]{Result}
\newtheorem{rem}[theorem]{Remark}

\def\cC{\mathcal C}

\def\cE{\mathcal E}

\def\cH{\mathcal H}

\def\cQ{\mathcal Q}

\def\cX{\mathcal X}
\def\cY{\mathcal Y}
\def\cZ{\mathcal Z}

\def\K{\mathbb{K}}
\def\PG{{\rm{PG}}}
\def\ord{\mbox{\rm ord}}

\def\fq{{\mathbb F}_q}

\def\gg{\mathfrak{g}}
\newcommand{\PSL}{\mbox{ PSL}}
\newcommand{\PGL}{\mbox{ PGL}}

\newcommand{\PSU}{\mbox{ PSU}}

\newcommand{\aut}{\mbox{\rm Aut}}

\newcommand{\ha}{{\textstyle\frac{1}{2}}}
\newcommand{\thi}{\textstyle\frac{1}{3}}

\newcommand{\qa}{\textstyle\frac{1}{4}}

\title{Algebraic curves with many automorphisms}
\date{}
\author{Massimo Giulietti and G\'abor Korchm\'aros}
\begin{document}
\maketitle

%\thanks{{\em Keywords}: Algebraic curves, Positive characteristic, Automorphism groups.}

    \begin{abstract} Let $\cX$ be a (projective, geometrically irreducible, non-singular) algebraic curve of genus $\gg \ge 2$ defined over an algebraically closed field $\K$ of odd characteristic $p$. Let $\aut(\cX)$ be the group of all automorphisms of $\cX$ which fix $\K$ element-wise. It is known that if $|\aut(\cX)|\geq 8\gg^3$ then the $p$-rank (equivalently, the Hasse-Witt invariant) of $\cX$ is zero. This raises the problem of determining the (minimum-value) function $f(\gg)$ such that whenever $|\aut(\cX)|\geq f(\gg)$ then $\cX$ has zero $p$-rank. For {\em{even}} $\gg$ we prove that $f(\gg)\leq 900 \gg^2$. The {\em{odd}} genus case appears to be much more difficult although, for any genus $\gg\geq 2$, if $\aut(\cX)$ has a solvable subgroup $G$ such that $|G|>252 \gg^2$ then $\cX$ has zero $p$-rank and $G$ fixes a point of $\cX$. Our proofs use the Hurwitz genus formula and the Deuring Shafarevich formula together with a few deep results from %\textcolor{black}{Finite} 
    Group theory characterizing finite simple groups whose Sylow $2$-subgroups have a cyclic subgroup of index $2$. We also point out some connections with the Abhyankar conjecture  and the Katz-Gabber covers.
     \end{abstract}

    \section{Introduction}
 Numerous deep results on automorphism groups of algebraic curves, defined over a groundfield of characteristic zero, have been achieved in the passed 125 years following up the seminal work by Hurwitz who was the first to prove that complex curves, other than the rational and elliptic ones, can only have a finite number of automorphisms. Later, Hurwitz's result was given a characteristic--free proof and the fundamental treatment of the theory of automorphisms from the view point of Galois coverings of curves was expanded to include groundfields of positive characteristic. Nevertheless, the theory of automorphism groups of curves present several different features in positive characteristic. This is especially apparent when the curve has many automorphisms, since the classical Hurwitz bound $|G|\leq 84(\gg-1)$ on the order of the automorphism group $\aut(\cX)$ of a curve with genus $\gg\geq 2$, may fail in positive characteristic when $|G|$ is divisible by the characteristic of the groundfield $\K$. As a matter of fact, curves in positive characteristic may happen to have much larger $\K$-automorphism group compared to their genus. From previous works by Roquette \cite{roquette1970}, Stichtenoth \cite{stichtenoth1973I,stichtenoth1973II}, Henn \cite{henn1978}, and Hansen \cite{hansen1993}, we know infinite families of curves with $|\aut(\cX)|\approx c\gg^3$ and it is plausible that there exist many others with $|\aut(\cX)|\approx c\gg^2$. Although curves with large automorphism groups may have rather different features, they seem to share a common property, namely their $p$-rank is equal to zero. This raises the problem of determining a function $f(\gg)$ such that if a curve $\cX$ defined over a field of  characteristic $p>0$ has an automorphism group $G$ with $|G|\geq f(\gg)$ then $\cX$ has zero $p$-rank. \textcolor{black}{An upper} bound on $f(\gg)$ is $8\gg^3$, as a corollary to Henn's classification \cite{henn1978}, see also \cite[Theorem 11.127]{hirschfeld-korchmaros-torres2008}. Our results stated in Theorems \ref{princA} and \ref{princC} show for {\bf odd} $p$ that if either $\aut(\cX)$ is solvable or $\gg$ is even then $f(\gg)>c\gg^2$ for some constant $c$ independent of $\gg$ and $p$.

For a subgroup $G$ of $\aut(\cX)$, a related question is to determine $p(G)$ and $G/p(G)$ where $p(G)$ is the (normal) subgroup of $G$ generated by its $p$-subgroups.
The necessary part of the \emph{Abhyankar conjecture for affine curves} states that if $G$ has exactly $r$ short orbits then the factor group $G/p(G)$ has a generating set of size at most $2\bar{\gg}+r-1$ where $\bar{\gg}$ is the genus of the quotient curve $\bar{\cX}=\cX/G$. If $\cX/G$ is rational (in particular when $G$ is large), then $G=p(G)$ for $r=1$, and $G/p(G)$ is cyclic for $r=2$. The Abhyankar conjecture was stated in \cite{Ab} and proven by  Harbater \cite{Ha}. Recent surveys on this conjecture are found in \cite{pries2011} and \cite{hops}. Theorems
\ref{princA} and \ref{princC} provide a complete description of $p(G)$ when $|G|>900 \gg^2$ and either $G$ is solvable or $\gg$ is even.

For a subgroup $G$ of $\aut(\cX)$ the associated Galois cover is a Katz-Gabber $G$-cover if $\cX|G$ is rational, $G$ fixes a point $P\in\cX$, it has at most one more short orbit $o$ and $o$ is tame. Such covers are related with the Nottingham groups consisting of power series over a field $\mathbb K$ of positive characteristic with substitution as group operation, and are useful to understand how the Nottingham groups can be realized as automorphisms of a curve over $\mathbb K$; see \cite{bcps}.

\begin{theorem}
\label{princA}
Let $G$ be a solvable automorphism group of an algebraic curve $\cX$ of genus $\gg \geq 2.$ If
$$
|G|> \textstyle\frac{p}{p-2}84\gg^2
$$ then $\cX$ has zero $p$-rank. Furthermore, $G$ fixes a point of $\cX$ and $G=p(G)\rtimes C$ where $p(G)$ is a Sylow $p$-subgroup of $G$ and $C$ is a cyclic group. In particular, the cover $\cX|(\cX/G)$ is a Katz-Gabber $G$-cover.
\end{theorem}
\begin{theorem}
\label{princC}
Let $G$ be a non-solvable automorphism group of an algebraic curve $\cX$ of even genus $\gg \geq 2.$ If  $|G|>900\gg^2$ then $\cX$ has zero $p$-rank, and
$p(G)$ is isomorphic to one of the following linear groups: $PSL(2,q)$ for $q\geq 5,$ $PSU(3,q)$ for $q\equiv 1 \pmod 4,$ $SL(2,q)$ for  $q\geq 5,$ $SU(3,q)$ for $q\equiv 5 \pmod {12},$ 
\textcolor{black}{for a power $q$ of $p.$}
\end{theorem}
%The possibilities for the structure of $G$ in Theorem \ref{princC} are also %determined; see  Theorem \ref{princD}.
 In particular, $G/p(G)$ is a cyclic group according to the Abhyankar conjecture; see Proposition \ref{3apr2016} \textcolor{black}{and Theorem \ref{princB}}.
Certain quotient curves of the Hermitian and of the GK curves \cite{GK} provide examples covering all cases in Theorem \ref{princC}; see Section \ref{examples}.

From the framework of this paper, it emerges that the automorphism groups of curves with even genus are subject to several constraints which allow to determine their structures completely; see Theorem \ref{structure}. In some sense, this is an unexpected result since each finite group is the automorphism group of some curve defined over $\K$.

The ingredients of the proofs are the Hurwitz genus formula and the Deuring-Shafarevich formula together with a few deep results from %\textcolor{black}{Finite} 
Group theory characterizing finite simple groups whose Sylow $2$-subgroups have a cyclic subgroup of index $2$. We also point out that some of our lemmas are related to the Abhyankar conjecture for curves, although our proof of Theorem \ref{princC} is independent of it.

The paper is organized as follows.

Section \ref{sec2} provides the necessary background on Galois \textcolor{black}{coverings} of algebraic curves in positive characteristic, it also collects results from Group theory which play a role in the paper.

Section \ref{principA} contains the proof of Theorem \ref{princA} which is carried out by induction on $\gg$. In proving the first statement,  two cases are distinguished according to weather $G$ has a normal subgroup of order prime to $p$, or all its normal subgroups are $p$-groups. In either case we show that if a curve satisfies the hypothesis of Theorem \ref{princA} but has positive $p$-rank, then a quotient curve with respect to a well chosen normal subgroup of $G$ also satisfies the hypotheses of
Theorem \ref{princA} with genus less than $\gg$; see Propositions \ref{lemmaA} and \ref{le110luglio2005}. The proof of the second statement of Theorem \ref{princA} is given at the end of the section as a consequence of Lemmas \ref{leB11luglio2015} and \ref{lemmaC}.

Section \ref{sec4} gives some new sufficient conditions for a curve to have $p$-rank zero; see Lemmas \ref{22dic2015}, \ref{22dicter2015}, \ref{lemA29dic2015}, and \ref{lemA30dic2015}. The key achievement is Lemma \ref{22dic2015}, which can be viewed as an improvement on the bound of Theorem \ref{princA} for a large family of solvable groups. These results are used in several proofs in this paper, and appear to be of independent interest.

The rest of the paper is dedicated to the proof of Theorem \ref{princC}.

The main ingredient is the classification given in Theorem \ref{structure} where the abstract structure of $G$ is determined without any assumption on the order of $G$. The key point is Lemma \ref{lem23agos2015} which together with the classification of $2$-groups containing a cyclic subgroup of index $2$ allow us to use several deeper results from Group theory obtained within the classification project of finite simple groups.  Theorem \ref{structure} is proved first for solvable groups, then for groups without non-trivial normal subgroups of odd order (see Subsection \ref{ssec61}), and finally for the general case (see Subsection \ref{ssec62}).

From Theorem \ref{structure} it is apparent the central role of a few low dimensional linear and semilinear groups in the study of curves with even genus. This motivates the previous accurate analysis in Section \ref{sec5} in which a series of rather technical lemmas shows that if the group is large enough with respect to the genus of the curve, then it is linear (and cannot be semilinear) and the $p$-rank of the curve must be zero.

For large automorphism groups compared to the genera of the curves, the list in Theorem \ref{structure} shortens. This is showed in Subsection \ref{refinements}; see Lemma \ref{4dic2015} for odd-core free groups, that is, for groups containing no non-trivial normal subgroup of odd order, and see Proposition \ref{3apr2016} for the general case.

The proof of the first statement of Theorem \ref{princC} is given in Section \ref{sectc}. A certain effort is made to fully understand how $O(G)$ and the center $Z(G)$ and
are related; see Theorem \ref{princB}. The proof is divided into three cases, according to the list given in Proposition \ref{3apr2016}, and also uses deeper result from Group theory.

The second part of Theorem \ref{princC} is proved in Section \ref{sec8}. It essentially depends on Theorem \ref{princB} which allows to use Hering's classification \cite{hering1979} of certain $2$-transitive group spaces in order to determine the possibilities for $p(G)$. Finally, Section \ref{examples} provides examples that illustrate Theorem \ref{princC}.

\section{Background and Preliminary Results}\label{sec2}

In this paper, $\aut(\cX)$ stands for the automorphism group of a (projective, non-singular,
geometrically irreducible, algebraic) curve $\cX$ of genus $\gg\geq
2$ defined  over an algebraically closed field $\mathbb K$ of odd characteristic $p$.

For a subgroup $G$ of $\aut(\cX)$, let $\bar \cX$ denote a non-singular model of \textcolor{black}{$\K(\cX)^G$}, that is,
a projective non-singular geometrically irreducible algebraic
curve with function field $\K(\cX)^G$, where $\K(\cX)^G$ consists of all elements of $\K(\cX)$
fixed by every element in $G$. Usually, $\bar \cX$ is called the
quotient curve of $\cX$ by $G$ and denoted by \textcolor{black}{$\cX/G$}. The field extension $\K(\cX)|\K(\cX)^G$ is  Galois of degree $|G|$.

Since our approach is mostly group theoretical, we prefer to use notation and terminology from Group theory rather than from Function field theory.

Let $\Phi$ be the cover of \textcolor{black}{$\cX\to \bar{\cX}$} where $\bar{\cX}=\cX/G$ is a quotient curve of $\cX$ with respect to $G$.
 A point $P\in\cX$ is a ramification point of $G$ if the stabilizer $G_P$ of $P$ in $G$ is non-trivial; the ramification index $e_P$ is $|G_P|$; a point $\bar{Q}\in\bar{\cX}$ is a branch point of $G$ if there is a ramification point $P\in \cX$ such that $\Phi(P)=\bar{Q}$; the ramification (branch) locus of $G$ is the set of all ramification (branch) points. The $G$-orbit of $P\in \cX$ is the subset of $\cX$
$o=\{R\mid R=g(P),\, g\in G\}$, and it is {\em long} if $|o|=|G|$, otherwise \textcolor{black}{$o$} is {\em short}. For a point $\bar{Q}$, the $G$-orbit $o$ lying over $\bar{Q}$ consists of all points $P\in\cX$ such that $\Phi(P)=\bar{Q}$. If $P\in o$ then $|o|=|G|/|G_P|$ and hence $\bar{\cQ}$ is a branch point if and only if $o$ is a short $G$-orbit. It may be that $G$ has no short orbits. This is the case if and only if every non-trivial element in $G$ is fixed--point-free on $\cX$, that is, the cover $\Phi$ is unramified. On the other hand, $G$ has a finite number of short orbits. For a non-negative integer $i$, the $i$-th ramification group of $\cX$
at $P$ is denoted by $G_P^{(i)}$ (or $G_i(P)$ as in \cite[Chapter
IV]{serre1979})  and defined to be
$$G_P^{(i)}=\{g\mid \ord_P(g(t)-t)\geq i+1, g\in
G_P\}, $$ where $t$ is a uniformizing element (local parameter) at
$P$. Here $G_P^{(0)}=G_P$.
The structure of $G_P$ is well known; see for instance \cite[Chapter IV, Corollary 4]{serre1979} or \cite[Theorem 11.49]{hirschfeld-korchmaros-torres2008}.
\begin{result}
\label{res74} The stabilizer $G_P$ of a point $P\in \cX$ in $G$ \textcolor{black}{has} the following properties.
\begin{itemize}
\item[\rm(i)] $G_P^{(1)}$ is the unique \textcolor{black}{Sylow} $p$-subgroup of $G_P$;
\item[\rm(ii)] For $i\ge 1$, $G_P^{(i)}$ is a normal subgroup of $G_P$ and the quotient group $G_P^{(i)}/G_P^{(i+1)}$ is an elementary abelian $p$-group.
\item[\rm(iii)] $G_P=G_P^{(1)}\rtimes U$ where the complement $U$ is a cyclic \textcolor{black}{group} whose order is prime to $p$.
\end{itemize}
\end{result}
Let $\bar{\gg}$ be the genus of the quotient curve $\bar{\cX}=\cX/G$. The Hurwitz
genus formula  gives the following equation
    \begin{equation}
    \label{eq1}
2\gg-2=|G|(2\bar{\gg}-2)+\sum_{P\in \cX} d_P.
    \end{equation}
    where
\begin{equation}
\label{eq1bis}
d_P= \sum_{i\geq 0}(|G_P^{(i)}|-1).
\end{equation}
Here $D(\cX|\bar{\cX})=\sum_{P\in\cX}d_P$ is the {\emph{different}}. For a tame subgroup $G$ of $\aut(\cX)$, that is for $p\nmid |G_P|$,
$$\sum_{P\in \cX} d_P=\sum_{i=1}^m (|G|-\ell_i)$$
where $\ell_1,\ldots,\ell_m$ are the sizes of the short orbits of $G$.

A subgroup of $\aut(\cX)$ is a $p'$-group (or a prime to $p$ group) if its order is prime to $p$. A subgroup $G$ of $\aut(\cX)$ is {\em{tame}} if the $1$-point stabilizer of any point in $G$ is \textcolor{black}{a} $p'$-group. Otherwise, $G$ is {\em{non-tame}} (or {\em{wild}}). Obviously, every $p'$-subgroup of $\aut(\cX)$ is tame, but the converse is not always true. From the classical Hurwitz's bound,
if $|G|>84(\gg(\cX)-1)$ then $G$ is non-tame; see  \cite{stichtenoth1973II} or \cite[Theorems 11.56]{hirschfeld-korchmaros-torres2008}.
An orbit $o$ of $G$ is {\em{tame}} if $G_P$ is a $p'$-group for $P\in o$, otherwise $o$ is a {\em{non-tame orbit}} of $G$.

%Stichtenoth's result  \cite{stichtenoth1973II} on the number of short orbits of large automorphism groups; see \cite[Theorems 11.56, 11.116]{hirschfeld-korchmaros-torres2008}:
\begin{result}[Stichtenoth's result   on the number of short orbits of large automorphism groups \cite{stichtenoth1973II}; see also Theorems 11.56 and 11.116 in  \cite{hirschfeld-korchmaros-torres2008}]
\label{res56.116} Let $G$ be a subgroup of $\aut(\cX)$ whose order exceeds $84(\gg(\cX)-1)$. Then \textcolor{black}{$G$ has}
\begin{itemize}
\item[\rm(a)] exactly three short orbits$,$ two tame and
one non-tame$,$ and $|G|< 24\gg(\cX)^2$;
\item[\rm(b)] exactly two short orbits$,$ both non-tame$,$ and  $|G|<16 \gg(\cX)^2$;
\item[\rm(c)] only one short orbit which is non-tame$;$
\item[\rm(d)] exactly two short orbits$,$ one tame and one non-tame.
\end{itemize}
If $|G|>4(\gg(\cX)-1)$ then the quotient curve $\cX/G$ is rational.
\end{result}
%Roquette's result; see \cite{roquette1970}
\begin{result}[Roquette's result \cite{roquette1970}]
\label{ROQ}
If the size of $\aut(\cX)$ exceeds $84(\gg(\cX)-1)$, then $\gg(\cX) \ge p-1$.
\end{result}
The $p$-rank of $\cX$ is  defined to be the rank of the (elementary abelian) group of the $p$-torsion points in the Jacobian variety of $\cX$, and it coincides with the Hasse-Witt invariant of $\cX$. \textcolor{black}{The $p$-rank of $\cX$ does not exceed the genus $\gg(\cX)$.}
Let $\gamma$ be the $p$-rank of $\cX$,
 and let $\bar{\gamma}$ be the $p$-rank of the quotient curve \textcolor{black}{$\bar{\cX}=\cX/S$ for a $p$-subgroup $S$ of $\aut(\cX)$}. The Deuring-Shafarevich formula, see \cite{sullivan1975} or \cite[Theorem 11,62]{hirschfeld-korchmaros-torres2008}, states that
 \textcolor{black}{
\begin{equation}
    \label{eq2deuring}
\gamma-1={|S|}(\bar{\gamma}-1)+\sum_{i=1}^k (|S|-\ell_i)
    \end{equation}
where $\ell_1,\ldots,\ell_k$ are the sizes of the short orbits of $S$.
}
\begin{result}
\label{lem29dic2015} If $\cX$ has zero $p$-rank then $\aut(\cX)$ has the following properties:
\begin{itemize}
\item[\rm(i)] a Sylow $p$-subgroup of $\aut(\cX)$ fixes a point $P\in \cX$ but its non-trivial elements have no fixed point other than $P$;
\item[\rm(ii)] the normalizer of a Sylow $p$-subgroup fixes a point of $\cX$;
\item[\rm(iii)] any two distinct Sylow $p$-subgroups have trivial intersection.
\end{itemize}
\end{result}
Claim (i) is \cite[Theorem 11.129]{hirschfeld-korchmaros-torres2008}. Claim (ii) follows from Claim (i). Claim (iii) is \cite[Theorem 11.133]{hirschfeld-korchmaros-torres2008}.

%Nakajima's bound \cite{nakajima1987}; see also \cite{gkp>2} and \cite[Theorem 11.54]{hirschfeld-korchmaros-torres2008}:
\begin{result}[Nakajima's bound \cite{nakajima1987}; see also \cite{gkp>2} and Theorem \textcolor{black}{11.84} in \cite{hirschfeld-korchmaros-torres2008}]
\label{resnaga} If $\cX$ has positive $p$-rank and $S$ is a $p$-subgroup of $\aut(\cX)$ then
\begin{equation}
\label{naka16feb2013}
|S|\leq \left\{
\begin{array}{lll}
\textstyle\frac{p}{p-2}\,(\gg(\cX)-1)\quad {\mbox{for}}\quad \gamma(\cX)>1,\\
\quad\quad\,\, \gg(\cX)-1\quad\,\, {\mbox{for}}\quad \gamma(\cX)=1.
\end{array}
\right.
\end{equation}
If, in addition, $S$ fixes a point of $\cX$ then $|S|\leq \frac{p}{p-1}\,\gg(\cX).$
\end{result}

We quote a result on the orders of Sylow $p$-subgroups in large automorphism groups.
%; see \cite[Proposition 11.138]{hirschfeld-korchmaros-torres2008}, and also \cite[Theorem 1.1]{gktrans}.
\begin{result}[Proposition 11.138 in \cite{hirschfeld-korchmaros-torres2008}; see also Theorem 1.1 in \cite{gktrans}]
\label{gktrans}
Let $\cX$ be a curve with $p$-rank $0$, genus $\gg(\cX)\ge 2$, and such that $G$ is a solvable automorphism group of $\cX$ with $|G|>24\gg(\cX)(\gg(\cX)-1)$. Let $S$ be a Sylow $p$-subgroup of $G$.
If $G$ has no fixed point then $|S|<p^2$.
\end{result}

For automorphism groups with special structures or actions on $\cX$, the classical Hurwitz bound can be improved; see for instance \textcolor{black}{\cite[Theorems 11.108, 11.79, 11.60]{hirschfeld-korchmaros-torres2008}}.
\begin{result}
\label{res60.79.108}
\begin{enumerate}
\item[\rm(i)] If $G$ is abelian then $|G|\leq 4\gg(\cX)+4$.
\item[\rm(ii)] If $G$ has prime order other than $p$ then $|G|\leq 2\gg(\cX)+1$.
\item[\rm(iii)] If $G$ fixes a point $P$ of $\cX$ and its order is prime to $p$ then $|G|\leq 4\gg(\cX)+2$.
\end{enumerate}
\end{result}
%For automorphism groups of elliptic curves and genus $2$ curves; see \cite[Theorem 94, Theorem 11.98, Proposition 11.98]{hirschfeld-korchmaros-torres2008}, and \cite[Theorem 2]{sv}:
\begin{result}[Automorphism groups of elliptic curves and genus $2$ curves; see \textcolor{black}{Theorem 11.94, Theorem 11.98, and Proposition 11.99 in \cite{hirschfeld-korchmaros-torres2008}}, and Theorem 2 in \cite{sv}]
\label{res94} If $G$ is an automorphism group of an elliptic curve $\cE$ over $\K$ then for every point $P\in\cE$ the stabilizer $G_P$ has order at most $6$ when $p>3$ and at most $12$ when $p=3$. If $G$ is an automorphism group of an hyperelliptic curve $\cH$ over $\K$ with hyperelliptic involution $\omega$ and $\omega \in G$ then $G/\langle \omega \rangle$ is isomorphic to a finite subgroup of $PGL(2,\K)$. If $\cX$ has genus $2$ then  $G/\langle \omega \rangle$ is not isomorphic to either $PSL(2,q)$ with $q\geq 5$, or $\rm{Alt}_7$, or the Mathieu group $M_{11}$, or $PSU(3,q)$ with $q\geq 5$, or $PGU(3,q)$ with $q\ge 5$.
\end{result}
The results from Group theory which play a role in the paper are quoted below, where $G$ denotes any abstract \textcolor{black}{finite} group. We use standard notation and terminology; see \cite{gorenstein1980,huppertI1967,hall}. In particular, $O(G)$ stands for the largest normal subgroup of $G$ of odd order. Furthermore, $q$ denotes a power of an odd prime $d$.

%Dickson's classification of finite subgroups of the projective linear group $\PGL(2,\mathbb{K})$; see \cite{maddenevalentini1982} and also \cite[Theorem A.8]{hirschfeld-korchmaros-torres2008}:
\begin{result}[Dickson's classification of finite subgroups of the projective linear group $PGL(2,\mathbb{K})$; see \cite{maddenevalentini1982} and also Theorem A.8 in  \cite{hirschfeld-korchmaros-torres2008}]
\label{resdickson}
Any finite subgroup of the group $PGL(2,\mathbb{K})$ is isomorphic to one of the following groups:
\begin{enumerate}
\item[\rm(i)] prime to $p$ cyclic groups;
\item[\rm(ii)] elementary abelian $p$-groups;
\item[\rm(iii)] prime to $p$ dihedral groups;
\item[\rm(iv)] the alternating group ${\rm{Alt}}_4$;
\item[\rm(v)] the symmetric group ${\rm{Sym}}_4$;
\item[\rm(vi)] the alternating group ${\rm{Alt_5}}$;
\item[\rm(vii)] the semidirect product of an elementary abelian
$p$-group of order $p^h$ by a cyclic group of order $n>1$ with
 $n\mid(p^h-1);$
\item[\rm(viii)] $\PSL(2,p^f)$;
\item[\rm(ix)] $\PGL(2,p^f)$.
\end{enumerate}
For $q>11$, the minimal size of a set on which $PSL(2,q)$ can act transitively is $q+1$; see \cite[Kapitel II, Satz 8.28]{huppertI1967}.
\end{result}
The maximal subgroups of $PSU(3,q)$ for $q$ odd were classified by Mitchell \cite{mitchell}; see also \cite{onan} and \cite[Theorem A.10]{hirschfeld-korchmaros-torres2008}. In this paper
we only need the following corollaries of Mitchell's classification; see \cite[Theorem 30]{mitchell}.
%%modifica 11 aprile 2018
\begin{result}
\label{resmitchell} Let $\mu={\rm{gcd}}(3,q+1)$. The subgroups of $PSU(3,q)$ \textcolor{black}{of} order $q^3(q^2-1)/\mu$ are maximal. For $q>5$, they are the largest subgroups of $PSU(3,q)$ for $q>5$. For $q=5$
the largest subgroups of $\PSU(3,q)$ have order $2520$ and are isomorphic to ${\rm{Alt}}_7$. Let $\Omega$ be a set of smallest size on which $PSU(3,q)$ has a non-trivial action. Then $|\Omega|\geq q^3+1$ with an exception for $q=5$ and $|\Omega|=60$.
\end{result}
%fine
%Classification of $2$-groups with a cyclic subgroup of index $2$; see \cite[Theorem 12.5.1]{hall}, or \cite[Kapitel I, Satz 14.9]{huppertI1967}:
\begin{result}[Classification of $2$-groups with a cyclic subgroup of index $2$; see Theorem 12.5.1 in \cite{hall}, or Kapitel I, Satz 14.9 in \cite{huppertI1967}]
\label{res26feb2018}
The non-cyclic groups of order $2^n$ which contain a cyclic subgroup of index $2$ are generated by two elements $a,b$ and are of the following types:
\begin{itemize}
\item $a^{2^{n-1}}=1$, $b^2=1$, $ba=ab$ (direct product of two cyclic groups);
\item $n\geq 3$, $a^{2^{n-1}}=1$, $b^2=a^{2^{n-2}}$, $ba=a^{-1}b$ (generalized quaternion group);
\item $n\geq 2$, $a^{2^{n-1}}=1$, $b^2=1$, $ba=a^{-1}b$ (dihedral group);
\item $n\geq 4$, $a^{2^{n-1}}=1$, $b^2=1$, $ba=a^{-1+2^{n-2}}b$ (semidihedral - also called quasidihedral - group);
\item $n\geq 4$, $a^{2^{n-1}}=1$, $b^2=1$, $ba=a^{1+2^{n-2}}b$ (modular maximal-cyclic group, that is type (3) group with notation as in \cite{huppertI1967}).
\end{itemize}
\end{result}
Observe that the dihedral group for $n=2$ is an elementary abelian group of order $4$.
%, that is, the direct product of two groups of order $2$.
Some elementary properties of the above $2$-groups are as follows; see \cite[Kapitel I, Satz 14.9]{huppertI1967}. A generalized quaternion group has rank $1$, the other groups listed in  Result \ref{res26feb2018} have rank $2$ as they contain an elementary abelian group of order $4$ but no elementary abelian group of order $8$.
Generalized quaternion groups, as well as dihedral and semidihedral groups, have center of order $2$. Dihedral, semidihedral and modular maximal-cyclic groups have non-central involutions.
The quotient groups of a generalized quaternion group, as well as of a semidihedral group, with respect to its center is a dihedral group.
A modular maximal-cyclic group has exactly three involutions and four elements of order $4$. In particular, it contains neither a dihedral group of order $2^n\geq 8$ nor a generalized quaternion group.

%The Schur-Zassenhaus theorem, see \cite[Kapitel I, Hauptsatz 18.1]{huppertI1967}:
\begin{result}[The Schur-Zassenhaus theorem; see Kapitel I, Hauptsatz 18.1 in \cite{huppertI1967}]
\label{reszass} Let $N$ be a normal subgroup of $G$ such that $|N|$ and $[G:N]$ are coprime. Then $N$ has a complement in $G$, that is, $G=N\rtimes U$ for a subgroup $U$ of $G$.
\end{result}

%The Gorenstein-Walter theorem \cite{gw1,gw2}; see also \cite[Chapter 16]{gorenstein1980}:
\begin{result}[The Gorenstein-Walter theorem \cite{gw1,gw2}; see also Chapter 16 in \cite{gorenstein1980}]
\label{resgor}  Let $G$ be a group with a dihedral Sylow $2$-subgroup $S_2$. Then, either $G=O(G)\rtimes S_2$, or $G/O(G)\cong\rm{Alt}_7$, or $G/O(G)$ is isomorphic to a subgroup of $P\Gamma L(2,q)$ containing  $PSL(2,q)$.
\end{result}
%The Brauer-Suzuki theorem \cite{bs}; see also \cite[Chapter 12]{gorenstein1980}:
\begin{result}[The Brauer-Suzuki theorem \cite{bs}; see also Chapter 12 in \cite{gorenstein1980}]
\label{resbs} Let $G$ be a group with a generalized quaternion Sylow $2$-subgroup. Then the center of $G/O(G)$ has order $2$.
\end{result}
Observe that a Sylow $2$-subgroup of $SL(2,q)$ is a generalized quaternion group.
%The Alperin-Brauer-Gorenstein theorem \cite{abg}:
\begin{result}[The Alperin-Brauer-Gorenstein theorem \cite{abg}]
\label{resabg} Let $G$ be a non-abelian simple group with a semidihedral Sylow $2$-subgroup $S_2$. Then, either
$G\cong \PSL(3,q)$ with $q\equiv 3 \pmod 4$, or $G\cong \PSU(3,q)$ with $q \equiv 1 \pmod 4$, or $G=M_{11}$, where $q$ is an odd prime power.
\end{result}
%The Alperin theorem \cite{alp}:
\begin{result}[The Alperin theorem \cite{alp}]
\label{resalp}  A Sylow $2$-subgroup of a non-abelian simple group with $2$-rank equal to $2$ is either dihedral, or semidihedral, or wreathed, or isomorphic to a Sylow $2$-subgroup of $SU(3,2)$.
\end{result}

A group $G$ of even order is said to have a normal $2$-complement if $G=O(G)\rtimes S_2$ with a Sylow $2$-subgroup $S_2$ of $G$. By an elementary result, see for instance \cite[Kapitel I, Aufgaben I.21]{huppertI1967}, the groups of even order with a cyclic Sylow $2$-subgroup have a normal $2$-complement. More generally, groups with a Sylow $2$-subgroup containing a cyclic subgroup of index $2$
have a normal  $2$-complement, apart from some exceptions; see \cite{wong1,wong}.
\begin{result}
\label{reswong1}  If a group $G$ has no $2$-complement but a Sylow $2$-subgroup $S_2$ of $G$ has a cyclic subgroup of index $\leq 2$ then $S_2$ is either a generalized quaternion, or a dihedral group, or a semidihedral group. If, in addition, $G$ is also solvable then one of the following cases occurs.
\begin{itemize}
\item $S_2$ is dihedral and $G/O(G)\cong \PSL(2,3), \PGL(2,3)$.
\item $S_2$ is semidihedral and $G/O(G)\cong GL(2,3)$.
\end{itemize}
\end{result}
\begin{result}
\label{reswong}  If a group $G$ has no $2$-complement but it has a semidihedral Sylow $2$-subgroup, then one of the following holds.
\item[\rm(i)] $G$ has a (normal) subgroup of index $2$, which has no normal subgroup of index $2$ and has a dihedral Sylow $2$-subgroup.
\item[\rm(ii)] $G$ has a (normal) subgroup of index $2$, which has no normal subgroup of index $2$ and has a generalized quaternion Sylow $2$-subgroup.
\item[\rm(iii)] $G$ has no normal subgroup of index $2$, and the involutions of $G$ form a single conjugacy class.
\end{result}

Let $G$ be a group with non-trivial center $Z(G)$. If $N$ is a subgroup \textcolor{black}{of} $Z(G)$ with quotient group $U=G/N$, then $G$ is called a central extension of $U$ by $N$. When the sequence splits, that is $G$ has a subgroup isomorphic to $U$ that meets $N$ trivially, $G$ is called a split extension of $U$ by $N$, and $G=N\rtimes U$. If $N\le G'$ then the
central extension is {\em{stem}}, and if $G$ is {\em{perfect}}, that is $G=G'$, the central extension is {\em{perfect}}.

%Schur's result on perfect central extensions of $PSL(2,q)$; see \cite[Kapitel V, Satz 25.7]{huppertI1967}:
\begin{result}[Schur's result on perfect central extensions of $PSL(2,q)$; see Kapitel V, Satz 25.7 in \cite{huppertI1967}]
\label{resschur} Let $G$ be a perfect central extension of $PSL(2,q)$ with $q\geq 5$. If $q\neq 9$ then $G=SL(2,q)$. This holds true for $q=9$ provided that extension is by a group of order $2$.
\end{result}
%Griess' result on perfect central extensions of $PSU(3,q)$; see \cite[Theorem 2]{griess}:
\begin{result}[Griess' result on perfect central extensions of $PSU(3,q)$; see Theorem 2 in \cite{griess}]
\label{resgriess} The unique perfect central extension of $PSU(3,q)$ is $SU(3,q)$.
\end{result}
A linear \textcolor{black}{group}  containing $SL(2,q)$ as an index $2$ subgroups is $SL^{\pm}(2,q)$, that is, the subgroup of $GL(2,q)$ consisting of all matrices with determinant equal to $\pm 1$.
$SL^{\pm}(2,q)$ contains an involution other than the central involution of $SL(2,q)$, and hence its Sylow $2$-subgroups are not generalized quaternion groups. More precisely, for $q\equiv -1 \pmod 4$ they are semidihedral groups. Similarly, the subgroup $SU^{\pm}(2,q)$ of $GU(2,q)$ consisting of all matrices with determinant equal to $\pm 1$ contains $SU(2,q)$ as an index $2$ subgroup. Furthermore, $SL(2,q)\cong SU(2,q)$. The center of $SU^{\pm}(2,q)$ has order either $2$ or $4$ according as $q\equiv 1 \pmod 4$ or $q\equiv -1 \pmod 4$, and in the former case the Sylow $2$-subgroups of $SU^{\pm}(2,q)$ are semidihedral.

 %Schur's classification of non-split central extension of $PGL(2,q)$ by an involution; see  \cite[pg. 176]{goha}:
\begin{result}[Schur's classification of non-split central extension of $PGL(2,q)$ by an involution; see p. 176 in  \cite{goha}]
\label{resgoha}
Let $G$ be a group with a Sylow $2$-subgroup not isomorphic to a generalized quaternion group. If $G$ is a stem non-split central extension of $PGL(2,q)$ by a group of order $2$ then a Sylow $2$-subgroup of $G$ is semidihedral, and either $G\cong SL^{\pm}(2,q)$ or $G\cong SU^{\pm}(2,q)$ according as $q\equiv -1 \pmod 4$ or $q\equiv 1 \pmod 4$. For each odd $q$, $PGL(2,q)$ has two non-isomorphic stem non-split central extension by a group of order $2$. One has semidihedral Sylow $2$-subgroups, the other has generalized quaternion and is denoted by $SL^{(2)}(2,q)$.
\end{result}

%Thompson's transfer lemma \cite[Lemma 5.38]{thompson}:
\begin{result}[Thompson's transfer lemma; see Lemma 5.38 in \cite{thompson}]
\label{resthompson} Let $G$ be a finite group of even order. Let $M$ be an index $2$ subgroup of a Sylow $2$-subgroup $S_2$ of $G$. If $G$ has no subgroup of index $2$, then every involution in $S_2\setminus M$ is conjugate to an involution in $M$.
\end{result}
%We quote a result on cross-representations for $PSU(3,q)$.
\begin{result}[The Land\'azuri and Seitz bound on the degree of cross-representations of $PSU(3,q)$; see \cite{lse}]
\label{reslse} Assume that $PSU(3,q)$ has an irreducible, faithful representation as a projective group of the projective space $PG(r-1,t)$. If  $t>2$ is an odd  prime and  $q$ is not a $t$-power, then $r-1 \ge q(q-1)$.
\end{result}
A group space $(\Omega,G)$  is a pair where $\Omega$ is a non-empty set   and $G$ is a group acting on $\Omega$. This action is not supposed to \textcolor{black}{be}  faithful, that is, the permutation representation of $G$ on $\Omega$ may have non-trivial kernel. A group space $(\Omega,G)$ is called transitive if $G$ induces on $\Omega$ a transitive permutation group.

%Hering's classification on doubly transitive group spaces; see \cite[Theorem 2.4]{hering1979}:
\begin{result}[Hering's classification on doubly transitive group spaces; see Theorem 2.4 in \cite{hering1979}]
\label{reshering} Let $(\Omega, G)$ be a finite transitive group space, where $|\Omega|>2$. Assume that for some $\alpha\in\Omega$ the stabiliser $G_{\alpha}$ contains a normal
subgroup $Q$ which is sharply transitive on $\Omega\setminus\{\alpha\}$. If $S$ is the normal closure of $Q$ in $G$, then one of the following holds:
\begin{itemize}
\item[(i)] $S\cong PSL(2,q) , SL(2, q), Sz(q), PSU(3, q), SU(3, q)$ or a group of
Ree type of order $q^{3}(q^{3}+1)(q-1)$, where $q$ is a prime power and the  \textcolor{black}{size of $\Omega$ is}
$q+1$ in the linear case, $q^{2}+1$ in the Suzuki case and $q^{3}+1$ in the unitary
and Ree case.
\item[(ii)] $S\cong P\Gamma L(2,8)$ and $|\Omega|=28$.
\item[(iii)] $(\Omega, S)$ is sharply 2-transitive.
\item[(iv)] $|\Omega|=p^{2}$ for a prime $p\in\{3,5,7,11,23,29, 59\}$,
%$S=O_{p}(S)\cdot Q$, $O_{p}(S)$ is extraspccial of order $p^{3}$ and exponent $p$, $S O_{p}(S)=$ is the kernel of $(\Omega, S)$ and$S$ induces a sharply 2-transitive group on $\Omega$.
\end{itemize}
\end{result}
%The Gow-Burnside theorem \cite{gow1975}; see \cite[Theorem 11.134]{hirschfeld-korchmaros-torres2008}:
\begin{result}[The Gow-Burnside theorem \cite{gow1975}; see Theorem 11.134 in \cite{hirschfeld-korchmaros-torres2008}]
\label{resgow} If a Sylow $d$-subgroup $S_d$  of \textcolor{black}{a} solvable group $G$ meets any other Sylow $d$-subgroup of $G$ trivially then $S_d$ is either cyclic,  or a normal subgroup of $G$, or $d=2$ and $S_d$ is generalized quaternion group.
\end{result}
%The Feit-Thompson theorem \cite{ft}:
\begin{result}[The Feit-Thompson theorem \cite{ft}]
\label{resft} Every group of odd order is solvable.
\end{result}

%%%%%%%%%%%%%%%%

\section{Large solvable automorphism groups of curves}\label{principA}
In this section, our goal is to provide a proof for  Theorem \ref{princA}. We begin with the first claim.
\begin{proposition}\label{princE} Let $\cX$ be an algebraic curve of genus $\gg \geq 2$ with a solvable subgroup $G$  of $\aut(\cX)$ such that
\begin{equation}
\label{eq13jan2015} |G|>84\frac{p}{p-2}\gg^2.
\end{equation}
Then the $p$-rank of $\cX$ is equal to zero.
\end{proposition}
The proof of Proposition \ref{princE} is by induction on the genus of $\cX$, and is carried out by a series of lemmas and propositions.
By (i) of Result \ref{res60.79.108}, the group $G$ is assumed to be non-abelian. Then, one the following cases occurs\textcolor{black}{:}
\begin{itemize}
\item[(A)] {\em{$G$ has some non-trivial normal $p'$-subgroup;}}
\item[(B)] {\em{ $G$ has some non-trivial normal $p$-subgroup, but it does not have any non-trivial normal $p'$-subgroup.}}
\end{itemize}
We investigate Case (A).
\begin{lemma}\label{lemmaA00}
Under the hypotheses of Proposition \ref{princE}, if (A) holds then one of the following cases occurs.
\begin{itemize}
\item[\rm(C)] For any normal $p'$-subgroup  $N$ of $G$ of maximal order the quotient curve $\bar{\cX}=\cX/N$ has zero $p$-rank.
\item[\rm(D)] Some curve of genus less than $\gg$ is a counterexample to Proposition \ref{princE} for an automorphism group  \textcolor{black}{satisfying Property}  (B).
\end{itemize}
\end{lemma}
\begin{proof}
Claim (C) certainly holds if $\bar{\cX}$ is rational.

If $\bar{\cX}$ is elliptic, take a Sylow $p$-subgroup $S$ of $G$. Since (\ref{eq13jan2015}) yields $|G|>84(\gg -1)$, Result \ref{res56.116} shows that $G$ has a non-tame orbit on $\cX$. Hence, some non-trivial subgroup of $S$ fixes a point of $\cX$. Let $\bar{S}=SN/N$. Then $\bar{S}$ is a Sylow $p$-subgroup of $\bar{G}=G/N$, and some non-trivial subgroup $\bar{T}$ of $\bar{S}$ fixes a point of $\bar{\cX}$. From Result \ref{res94}, $|\bar{T}|=3$ (and $p=3$). The Deuring-Shafarevich formula applied to $\bar{T}$ gives $\gamma(\bar{\cX})-1=3(\gamma(\bar{\cY})-1)+2\lambda$ where $\bar{\cY}$ is the quotient curve of $\bar{\cX}$ by $\bar{T}$ while $\lambda>0$ is the number of fixed \textcolor{black}{points}  of $\bar{T}$ on $\bar{\cX}$. Since $0\leq \gamma(\bar{\cY})\le\gamma(\bar{\cX})\leq \gg(\bar{\cX})=1$, this equation implies $\gamma(\bar{\cX})=0$. Then Claim (C) also holds for $\gg(\bar{\cX})=1$. If $\gg=2$, then $\bar{\cX}$ is either rational or elliptic and both cases have already been considered.

Therefore, $\gg\ge 3$ and $\gg(\bar{\cX})\ge 2$ are assumed. From the Hurwitz genus formula applied to $N$,
$$
\gg\ge |N|({\gg}(\bar{\cX})-1)+1.
$$
This together with $(|N|(\gg(\bar{\cX})-1)+1)^2>|N|\gg(\bar{\cX})^2$ yield
$$
|\bar G|=\frac{|G|}{|N|}>\textstyle\frac{p}{p-2}84\frac{\gg^2}{|N|}>\textstyle\frac{p}{p-2}84 \gg(\bar{\cX})^2.
 $$
The group $\bar{G}$ is a non-abelian solvable group and hence it has a non-trivial minimal normal subgroup $\bar{M}$. Let $M$ be the normal subgroup of $G$ containing $N$ such that $M/N=\bar{M}$. Then $p$ divides $|M|$ as $N$ has been chosen to be a maximal normal $p'$-subgroup of $G$. Since $p\nmid |N|$, this yields that $p$ divides $|\bar{M}|$, as well. Since $\bar{M}$ is a minimal normal subgroup, and $G$ is solvable,
$\bar{M}$ is a normal (elementary abelian) $p$-subgroup of $\bar{G}$. If the $p$-rank of $\bar{\cX}$ is positive, $\bar{\cX}$ turns out to be a counterexample to Proposition \ref{princE} for its automorphism \textcolor{black}{group}  $\bar{G}$ satisfying property (B).
%
%By the inductive hypothesis $\gamma(\bar\cX)=0$. NON MI TORNA IN REALT\`A PERCH\'E QUESTA CURVA QUOZIENTE NON HA SOTTOGRUPPI DI QUEL TIPO! BISOGNER\`A FAR RIFERIMENTO AL CASO IN CUI NON CI SONO $p'$-GRUPPI. MA POI DIVENTA UN GATTO CHE SI MORDE LA CODA, VISTO CHE ANCHE NELL'ALTRO CASO ABBIAMO BISOGNO DI QUESTO.
\end{proof}

\begin{proposition}
\label{lemmaA}
Under the hypotheses of  Proposition \ref{princE}, if (A) holds then either $\cX$ has zero $p$-rank,  or {\rm(D)} holds.
\end{proposition}
\begin{proof}
Among the normal $p'$-subgroups of $G$ choose one, say $N$, of maximal order. Let ${\bar M}$ be a normal (elementary abelian) $p$-subgroup of $\bar{G}$ defined as in the final part of the proof of Lemma \ref{lemmaA00}.

By Lemma \ref{lemmaA00} we may assume that (C) holds. Then $\gamma(\bar{\cX})=0$, and hence $\bar{M}$ fixes a unique point $\bar{P}\in \bar{\cX}$, by (i) of Result \ref{lem29dic2015}.

Let $\theta$ be the $N$-orbit consisting of all points of $\cX$ lying over $\bar{P}$ in the cover $\cX|\bar{\cX}$. Since $p\nmid |N|$, we also have $p\nmid |\theta|$. Furthermore, as $\bar{M}$ is a normal subgroup of $\bar{G}$,  $\bar{P}$ is fixed by $\bar{G}$, as well. Therefore, $\theta$ is a $G$-orbit. Choose a point $P\in \theta$. Then $|G|=|G_P||\theta|$ and (iii) of Result \ref{res74} shows that the stabilizer $G_P$ of $P$ in $G$ is the semidirect product of a Sylow $p$-subgroup $S$ of $G_P$ by a cyclic group $H$ of order $h$ with $p\nmid h$. Actually, $S$ is a Sylow $p$-subgroup of $G$ since $p$ does not divide $|\theta|$.

If $\theta=\{P\}$ then $G=G_P=S\rtimes H$. From this and (\ref{eq13jan2015}), $24\gg^2<|G|=|S||H|\leq |S|(4\gg+2)$ by (iii) of Result \ref{res60.79.108}  whence  $$|S|>\frac{24\gg^2}{4\gg+2}>6\frac{\gg^2}{\gg+1}>4\gg.$$
From Result \ref{resnaga}, $\cX$ has zero $p$-rank.

 Therefore, $|\theta|\ge 2$ is assumed. Since $N$ is a normal subgroup of $G$, the set $L=NH$ is a subgroup of $G$ whose order $|L|$ is equal to $|H||N|/|H\cap N|$. Obviously, $\theta$ is an $L$-orbit, and hence $|L|=|L_P||\theta|$ with $L_P$ being the stabilizer of $P$ in $L$ and $H\leq L_P$. Since $|G|=|G_P||\theta|=|S||H||\theta|$, the Lagrange theorem yields that
$$\frac{|G|}{|L|}=|S|\frac{|H|}{|L_p|}$$
is an integer. Since $|S|$ and $|L_P|$ are coprime, this yields that $|L_P|$ must divide $|H|$. Hence, $L_P$ cannot contain $H$ properly. Therefore, $H=L_P$, and $|G|=|S||L_P||\theta|=|S||L|$.
As $p$ does not divide $|L|$, Result \ref{res56.116} yields $|L|\le 84(\gg-1)$. It turns out that
\begin{equation}
\label{eqq2}
84\textstyle\frac{p}{p-2}\gg^2<|G|=|S||L|\le|S|84(\gg-1)<84|S|\gg,
\end{equation}
whence $|S|>\textstyle\frac{p}{p-2}\gg$.
From Result \ref{resnaga}, $\cX$ has zero $p$-rank.
\begin{comment}
If $L$ contains a (cyclic) subgroup $T$ of index $d$ that fixes a point, then $|L|=d|T|\leq d(4\gg+2)$ by (iii) of Result \ref{res60.79.108}.  For $d\le 3$, this gives
$24\gg^2<|G|\leq 12 |S| (\gg+\ha)$ whence
$$|S|>2\frac{\gg^2}{\gg+\ha}>\frac{p}{p-1}\gg.$$
From Result \ref{resnaga}, $\cX$ has zero $p$-rank. Therefore, we may assume that every subgroup of $L$ fixing some point has index $d\geq 4$.

Now, apply the Hurwitz genus formula to the quotient curve $\cY=\cX/L$ of $\cX$ by $L$. If its genus $\gg(\cY)$ is at least $2$, then $\gg-1\geq |L|$. If $\gg(\cY)\le 2$ and $|L|>\gg -1$, then $L$ must have at least one short orbit $\Lambda$.
Let $Q\in \Lambda$. If $\cY$ is elliptic,
 then $2\gg-2\geq |\Lambda|(|L_Q|-1)\geq \ha |L|$ whence
$4(\gg-1)\ge |L|$. If $\cY$ is rational, then $\Lambda$ cannot be the unique short orbit of $L$ on $\cX$, as $2\gg-2=-2|L|+|\Lambda|(|L_Q|-1)=-|L|-\Lambda$ is impossible by $\gg(\cX)\geq 2$. Actually, there must be at least two more short orbits, and if $n_1$ and $n_2$ are the sizes of two of them, then
$$2\gg-2\geq -2|L|+|L|-|\Lambda|+2|L|-(n_1+n_2)=|L|-(|\Lambda|+n_1+n_2).$$
Since $d\geq 4$,  each of the sizes $|\Lambda|,n_1,n_2$ may be assumed to be at most $\qa |L|$. Therefore, $2\gg-2\geq \qa |L|$ whence $|L|\leq 8\gg-8$.

This shows that $|L|\leq 8\gg-8$ always holds. Now,  $24\gg^2<|G|= |S| |L|\le 8|S|(\gg-1)$ whence
$$|S|>3\frac{\gg^2}{\gg-1}>\frac{p}{p-1}\gg.$$
From Result \ref{resnaga}, $\cX$ has zero $p$-rank.
\end{comment}
\end{proof}
\begin{lemma}
\label{leB11luglio2015} Under the hypotheses of  Proposition \ref{princE}, if $\rm{(A)}$ holds and $\cX$ has zero $p$-rank then $G$ fixes a point.
\end{lemma}
\begin{proof}
Among the normal $p'$-subgroups of $G$ choose one, say $N$, of maximal order and let $\bar{\cX}=\cX/N$.
Let $M$ and ${\bar M}$ be as in the final part of the proof of Lemma \ref{lemmaA00}. We first show that
$\bar{M}$ fixes a unique point $\bar{P}\in \bar{\cX}$.
Since  $|M|=|\bar{M}||N|$, from Result \ref{reszass},
there exists a $p$-subgroup $T$ in $G$ such that $M=N\rtimes T$. Clearly, $\bar T=T/N$ coincides with $\bar M$. By (i) of Result \ref{lem29dic2015}, $T$ fixes a point $P\in \cX$ and no non-trivial element of $T$ fixes a point other than $P$.
Therefore, $\bar{T}$ fixes the point $\bar{P}$ lying under $P$ in the cover $\cX|\bar{\cX}$. Assume that $\bar{T}$ fixes a point $\bar{Q}\neq \bar{P}$.
The $N$-orbit $\omega$ consisting of all points lying over $\bar{Q}$ is left invariant by $M$. Since $|\omega|$ divides $|N|$ and $p$ is prime to $|N|$, $p$ is prime to $|\omega|$ as well. Then $T$ must fix a point $Q\in \omega$. Since $P\not\in \omega$, this contradicts (i) of Result \ref{lem29dic2015}. Therefore, $\bar{P}$ is the unique fixed point of $\bar{M}$. Since $\bar{M}$ is a normal subgroup of $\bar{G}$, $\bar{P}$ is the unique fixed point of $\bar{G}$, as well. This shows that $\omega$ is a non-tame short orbit of $G$. If $\omega=\{P\}$, then $G$ fixes $P$. Otherwise, the final part in the proof of  Proposition \ref{lemmaA} shows that $|G|\leq |S|84(\gg-1)$ where $S$ is
the Sylow $p$-subgroup of $G$ fixing $P$. This together
with our hypothesis (\ref{eq13jan2015}) give
$
|S|>\textstyle\frac{p}{p-1}\gg.
$
On the other hand, $\gg\ge p-1$ by Result \ref{ROQ}.
Therefore $|S|>p$ and hence $|S|\ge p^2$. By Result \ref{gktrans} $G$ fixes a point.
\end{proof}
So far, we have proven that if a counterexample for Proposition \ref{princE} exists then Case (D) must hold. Therefore, we are left with Case (B).
\begin{proposition}
\label{le110luglio2005} Under the hypotheses of Proposition \ref{princE},
if $\rm{(B)}$  holds then either $\cX$ has zero $p$-rank, or  $\rm{(D)}$  occurs.
\end{proposition}
\begin{proof}
 Choose a maximal normal $p$-subgroup $N$ of $G$. Then $|N|=p^h$ with $h\geq t$, and the factor group $\bar{G}=G/N$ contains no normal $p$-subgroup. If $\bar{G}$ is trivial then $G=N$, that is, $G$ is a $p$-group, and $\gamma(\cX)=0$ by Result \ref{resnaga}. Therefore $|\bar{G}|>1$ is assumed.  Let $S$ be a Sylow $p$-subgroup of $G$. Then $S\geq N$. From Result \ref{res56.116}, $G$ has a unique non-tame short orbit, say $\theta$.
Since $N$ is a normal subgroup, $\theta$ is partitioned into $N$-orbits, each of the same length $p^{h-k}$ where $p^k=|N_P|$ for \textcolor{black}{any point}  $P\in \theta$. If $\lambda$ counts the components of the partition, then $N$ has exactly $\lambda$ short orbits and $|\theta|=\lambda p^{h-k}$.

Now we prove three preliminary claims that will be useful when dealing separately, as usual, with the cases $\gg(\bar{\cX})=0,1,\ge 2$ \textcolor{black}{for}  the quotient curve $\bar{\cX}=\cX/N$.
\begin{claim}
\label{claim1} If $S=N$, or $|\theta|=1$ then $\cX$ has zero $p$-rank.
\end{claim}
\begin{proof}
If $S=N$ then by Result \ref{reszass}, $G=S\rtimes L$ with a $p'$-subgroup $L$ of $G$. Therefore, $|G|=|S||L|\leq |S| 84(\gg-1)$ whence $|S|>\frac{p}{p-2}\cdot \frac{\gg(\cX)^2}{\gg(\cX)-1}$. Hence $\cX$ has zero $p$-rank by Result \ref{resnaga}. Therefore, $S\supsetneq N$ is assumed. Then
 $\bar{S}=S/N$ is a non-trivial $p$-subgroup of $\aut(\bar{\cX})$. If $|\theta|=1$ then $G$ would fix a point and Result \ref{res74} would imply $S=N$.
 \end{proof}
Therefore, $S\supsetneq N$ and $|\theta|>1$ are assumed.
\begin{claim}
\label{claim2}
If $\lambda=1$ then $\gamma(\bar{\cX})\geq 2$.
\end{claim}
\begin{proof}
For $\lambda=1$, the Deuring-Shafarevich formula applied to $N$ gives $\gamma(\cX)-1=p^h(\gamma(\bar{\cX})-1)+(p^h-p^{h-k})$. From this $\gamma(\bar{\cX})-1\geq 0$, otherwise $\gamma(\cX)=0$. Furthermore, if $\gamma(\bar{\cX})=1$ then $\gamma(\cX)-1=p^h-p^{h-k}$. On the other hand, if $\cX$ has positive $p$-rank then Result \ref{resnaga} together with $|S|>|N|$ yield
$\gamma(\cX)-1\geq \textstyle\frac{p-2}{p}|S|\geq \textstyle\frac{p-2}{p}p^{h+1}$. Therefore,
$$p^h-p^{h-k}\geq (p-2)p^{h},$$
a contradiction with $p>2$.
\end{proof}
\begin{claim}
\label{claim3}
If $\lambda >1$, then  some non-trivial subgroup of $\bar{S}=S/N$ has a fixed point on $\bar{\cX}$.
\end{claim}
\begin{proof}
We show that if $\lambda>1$ then $G_P^{(1)}\gvertneqq N_P$ for any point $P \in\theta$. For $|N_P|=1$, this is trivial. Thus $k>0$ is assumed. From the Deuring-Shafarevich formula applied to $N$, $\gamma(\cX)-1=p^h(\gamma(\bar{\cX})-1)+\lambda(p^h-p^{h-k})\geq -p^h+\lambda(p^h-p^{h-k})$ whence
\begin{equation}
\label{eq13luglio2015}
\gamma(\cX)-1\geq \lambda p^h\left(1-\frac{1}{\lambda}-\frac{1}{p^{k}}\right)\geq \lambda p^h\left(1-\ha -\thi \right)\geq \textstyle\frac{1}{6}\lambda p^h=\textstyle\frac{1}{6}p^{k}|\theta|.
\end{equation}
According to (iii) of Result \ref{res60.79.108}, for a point $P\in \theta$ write $G_P=G_P^{(1)}\rtimes H$ where $H$ is a $p'$-subgroup of $G$ fixing $P$. Then $|H|\leq 4\gg+2$ by (iii) of Result \ref{res60.79.108}. Now if $G_P^{(1)}=N_P$, this together with (\ref{eq13luglio2015}) would yield $6(\gg(\cX)-1)\geq 6(\gamma(\cX)-1)\geq |G_P^{(1)}||\theta|$ whence $$24 \gg(\cX)^2>6(4\gg(\cX)+2)(\gg(\cX)-1)\geq |G_P^{(1)}||H||\theta|=|G|,$$
a contradiction.
\end{proof}

Since Case (B) occurs and $G$ is not $p$-group, $\bar{G}$  has a non-trivial minimal normal subgroup $\bar{T}$ which is an elementary abelian group of order $d^j$ with a prime $d\neq p$. Let $T$ be the subgroup of $G$ containing $N$ such that $\bar{T}=T/N$.
Three cases are investigated separately according to the values of $\gg(\bar{\cX})$.

If $\bar{\cX}$ is rational, then $\bar{G}$ is isomorphic to a subgroup of $PGL(2,\mathbb{K})$ of order divisible by $p$ which has a normal $p'$-subgroup isomorphic to $\bar{T}$. From Result \ref{resdickson}, $p=3$  and $\bar{G}\cong {\rm{Alt}}_4$ or $\bar{G}\cong {\rm{Sym}}_4$. Therefore, $24\gg(\cX)^2<|G|\leq 3^h|\bar{G}|\leq 8\cdot 3^{h+1}$ whence $|S|= 3^{h+1}\geq 3 \gg(\cX)^2>3\gg(\cX)$. From Result \ref{resnaga}, $\cX$ has
zero $p$-rank.

 If $\bar{\cX}$ is elliptic then $\lambda>1$ by Claim \ref{claim2}. Hence $\bar{S}$ is a subgroup of $\bar{\cX}$ with a non-trivial subgroup $\bar{R}$ fixing a point of $\bar{\cX}$. From Result \ref{res94}, $|\bar{R}|=3$ (and $p=3$).  The Deuring-Shafarevich formula applied to $\bar{R}$ gives $\gamma(\bar{\cX})-1=3(\gamma(\bar{\cY})-1)+2\tau$ where $\bar{\cY}$ is the quotient curve of $\bar{\cX}$ by $\bar{R}$ while $\tau>0$ is the number of fixed \textcolor{black}{points}  of $\bar{R}$ on $\bar{\cX}$. This equation is only consistent with $\gamma(\bar{\cX})=0$ and $\tau=1$. As a consequence, $\bar{S}$ fixes a unique point $\bar{P}$. Also, $\bar{S}=\bar{R}$ and hence $|S|=3|N|$. Furthermore, the $N$-orbit $\omega$ consisting of all points of $\cX$ lying over $\bar{P}$ is an $S$-orbit. Therefore, $|S_P|/|N_P|=|S|/|N|=3$ for $P\in \omega$. Write $G_P=S_P\rtimes H$ with a $p'$-subgroup fixing $P$.
If $k=0$ then no non-trivial element of  $N$  fixes a point,  and $\gamma(\bar{\cX})=0$ gives a contradiction.
Therefore $k>0$ and (\ref{eq13luglio2015}) holds as in the proof of Claim \ref{claim3}.
 From (\ref{eq13luglio2015}),
$18(\gamma(\cX)-1)\geq 3\cdot 3^k |\theta|=|S_P||\theta|$ which together with (iii) of Result \ref{res60.79.108} give
$$
|G|=|S_P| |\theta | |H|\le 18(\gamma(\cX)-1)(4\gg(\cX)+2)\le
18(4\gg(\cX)+2)(\gg(\cX)-1).$$
But this contradicts our hypothesis \eqref{eq13jan2015}.

If $\gg(\bar{\cX})\geq 2$, we may argue as in the proof of Lemma \ref{lemmaA00} to prove  $$|\bar{G}|>84\textstyle\frac{p}{p-2}\gg(\bar{\cX})^2.$$
Since $\bar{T}$ is a normal $p'$-subgroup of $\bar{G}$, Proposition  \ref{lemmaA} applies to $\bar{\cX}$. Therefore, either (D) occurs, or
$\bar{\cX}$ has zero $p$-rank. In the latter case, from Lemma \ref{leB11luglio2015}, $\bar{G}$ fixes a point $\bar{P}\in \bar{\cX}$.
Hence, the $N$-orbit consisting of all points of $\cX$ lying over $\bar{P}$ coincides with $\theta$. Then $\theta$ is also an $S$-orbit and $|\theta|=|N|/|N_P|=|S|/|S_p|$ holds. Write $G_P=G_P^{(1)}\rtimes H$ with a $p'$-subgroup $H$. Since $|H|\leq 4\gg(\cX)+2$ by (iii) of Result \ref{res60.79.108}.
%and $N$ is contained in $S$,
this yields
$$|G|=|G_P||\theta|\le |S_P||H|\frac{|S|}{|S_P|} \leq |S|(4\gg(\cX)+2).$$
From this and (\ref{eq13jan2015}),
$$|S|> \frac{c\gg(\cX)^2}{4\gg(\cX)+2}>3\gg(\cX).$$
Therefore $\cX$ has zero $p$-rank by Result \ref{resnaga}.
\end{proof}

Proposition \ref{princE} is a corollary of the above lemmas and propositions. In fact, let $\cX$ be a counterexample of minimal genus. From Proposition \ref{lemmaA} and Lemma \ref{leB11luglio2015}, case (B) occurs for $\cX$.
But this is ruled out \textcolor{black}{by}  Proposition \ref{le110luglio2005}.

A further ingredient of the proof of Theorem \ref{princA} is the following lemma.

\begin{lemma}
\label{lemmaC} Let $\cX$ be an algebraic curve of genus $\gg \geq 2$ and $p$-rank zero. Let $G$ be a solvable subgroup of $\aut(\cX)$. If $G$ has no non-trivial normal $p'$-subgroup then $G$ fixes a point of $\cX$.
\end{lemma}
\begin{proof} Let $N$ be a minimal normal subgroup of $G$. Since $G$ is solvable, $N$ is an elementary abelian group. By our hypothesis, $p\mid |N|$. Therefore $N$ is a normal $p$-subgroup of $G$.
Since $\cX$ has zero $p$-rank, $N$ fixes a unique point $P\in\cX$ by (i) of Result \ref{lem29dic2015}. Therefore $G$ fixes $P$.
\end{proof}

 Now we are in a position to give a proof for Theorem \ref{princA}.
From Proposition \ref{princE} together with Lemmas \ref{leB11luglio2015} and \ref{lemmaC}, the $p$-rank of $\cX$ is zero and $G$ fixes a point of $\cX$.
Then $G=S\rtimes C$ with a Sylow $p$-subgroup $S$ and a cyclic group $C$.  From (i) of Result \ref{lem29dic2015}, no non-trivial element in $S$ fixes a point other than $P$. Therefore, if $G$ has a short orbit distinct from $\{P\}$ then it is tame. Actually, it is also unique by
 Result \ref{res56.116} by virtue of the hypothesis  $|G|>24\gg^2$. Then  the cover $\cX|(\cX/G)$ is a Katz-Gabber $G$-cover.

%%%%%%%%%%%%%%%%%%%%%%%%%%%%%%%%%%%%%%%%%%%%%%%%%%%%%%%%%%%%%%%%%%%%%%%%%%%%%%%%%%%%%%%%%%%%%%%%%%%%%%%%%%%%%%%%%%%%%%%%%%%%%%%%%%%%%%
\section{Automorphism groups and $p$-rank}\label{sec4}
In this section, we provide some sufficient conditions for a curve to have $p$-rank equal to zero, which will be useful for the proof of Theorem \ref{princC}. Throughout the section $d$ stands for an odd prime \textcolor{black}{and $\cX$ is an algebraic curve of genus $\gg\ge 2$}.
We first show that for special solvable groups, the bound of Theorem \ref{princA} can be improved.
\begin{lemma}
\label{22dic2015} Let $H$ be an automorphism group of an algebraic curve of genus $\gg \ge 2$ with a normal \textcolor{black}{Sylow} $d$-subgroup $Q$ of odd order. %such that $|Q|$ and $[H:Q]$ are coprime.
 Suppose that a complement $U$ of $Q$ in $H$ is abelian, and that $N_H(U)\cap Q=\{1\}$.  If
\begin{equation}
\label{eq22bisdic2015}
{\mbox{$|H|\geq 30(\gg-1)$}},
% with $c\geq 30$}},
\end{equation}
then  $d=p$ and $U$ is cyclic. Moreover, the quotient curve $\bar{\cX}=\cX/Q$ is rational and either
\begin{itemize}
\item[\rm(i)]  $\cX$ has positive $p$-rank, $Q$ has exactly two (non-tame) short orbits, and they are also the only short orbits of $H$; or
\item[\rm(ii)] $\cX$ has zero $p$-rank and $H$ fixes a point.
\end{itemize}
\end{lemma}
\begin{proof}
We have $H=Q\rtimes U$. Set $|Q|=d^k$, $|U|=u$. Then $u\neq d$.
If $u=2$ then (\ref{eq22bisdic2015}) reads $|Q|>15(\gg-1)$. If $d=p$ then $\cX$ has zero $p$-rank by Result \ref{resnaga}; then (ii) holds, since $Q$ fixes a point by (i) of Result \ref{lem29dic2015}  and $Q$ is normal in $H$.  Otherwise, $Q$ is a prime to $p$ subgroup of $H$. The Hurwitz genus formula applied to $Q$ gives $|Q|<6(\gg-1)$ for $d\neq 3$ and $|Q|<12(\gg-1)$ for $d=3$, contradicting (\ref{eq22bisdic2015}). Therefore we may assume that $u\geq 3$.

Three cases are treated separately according as the quotient curve $\bar{\cX}=\cX/Q$ has genus $\bar{\gg}$ at least $2$, or $\bar{\cX}$ is elliptic, or rational.

If $\gg(\bar{\cX})\geq 2$, then $\aut(\bar{\cX})$ has a subgroup isomorphic to $U$, and (i) of Result \ref{res60.79.108} yields $4\gg(\bar{\cX})+4\geq |U|.$ Furthermore, from
the Hurwitz genus formula applied to $Q$, $\gg-1\geq |Q|(\gg(\bar{\cX})-1)$. Therefore,
$$(4\gg(\bar{\cX})+4)|Q|\geq |U||Q|=|H|\geq 30(\gg-1) \geq 30 |Q|(\gg(\bar{\cX})-1),$$
whence $$30 \leq 4\frac{\gg(\bar{\cX})+1}{\gg(\bar{\cX})-1}\leq 12$$
contradicting (\ref{eq22bisdic2015}).

If $\bar{\cX}$ is elliptic, then the cover $\cX|\bar{\cX}$ ramifies, otherwise $\cX$ itself would be elliptic. Thus, $Q$ has some short orbits. Take one of them together with its images
$o_1,\ldots,o_u$ under the action of $H$. Since $Q$ is a normal subgroup of $H$, $o=o_1\cup\ldots\cup o_m$ is a $H$-orbit of size $u_1d^v$ where $d^v=|o_1|=\ldots |o_m|$. Equivalently, the stabilizer of a point $P\in o$ has order
$d^{k-m}u/u_1$, and it is the semidirect product $Q_1\rtimes U_1$ where $|Q_1|=d^{k-v}$ and $|U_1|=u/u_1$ for a subgroup $Q_1$ of $Q$ and $U_1$ of $U$ respectively.
The point $\bar{P}$ lying under $P$ in the cover $\cX| \bar{\cX}$ is fixed by the factor group $\bar{U_1}=U_1Q/Q$. Since $\bar{\cX}$ is elliptic, Result \ref{res94} implies $|\bar{U}_1|\leq 12$ for $p=3$ and $|\bar{U}_1|\leq 6$ for $p>3$.
As $\bar{U}_1\cong U_1$, this yields the same bound for $|U_1|$, that is, $u\leq 12 u_1$ for $p=3$ and $u\leq 6 u_1$ for $p>3$. Furthermore, if $p=3$ then $d>3$, and $d_P\geq d^{k-m}-1\geq \textstyle\frac{4}{5} d^{k-m}$. If $p>3$ then $d_P\geq d^{k-m}-1\geq \textstyle\frac{2}{3} d^{k-m}$.
From the Hurwitz genus formula applied to $Q$, if $p=3$ then
$$2\gg-2\geq d^m u_1d_P\geq d^mu_1(\textstyle\frac{4}{5} d^{k-m}) \geq \textstyle\frac{4}{5} d^m u_1\geq \textstyle\frac{1}{15} d^mu=\textstyle\frac{1}{15}|Q||U|=\textstyle\frac{1}{15} |H|,$$
while for $p>3$,
$$2\gg-2\geq d^m u_1d_P\geq d^mu_1(\textstyle\frac{2}{3} d^{k-m}) \geq \textstyle\frac{2}{3} d^m u_1\geq \textstyle\frac{1}{9} d^mu=\textstyle\frac{1}{9}|Q||U|=\textstyle\frac{1}{9} |H|,$$
But this contradicts (\ref{eq22bisdic2015}).

If $\bar{\cX}$ is rational, then $Q$ has at least one short orbit. Furthermore, $\bar{U}=UQ/Q$ is isomorphic to a subgroup of $PGL(2,\K)\cong\aut(\bar{\cX})$. Since $U\cong \bar{U}$, Result  \ref{resdickson} shows that either
\begin{itemize}
\item[(I)] $U$ is an elementary abelian $p$-group and $\bar{U}$ fixes a point $\bar{P}_\infty$ but no non-trivial element of $\bar{U}$ fixes a point other than $\bar{P}_\infty$, or
\item[(II)]
$U$ is a cyclic group of order prime to $p$, $\bar{U}$ fixes two points $\bar{P}_0$ and $\bar{P}_\infty$ but no non-trivial element in $\bar{U}$ fixes a point other than $\bar{P}_0$ and $\bar{P}_\infty$.
\end{itemize}
Let $o_\infty$ and $o_0$ be the $Q$-orbits lying over $\bar{P}_0$ and $\bar{P}_\infty$, respectively. Obviously,
$o_\infty$ (and in case (II) also $o_0$) is a short orbit of $H$. We show that $Q$ has at most one  or two short orbits according to cases (I) and (II), the candidates being $o_\infty$ and in case (II) also $o_0$.
By absurd, there is a $Q$-orbit $o$ of size $d^m$ with $m<k$ which lies over a point $\bar{P}\in \bar{\cX}$ different from $\bar{P}_\infty$ (in case (II) from both $\bar{P}_0$ and $\bar{P}_\infty$). Since the orbit of $\bar{P}$ in $\bar{U}$ has length $u$, then the $H$-orbit of a point $P\in o$ has length $ud^m$. If $u>3$, the Hurwitz genus applied to $Q$ gives $$2\gg-2\ge -2d^k+ud^m(d^{k-m}-1)\geq -2d^k+ud^m\textstyle\frac{2}{3} d^{k-m}\geq -2d^k+\textstyle\frac{2}{3}ud^k\geq
\textstyle\frac{2}{3}(u-3)d^k>\textstyle\frac{1}{6}ud^k\geq \frac{1}{6}|H|,$$
a contradiction with $|H|>c\gg$ for $c=36$. If $u=3$ then $d\geq 3$, and hence $$2\gg-2\ge -2d^k+3d^m(d^{k-m}-1)=d^k-3d^m>\textstyle\frac{1}{3} d^k,$$
whence $|H|=3d^k<9\gg$, a contradiction with (\ref{eq22bisdic2015}).

In case (I), $u$ is a power of $p$ with $p\neq d$ while $|o_\infty|=d^k$. Hence $U$ must fix a point $P_\infty\in o_\infty$. As $o_\infty$ is a short $Q$-orbit, $H_P$ contains a non-trivial element $v$ from $Q$. From (i) of Result \ref{res74}, $U$ is a normal subgroup of $H_P$. Therefore, $w\in N_H(U)\cap Q$, a contradiction.

We are left with case (II). Since $H$ has exactly two short orbits, (\ref{eq22bisdic2015}) together with Result \ref{res56.116} yield that $p$-divides $|H|$. Therefore, $d=p$. Assume that $Q$ has two short orbits. They are $o_\infty$ and $o_0$. If their lengths are $p^a$ and $p^b$ with $a,b<k$, the Deuring-Shafarevich formula applied to $Q$ gives
$$\gamma(\cX)-1=-p^k+(p^k-p^a)+p^k-p^b$$
whence $\gamma(\cX)=p^k-(p^a+p^b)+1>0$. The same argument shows that if $Q$ has just one short orbit, then $\gamma(\cX)=0$ and the short orbit consists of a single point $P$. Since $Q$ is a normal subgroup of $H$, $P$ is also fixed by $H$.
\end{proof}

\begin{rem}\label{scholter} {\em{As a byproduct of the the above proof, if the hypotheses of Lemma \ref{22dic2015} are satisfied then Case (II) occurs, that is,
 $\gg(\bar{\cX})=0$, $U$ is a cyclic group of order prime to $p$, $\bar{U}=UQ/Q$ fixes two points $\bar{P}_0$ and $\bar{P}_\infty$ in $\bar{\cX}$
but no non-trivial element in $\bar{U}$ fixes a point other than $\bar{P}_0$ and $\bar{P}_\infty$.
In fact, the other possibilities have explicitly been ruled out in that proof for $u\ge 3$. The claim is true also for $u=2$  since $|Q|>15(\gg -1)$ yields
\textcolor{black}{$\gg(\bar{\cX})=0$} by
the last statement of Result \ref{res56.116}.}}
\end{rem}

\begin{rem}
\label{rem3jan2016} {\em{In Lemma \ref{22dic2015}, if $d=p$ is assumed then the extra-hypothesis $N_H(U)\cap Q=\{1\}$ is unnecessary, since in the proof of  Lemma \ref{22dic2015} the hypothesis $N_H(U)\cap Q=\{1\}$ is only used to rule out Case
(I), in which $d\neq p$ holds. In fact, if $Q$ is a $p$-subgroup then $U$ is a $p'$-subgroup, and (I) cannot occur.
Furthermore, Lemma \ref{22dic2015} applies whenever $H$ is taken for a group stabilizing a point of $\cX$ and satisfying (\ref{eq22bisdic2015})}.}
\end{rem}

\begin{rem}
\label{rem17mar2016}
{\em{Both Cases (i) and (ii) in Lemma \ref{22dic2015} occur as the following examples show. For Case (i), let $\cX$ be a non-singular model $\cX$ of the Artin-Mumford curve where $\cC$ is the plane irreducible (singular) curve $\cC$ of affine equation $(X^p-X)(Y^p-Y)-1=0$. The main properties of $\cX$ are well known; see \cite{CKK,AKK, kgmm}: $\cX$ is an ordinary curve $\cX$ of genus $\gg$ (and $p$-rank) equal to $(p-1)^2$. Furthermore, $\cC$ has two (ordinary) singular points $X_\infty$ and $Y_\infty$ both $p$-fold points. Each of them is the center of $p$ places of $\mathbb{K}(\cX)$.
Let $o_\infty$ be the set of the $p$ points of $\cX$ associated to the $p$ places centered at $X_\infty$, and $o_0$ that arising from $Y_\infty$ in the same way. For every $a,b\in \mathbb{F}_p$, the linear map $(X,Y)\to (X+a,Y+b)$ is an automorphism of $\cC$ and hence of $\cX$. The linear maps form an (elementary abelian) group $Q=C_p\times C_p$ of order $p^2$, and $Q$ has two (short) orbits of length $p$, namely $o_\infty$ and $o_0$. Also, for every $c\in \mathbb{F}_p^*$ the linear map $(X,Y)\to(cX,c^{-1}Y)$ is an automorphism of $\cC$ and hence of $\cX$. They form a cyclic group $C_{p-1}$ of order $p-1$ that preserves both $o_\infty$ and $o_0$. A further (linear) automorphism is $\varphi:\,(X,Y)\to (Y,X)$ which interchanges $o_\infty$ and $o_0$. The group generated by $C_{p-1}$ together with $\varphi$ is a dihedral group $D_{p-1}$. Therefore, $\aut(\cX)$ has a subgroup $G=(C_p\times C_p)\rtimes D_{p-1}$ with a dihedral group of order $2(p-1)$. If $\mathbb{K}=\overline{\mathbb{F}_p}$ or $(\mathbb{K},|\cdot|)$ is non-Archimedean valued field of characteristic $p$, then $\aut(\cX)=G$; see \cite{AKK} and \cite{CKK}. For $p\geq 31$, $|H|=p^2(p-1)\geq 30(\gg-1)$ and $H=(C_p \times C_p) \rtimes C_{p-1}$ satisfies the hypotheses of Lemma \ref{22dic2015}, and we have Case (i). For Case (ii), let $\cX$ be a non-singular model $\cX$ of the the plane irreducible (singular) curve $\cC$ of affine equation $Y^p+Y-X^m=0$ with $m<p$. For the main properties of $\cX$ see \cite{stichtenoth1973II}, and \textcolor{black}{ \cite[Theorems 12.9 and 12.11]{hirschfeld-korchmaros-torres2008}. If $m$ does not divide $p+1$ then} $\cX$ has genus $\ha (p-1)m$ and $p$-rank zero while $H=\aut(\cX)$ fixes the infinite point and $H=S\rtimes C$ with $|S|=p$ and $|C|=m(p-1)$. If $p\geq 17$ then $|H|>30(\gg-1)$, and we have Case (ii).
}}
\end{rem}
\begin{rem}
\label{rem16mar2016} {\em{In Lemma \ref{22dic2015}, if we assume $d\neq p$ and $[H:Q]\ge 3$ but drop the extra-hypothesis $N_H(U)\cap Q=\{1\}$ then Case (I) in the proof of Lemma \ref{22dic2015} occurs. Then $H$ has a unique short orbit $o_\infty$ and the quotient curve $\tilde{\cX}=X/H$ is rational. From the Abhyankar conjecture as stated in Section \ref{sec2},
$H$ is generated by $U$ together with its conjugates in $H$. In particular, $H$ is not a direct product of $Q$ and $U$.
We do not use this fact in the forthcoming proofs.
}}
\end{rem}
\begin{lemma}
\label{22dicter2015} Let $T$ be a subgroup of $\aut(\cX)$ containing a normal subgroup $H$ that satisfies the hypotheses of Lemma \ref{22dic2015}. Then $T=Q\rtimes V$ where $Q$ is a $p$-group, $V$ is a cyclic or a dihedral group, and $p$ is prime to $|V|$. If in addition $\cX$ has positive $p$-rank and $T$ fuses two short $H$-orbits together into one orbit then $V$ is dihedral.
\end{lemma}
\begin{proof} By Remark \ref{scholter} we may suppose that (II) in the proof of Lemma \ref{22dic2015} holds. Since $H$ is a normal subgroup of $T$, and $o_\infty$ and $o_0$ are the only short orbits of $H$, two possibilities
arise, namely either $T$ preservers both $o_\infty$ and $o_0$, or $o_\infty\cup o_0$ is a $T$-orbit. On the other hand, 
%as $[H:Q]$ is prime to $Q$ and $Q$ is a normal subgroup of $H$, 
$Q$ is the unique Sylow $p$-subgroup of $H$. Thus $Q$ is a characteristic subgroup of $H$, and hence $Q$ is a normal subgroup of $T$. Therefore, $\bar{T}=G/Q$ is a subgroup of $PGL(2,\K)$. It turns out that either $\bar{T}$ fixes both points $\bar{P}_0$ and $\bar{P}_\infty$, or $\bar{T}$ interchanges them. From Result \ref{resdickson}, $\bar{T}$ is cyclic in the former case, and it is a dihedral group in the latter case. In both cases $p$ is prime to $|\bar{T}|$. Therefore $Q$ is prime to $[T:Q]$, and hence $T=Q\rtimes V$ where $V$ is either cyclic or dihedral, and $p$ is prime to $|V|$.
\end{proof}
\begin{rem}
\label{abh}
{\em{In the dihedral case, some (involutory) elements of $\bar{T}$ fixes a point $\bar{P}$ distinct from $P_\infty$ and $P_0$. This agrees with the Abhyankar conjecture as stated in Section \ref{sec2} which shows that the possibility for $T$ to be dihedral in Lemma \ref{22dicter2015} can only occur when $T$ has a short orbit other than the fused one. This is the case when $\cX$ is a \textcolor{black}{non-singular} model of the Artin-Mumford curve \textcolor{black}{$\cC$}; see Remark \ref{rem17mar2016}. In fact, the automorphism $\varphi:\,(X,Y)\to (Y,X)$ of $\cC$ fixes each point $P(a,a)\in \cC$ with $a^p-a\pm 1=0$ and these points are outsides $o_\infty$ and $o_0$. We do not use this fact in the forthcoming proofs.}}
\end{rem}
\begin{lemma}
\label{lemA29dic2015} Let $G$ be a subgroup of $\aut(\cX)$ of order at least  $16\gg^2$ which satisfies property {\rm{(iii)}} of Result \ref{lem29dic2015}.
If $G$ has a Sylow $d$-subgroup $Q$ of odd order such that the normalizer $N_G(Q)$ satisfies  the hypotheses on $H$ in Lemma \ref{22dic2015}, that is
$Q$ has an abelian complement $U$ in $N_G(Q)$ with
$N_{N_G(Q)}(U) \cap Q=\{1\}$ and $|N_G(Q)|\geq 30(\gg-1)$,
 then $\cX$ has zero $p$-rank.
\end{lemma}
\begin{proof} By Lemma \ref{22dic2015}, $d=p$.  From Remark \ref{scholter} Case (II) in the proof of Lemma \ref{22dic2015} occurs. Assume by contradiction that $\cX$ has positive $p$-rank.
Then from the proof of Lemma \ref{22dic2015},  there exist exactly two short $Q$-orbits $o_\infty$ and $o_0$.
Each is left invariant by $N_G(Q)$, as well. If they are in two different orbits of $G$ then Case (b) in Result \ref{res56.116} with $|G|<16\gg^2$ holds, a contradiction.
Therefore, if  $P_\infty\in o_\infty$ and $P_0\in o_0$ then there exists $w\in  G$ taking $P_\infty$ to $P_0$.
Take two non-trivial elements $r_\infty,r_0\in Q$ fixing $P_\infty$ and $P_0$, respectively.
The conjugate $r$ of $r_\infty$ by $w$ fixes $P_0$. Hence both $r$ and $r_0$ are in the stabilizer of $P_0$ in $G$. From (i) of Result \ref{res74},
 a $p$-subgroup $R$ of $G$ contains both $r$ and $r_0$. Since the conjugate $Q'$ of $Q$ by $w$ is a Sylow $p$-subgroup of $G$, $R$ has non-trivial
 elements from both two Sylow $p$-subgroups $Q$ and $Q'$. Since $G$ is a group in which any two distinct Sylow \textcolor{black}{$p$-subgroups} have trivial intersection,
 it turns out that $Q=Q'$, that is, $w\in N_G(Q)$. This is a contradiction as $o_\infty$ is \textcolor{black}{an}  $N_G(Q)$-orbit.
\end{proof}
\begin{lemma}
\label{lemA30dic2015} Let $G$ be a subgroup of $\aut(\cX)$ with a non-trivial $p$-subgroup $Q$ that satisfies the following properties.
\begin{itemize}
\item[\rm(i)] $Z(G)$ has a non-trivial element $g$ of order $t$ with $t<p$,
\item[\rm(ii)] The normalizer of $Q$ in $G$ contains a subgroup $N$  such that $N=Q\rtimes M$ with an abelian group $M$ of order prime to $p$ properly containing $\langle g\rangle $.
\end{itemize}
If $\bar{\cX}=\cX/\langle g \rangle$ has zero $p$-rank and $\tilde{\cX}=\cX/Q$ is rational then $\cX$ has also zero $p$-rank.
\end{lemma}
\begin{proof} Let $U=\langle g \rangle$. The quotient group $\bar{Q}=QU/U$ is a $p$-subgroup of $\aut(\bar{\cX})$. Since $\gamma(\bar{\cX})=0$, there exists a point $\bar{P}\in \bar{\cX}$ fixed by $\bar{Q}$ but no non-trivial element in $\bar{Q}$ fixes a point of $\bar{\cX}$ other than $\bar{P}$. Let $o$ be the $U$-orbit in $\cX$ lying over $\bar{P}$ in the cover $\cX|\bar{\cX}$. Since $|o|\leq t$ and $t<p$, $Q$ fixes $o$ pointwise. No point $R\in\cX$ other than those in $o$ is fixed by $Q$, otherwise the point $\bar{R}$ lying under $R$ in the cover $\cX|\bar{\cX}$
would be a fixed point of $\bar{Q}$ distinct from $\bar{P}$. Therefore, $N$ preserves $o$. From (ii), $|o|\le t<|M|$. Hence, some non-trivial element $m\in M$ fixes a point in $o$. Since $M$ is abelian and contains $U$, $m$ fixes $o$ pointwise.

Furthermore, the cover $\cX|\tilde{\cX}$ totally ramifies at the points of $\tilde{\cX}$ lying under those in $o$ while it is unramified elsewhere. Now look at the action of the quotient group $\tilde{N}=N/Q$ as a subgroup of $\aut(\tilde{\cX})$. In the natural group homomorphism \textcolor{black}{$N\to \tilde{N}$}, the image $\tilde{m}$ of $m$ is a non-trivial automorphism of $\tilde{\cX}$ which fixes each of the  $|o|$ points of $\tilde{\cX}$ lying under the points of $o$ in the cover $\cX|\tilde{\cX}$. Since $\tilde{\cX}$ is rational, $\tilde{N}$ is isomorphic to a subgroup of $PGL(2,\K)$. Moreover,   $\tilde{N}$ is abelian by $M\cong\tilde{N}$. From Result \ref{resdickson}, $\tilde{N}$ is a cyclic group of order prime to $p$ and hence it fixes exactly two points in $\tilde{\cX}$, and no non-trivial element of $\tilde{N}$ fixes a further point in \textcolor{black}{$\tilde{\cX}$}; see \cite[Theorem 1]{maddenevalentini1982}. In particular, the image $\tilde{g}$ of $g$ has exactly two fixed points in $\tilde{\cX}$, and they are also the fixed points of \textcolor{black}{$\tilde{m}$}. From this, $|o|\leq 2$. If equality holds then both points in $o$ must be fixed by $N$. But this is impossible since $o$ is an $U$-orbit and $U\leq N$.
Therefore, $|o|=1$, and hence $Q$ has exactly one fixed point in $\cX$. Since $\gamma(\tilde{\cX})=0$, the Deuring-Shafarevich formula  yields that $\gamma(\cX)=0$ also holds.
\end{proof}

\section{Some low dimensional linear groups as automorphism groups of curves}\label{sec5}
%Now, we exhibit some applications of Lemmas \ref{22dic2015} and \ref{22dicter2015} that will be useful in Section \ref{sectc}.
For a \textcolor{black}{prime power} $q=d^k$ with an odd prime $d$, let $W$ denote one of the following non-solvable linear groups
\begin{equation}
\label{25dic2015}
{\mbox {$PSL(2,q)$, $PGL(2,q)$, with $q\geq 5$, $PSU(3,q)$, $PGU(3,q)$, $PSL(3,q)$, $PGL(3,q)$.}}
\end{equation}
Choose a Sylow $d$-subgroup $Q$ of $W$. If $W\neq PSL(3,q),\,PGL(3,q)$ then the normalizer $H$ of $Q$ is a semidirect product $Q\rtimes C_m$ with a cyclic complement $C_m$ of order $m$.
Here $(|Q|,m)$ is $(q,\ha(q-1))$, or $(q,q-1)$, or $(q^3,(q^2-1)/\mu)$, or $(q^3,q^2-1)$ according as $W$ is $PSL(2,q)$, or $PGL(2,q)$, or $PSU(3,q)$ or $PGU(3,q)$ where $\mu={\rm{gcd}}(3,q+1)$. In these cases,
$W$ satisfies property {\rm{(iii)} of Result \ref{lem29dic2015}.} Some but not all of these facts remain true for $PSL(3,q)$ and $PGL(3,q)$. For $PSL(3,q)$ with $q=d^k$ and $d>2$ prime, the normalizer $H$ of a Sylow $d$-subgroup $Q$ is a semidirect product of $Q$ by $C_{q-1}\times C_{(q-1)/\mu}$ where $\mu$ is the greatest common divisor of $3$ and $q-1$. The same holds for $PGL(3,q)$ with $\mu=1$ for any $q$. $PSL(3,q)$ (and hence $PGL(3,q)$) does not satisfy property %(ii) of Lemma \ref{lemA29dic2015}
\rm{(iii)} of Result \ref{lem29dic2015}.

Throughout the section $d$ stands for an odd prime \textcolor{black}{and $\cX$ is an algebraic curve of genus $\gg\ge 2$}.
\begin{lemma}
\label{27dicbis2015} Let $W$ be one of the groups on the list {\rm{(\ref{25dic2015})}},  $W\neq PSL(3,q), PGL(3,q)$.
%where $W$ is one of the groups on the list {\rm{(\ref{25dic2015}).}}
If $W$ is isomorphic to a subgroup of $\aut(\cX)$ with
\begin{equation}
\label{eq125dice2015} |W|>900\gg^2,
\end{equation}
%with $c=\sqrt{q-2}$ for $W=PSL(2,q), PGL(2,q)$ and $c=\sqrt{q^2-2q}$ for $W=PSU(3,q),PGU(3,q)$,
then $\cX$ has zero $p$-rank and $q$ is a power of $p$.
\end{lemma}
\begin{proof} Let $H$ be the normalizer of a Sylow $d$-subgroup $Q$ of $W$. For $W\cong PSL(2,q)$, we have $|W|=(q+1)|H|$ with $|H|=\textstyle\frac{1}{2}q(q-1)$ whence $|H|^2>|W|>900\gg(\cX)^2$ which yields $|H|>30(\gg(\cX)-1)$. As $900>16$, the assertions follow from Lemmas \ref{lemA29dic2015} and \ref{22dic2015}. Similar arguments apply to the other cases.
\end{proof}
%%modifica 1 ottobre
Recall that $\aut(W)=P\Gamma L(2,q)$ when $W=PSL(2,q)$, or $W=PGL(2,q)$ while $\aut(W)=P\Gamma L(3,q)$ when $W=PSU(3,q)$ or $W=PGU(3,q)$; see for instance \cite[Sections 3.3.4, 3.6.3]{RAW}.
So, $W\leq \aut(W)$ may be assumed.
%%fine
%Let $G$ denote any subgroup of $\aut(W)$ containing $W$.

\begin{rem}
\label{28feb2016}
{\em{
%%modifica 5 aprile 2018
We point out that in the natural representation of $P\Gamma L(2,q)$, the $2$-point stabilizer is neither abelian  \textcolor{black}{nor dihedral}. More precisely, a semilinear subgroup of $P\Gamma L(2,q)$ containing $PSL(2,q)$ is non-abelian and, apart from the smallest value of $q$, namely $q=9$ (and $|G|=720)$, it is also non-dihedral. The same result is valid for the $2$-point stabilizer of $P\Gamma U(3,q)$ including all cases.

In its natural representation, $P\Gamma L(2,q)$ is the semilinear group on the projective line $PG(1,q)$ over $\mathbb{F}_q$. Its subgroup $PSL(2,q)$ acts on $PG(1,q)$ as a doubly transitive permutation group such that its stabilizer $M$ of the origin $O$ and the infinite point $\infty$ has order $\ha (q-1)$ and consists of all permutations $x\mapsto ax$ with a non-zero quadratic element $a$ in $\mathbb{F}_q$. The stabilizer of $\infty$ is $H=Q\rtimes M$ where $Q$ is the Sylow $d$-subgroup of $PSL(2,q)$ consisting of all translations $x\mapsto x+c$ with $c\in \mathbb{F}_q$. If a subgroup $G$ of $P\Gamma L(2,q)$ contains $PSL(2,q)$ then the stabilizer of $\infty$ in G is $G_{\infty,O}\rtimes Q$ where $G_{\infty,O}$ is the stabilizer of $\infty$ and $O$ in $G$, and $M$ is a normal subgroup of $G_{\infty,O}$. Suppose that $G$ contains $PSL(2,q)$ properly. Then $G_{\infty,O}$ has an element $v$ other than those in $M$. If $v\in PGL(2,q)$ then $G$ also contains $PGL(2,q)=\langle PSL(2,q),v \rangle$.  Since $PGL(2,q)$ is the only linear subgroup of $P\Gamma L(2,q)$ containing $PSL(2,q)$ properly, it turns out that if $G$ is not linear then  $G_{\infty,O}$ contains a permutation $\varphi:\,x\mapsto \alpha x^\sigma$ where $\alpha\in \mathbb{F}_q$ and $\sigma$ is a non-trivial automorphism of $\mathbb{F}_q$. Then $\varphi\psi_a \neq \psi_a\varphi$ for $\psi_a:\,x\mapsto ax$ with $a\in \mathbb{F}_q$ such that $\sigma(a)\neq a$. Hence $G_{\infty,O}$ is not abelian. We show that $G_{\infty,O}$ is neither dihedral for $|G|>720$. Since $M$ is a cyclic subgroup of order $\ha (q-1)/2\geq 4$, the index $2$ cyclic subgroup of any dihedral group containing $M$ also contains $M$. As we have already showed the centralizer of $M$ does not contain $\varphi$. Therefore, if $G_{\infty,O}$ is assumed to be dihedral then $\varphi$ is an element in the coset of the index $2$ cyclic subgroup of $G_{\infty,O}$. Thus $\varphi^2=1$, that is, $\sigma$ is involutory and $\alpha^{\sigma+1}=1$. Furthermore, $\varphi\psi_a\varphi=(\psi_a)^{-1}=\psi_{a^{-1}}$ holds for every $a\in D$. Therefore, $a^{\sigma+1}=1$ for every \textcolor{black}{non-zero quadratic element} $a \in \fq$. Since the number of such elements is $\ha (q-1)$, this occurs if and only if $d=3,q=9$. Hence, either $G=P\Gamma L(2,9)$, and $|G|=1440$, or $G$ has order $720$. The former case $G=P\Gamma L(2,9)$, as $P\Gamma L(2,9)_{\infty,O}$ is a semi-dihedral group of order $16$. In the latter case, $G$ is one of the two subgroups of $P\Gamma L(2,9)$ of order $720$, other than $PGL(2,9)$. Therefore, if $G_{\infty,O}$ is dihedral then $|G|=720$.

In its natural representation, $P\Gamma U(3,q)$ is the semilinear group \textcolor{black}{preserving}  the set of all $\mathbb{F}_{q^2}$-rational points of the Hermitian curve $\cH$ with equation $y^q+y=x^{q+1}$. Its subgroup $PSU(3,q)$ acts as a doubly transitive permutation group.
Let $D$ be the subgroup of index $\mu$ in the multiplicative group of $\mathbb{F}_{q^2}$. Then a Sylow $d$-subgroup $Q$ of $PSU(3,q)$ consists of all maps $(x,y)\mapsto (x+a,y+a^qx+b)$ with $a,b\in \mathbb{F}_{q^2}$ and $b^q+b=a^{q+1}$. The subgroup $PSU(3,q)$ acts on $\cH$ as \textcolor{black}{a}  doubly transitive permutation group such that its stabilizer $M$ of the origin $O=(0,0)$ and the infinite point $\infty$ of the $y$-axis has order $(q^2-1)/\mu$ and it consists of all maps $\psi_{a,b}:\,(x,y)\mapsto (ax,by)$ with $a\in D$ and $b=a^{q+1}$. Suppose that a subgroup $G$ of $P\Gamma U(3,q)$ contains $PSU(3,q)$ properly. Then the stabilizer $G_{\infty,O}$ of $G$ contains an element $v$ other than those in $M$. If $v\in PGU(3,q)$ (and $\mu=3$) then $G$ also contains
 $PGU(3,q)=\langle PSU(3,q),v \rangle$.  Since $PGU(3,q)$ is the only linear subgroup of $P\Gamma U(3,q)$ containing $PSU(3,q)$ properly, it turns out that if $G$ is not linear then  $G_{\infty,O}$ contains a map $\varphi:\,(x,y)\mapsto (\alpha x^\sigma,\beta y^\sigma)$ where $\alpha\in D$, $\beta=\alpha^{q+1}$ and $\sigma$ is a non-trivial automorphism of $\mathbb{F}_{q}$. Then $\varphi\psi_{a,b} \neq \psi_{a,b}\varphi$ for $a\in \mathbb{F}_q$ such that $\sigma(a)\neq a$. Hence $G_{\infty,O}$ is not abelian. We show that $G_{\infty,O}$ is neither dihedral. Since $M$ is a cyclic subgroup of order $(q^2-1)/\mu \geq 3$, the index $2$ cyclic subgroup of any dihedral group containing $M$ also contains $M$. The centralizer of $M$ does not contain $\varphi$. Therefore, if $G_{\infty,O}$ is dihedral by absurd, then $\varphi$ is in the coset of the index $2$ cyclic subgroup of $G_{\infty,O}$. Hence $\sigma$ is involutory, $\alpha^{\sigma+1}=1$, and $\varphi\psi_{a,b}\varphi=(\psi_{a,b})^{-1}=\psi_{a^{-1},b^{-1}}$ holds for every $\psi_{a,b}\in M$. From this, $a^{\sigma+1}=1$ follows for every $a \in D$. Since $|D|=(q^2-1)/\mu$ this yields $d^h+1=(d^{2k}-1)/\mu$ with $\sigma=d^h$. As $\mu=1$ or $\mu=3$, this is impossible. Therefore $G_{\infty,O}$ is not dihedral.}}
 \end{rem}
 \begin{lemma}
\label{25dicbis2015} Let $G$ be a subgroup of $\aut(W)$ containing $W$ where $W$ is one of the groups on the list {\rm{(\ref{25dic2015}).}} If $G$ is also a subgroup of $\aut(\cX)$ up to isomorphism, and
\begin{equation}
\label{eq125Edice2015} |G|>\textcolor{black}{c\cdot 900\gg^2}
\end{equation}
with
\[c=\left\{
\begin{array}{lll}
\frac{2k(q+1)}{(q-1)q} & \mbox{if\,\, $W=PSL(2,q)$},\\[2mm]
\frac{k(q+1)}{(q-1)q} & \mbox{if\,\, $W=PGL(2,q)$,}\\[2mm]
\frac{2k\mu(q^3+1)}{(q^2-1)q^3} & \mbox{if\,\, $W=PSU(3,q)$,}\\[2mm]
\frac{2k(q^3+1)}{(q^2-1)q^3} & \mbox{if\,\, $W=PGU(3,q),$}\\[2mm]
\frac{k\mu(q^2+q+1)(q+1)}{(q^2-1)q^3} & \mbox{if\,\, $W=PSL(3,q)$,}\\[2mm]
\frac{k(q^2+q+1)(q+1)}{(q^2-1)q^3} & \mbox{if\,\, $W=PGL(3,q),$}
\end{array}
\right.
\]
then $G$ is linear, and $W\neq PSL(3,q), PGL(3,q)$.
\end{lemma}
\begin{proof} We begin with the case $W=PSL(2,q)$. Let $Q$ be a Sylow $d$-subgroup of $PSL(2,q)$ and $H$ its normalizer in $PSL(2,q)$. We may assume that $Q$ and $H$ are taken as in Remark \ref{28feb2016}. Then $T=G_\infty$ is a subgroup of $G$ containing $H$ as a normal subgroup such that $T=Q\rtimes G_{\infty,O}$. Assume that $G$ is nonlinear. From Remark  \ref{28feb2016}, $G_{\infty,O}$ is not cyclic and it is dihedral only for $d=3$ and $|G|=720$.
On the other hand,  $k\ge 2$ and  $|G|=\ha h(q-1)q(q+1)$ where $h>1$ is a divisor of $k$.
Assumption $|G|>900c\gg^2$ yields
\begin{equation}
\label{eq5Edic2015} |H|=\ha(q-1)q>30\sqrt {\frac{c(q-1)q}{2h(q+1)}}\,\gg>30 \gg.
\end{equation}
By (\ref{eq125Edice2015}), $H$ satisfies the hypotheses of Lemma \ref{22dic2015}. By Lemma \ref{22dicter2015}, $G_{\infty,O}$ is either cyclic or dihedral. Since $|G|>720$ this is impossible.

Now, let $W=PSU(3,q)$. By the above argument applied to a subgroup $G$ of $P\Gamma L(3,q)$ containing $PSU(3,q)$ properly, Remark \ref{28feb2016} shows that $T=G_\infty$ is a subgroup of $G$ containing $H$ as a normal subgroup such that $T=Q\rtimes G_{\infty,O}$ where $G_{\infty,O}$ is neither cyclic nor dihedral when $G$ is nonlinear. On the other hand,
 $|G|=\frac{1}{\mu}h(q^2-1)q^3(q^3+1)$ where $h$ is a divisor of $2k$. Assumption $|G|>900c\gg^2$ yields

\begin{equation}
\label{eqA25dic2015} |H|={\textstyle\frac{1}{\mu}}(q^2-1)q^3>30\sqrt {\frac{c(q^2-1)q^3}{\mu h(q^3+1)}}\, \gg>30 \gg.
\end{equation}
By (\ref{eq125Edice2015}), $H$ satisfies the hypotheses of Lemma \ref{22dic2015}. From Lemma \ref{22dicter2015}, $T=Q\rtimes G_{\infty,O}$ where $V$ is either cyclic or dihedral, a contradiction.

Finally, let $W=PSL(3,q)$. Then $|W|=\frac{1}{\mu}h(q-1)^2q^3(q^2+q+1)(q+1)$ where $h$ is a divisor of $k$. Then
\begin{equation}
\label{eqB25dic2015} |H|={\textstyle\frac{1}{\mu}}(q-1)^2q^3>30\sqrt {\frac{c(q-1)^2q^3}{\mu h(q^2+q+1)(q+1)}}\, \gg> 30 \gg.
\end{equation}
By (\ref{eq125dice2015}), the hypotheses on $H$ in Lemma \ref{22dic2015} are satisfied. From Lemma \ref{22dicter2015}, $T=Q\rtimes V$ where $V$ is either cyclic or dihedral. However, $H$ contains the direct product of two cyclic groups, and hence $V$ is neither cyclic, nor a dihedral group. This contradiction shows that $W\neq PSL(3,q)$. A similar argument shows that $W=PGL(3,q)$ cannot occur.
\end{proof}
%%%%%%%%%%%%%%%%%%%%%%%%%%%%%%%%%%%%modifica 15 marzo 2016

\begin{rem}
\label{rem7nov2016} {\em{In Lemma \ref{25dicbis2015} $c\le 1$ holds in each case}}.
\end{rem}

\begin{rem}
\label{remA16mar2016} {\em{In some but not all cases, Lemma \ref{25dicbis2015} follows from the Abhyankar conjecture under a bit weaker condition, namely  $|G|>24\gg^2$. This condition yields one of the Cases (c) and (d) in Result \ref{res56.116}. Furthermore, the quotient curve $\cX/G$ is rational, and  the Abhyankar conjecture stated in Introduction applies to $G$. If $G$ has one short orbit, that is in Case (c),  then $G$ is generated by its Sylow $p$-subgroups and hence $G \ge W$ implies $G=W$; thus $G$ is linear, and Lemma \ref{25dicbis2015} follows. If $G$ has two short orbits, that is in
Case (d), then $G/W$ is a cyclic group. For instance, this occurs when $G\cong PGL(2,q)$ or $G\cong PGU(3,q)$. The Abhyankar conjecture rules out the possibility that in Case (d) $G$ properly contains $PGL(2,q)$ or $PGU(3,q)$, as $G/W$ is not cyclic in these cases.}}
\end{rem}

The following lemma shows that the hypothesis $W\neq PSL(3,q), PGL(3,q)$ in Lemma \ref{27dicbis2015} is meaningful.
\begin{lemma}
\label{30ago2015} If $W=PSL(3,q)$ with $q>3$ is a subgroup of $\aut(\cX)$ then $|G|<72\gg^2$.
\end{lemma}
\begin{proof} As we have already pointed out $H=Q\rtimes(C_{q-1}\times C_{(q-1)/\mu})$. Let $M$ be the subgroup of $C_{(q-1)/\mu}$ of index $2$. Then the factor group $\bar{T}=(C_{q-1}\times C_{(q-1)/\mu})/M$ has order
 $2(q-1)$. Here $\bar{T}$ is a non-cyclic abelian group of order greater than $4$. Obviously, the subgroup $U=Q\rtimes M$ is a normal subgroup of $H$. Therefore, $\bar{H}=H/U$ is an automorphism group of the quotient curve $\bar{\cX}=\cX/U$. Since $\bar{H}\cong \bar{T}$, no subgroup of $PGL(2,\mathbb{K})$ is isomorphic to $\bar{H}$ by Result \ref{resdickson}. Therefore $\bar{\cX}$ is not rational. If $\bar{\cX}$ is elliptic then  $U$ has a short orbit $o$, and for a point $P\in o$,
 the contribution of $o$ to the different $D(\cX|\bar{\cX})$ is at least $|o|(|U_P|-1)\geq \ha o|U_P|=\ha |U|$.
 Assume that $o$ is the unique short orbit of $U$. Then $o$ is also an $H$-orbit. Thus $|H_P|/|U_P|=|H|/|U|$ whence
 $$|H_PU/U|=\frac{|H_PU|}{|U|}=\frac{|H_P||U|}{|H_P\cap U|}\cdot\frac{1}{|U|}=\frac{|H_P|}{|U_P|}=\frac{|H|}{|U|}=|\bar{T}|=2(q-1).$$
 This shows that $H_PU/U$ has order $2(q-1)$. Furthermore, $H_PU/U$ viewed as an automorphism group of $\bar{\cX}$ fixes the point $\bar{P}\in \bar{\cX}$ lying under $o$. But then $q<5$, as $\bar{\cX}$ is elliptic, see Result \ref{res94}. Therefore, there is another such orbit, and the Hurwitz genus formula gives $2(\gg-1)\geq |U|$ whence
 %\begin{equation}
 %\label{eq30agos2015}
 $$\gg-1\geq \ha\,|U|=\qa\frac{q-1}{\mu}\, q^3.$$
 %\end{equation}
 This holds true for $\gg(\bar{\cX})\geq 2$, by the Hurwitz genus formula. Therefore,
$$\textstyle{72\gg^2>72(\gg-1)^2\geq \frac{24}{16\mu} \frac{(q-1)^2}{\mu}q^6\geq \frac{3}{2}\, q^3\, q^3 \frac{(q-1)^2}{\mu}>(q^2+q+1)(q+1)q^3\frac{(q-1)^2}{\mu}=|G|},$$
a contradiction with the hypothesis of the lemma.
\end{proof}
A variant of the proof of Lemma \ref{25dicbis2015} can be used to prove a similar result for the group $SL(2,q)$ with $q=d^k$ and $d$ an odd prime. In the natural representation of $SL(2,q)$ as a linear group of the vector space $V(2,\mathbb{F}_q)$, a Sylow $d$-subgroup is the subgroup $Q$ consisting of all maps $(x,y)\mapsto (x+ry,y)$ with $r\in \mathbb{F}_q$, and the normalizer  of $Q$ in $SL(2,q)$ is the semidirect product $Q\rtimes M$ where
$M\cong C_{q-1}$ consists of all transformations $(x,y)\mapsto (ax,a^{-1}y)$ with $a\in \mathbb{F}_q^{*}$.
 \begin{lemma}
\label{26dic2015} Let $G$ be a group with commutator subgroup $G'=SL(2,q)$ such that the centralizer of $G'$ in $G$ has order $2$.
If $G$ is also a subgroup of $\aut(\cX)$, and
\begin{equation}
\label{eq125Fdice2015} |G|>900\gg^2
\end{equation}
%with
%\begin{equation}
%\label{eq526dic2015} c=\frac{4k(q+1)}{(q-1)q},
%\end{equation}
then $\cX$ has zero $p$-rank and $q$ is a power of $p$.
\end{lemma}
\begin{proof} For every $g\in G$, let $\varphi_g$ denote the automorphism of $G'$ taking the element $h\in G'$ to its conjugate $h^g$. The map from $G$ to $\aut(G')$ which takes $g$ to $\varphi_g$ is a group homomorphism whose kernel consists of the elements in $G$ which centralize $G'$. Hence, $|G|=2|L|$ with a subgroup $L$ of $\aut(G')$. From $\aut(SL(2,q))\cong P\Gamma L(2,q)$, we have that
$|G|\leq 2 k(q+1)q(q-1)$ where $q=d^k$ and \textcolor{black}{$d$ is an odd prime}.
%%%modifica 22sep 2016
Observe that if $4k(q+1)>(q-1)q$ then $q=3,5,9$, but for these values of $q$, $|G|\leq 2880<3600\leq 900 \gg^2$. Hence
\begin{equation}
\label{eq526dic2015} \frac{4k(q+1)}{(q-1)q}\leq 1.
\end{equation}
%%fine
For the normalizer $H=Q\rtimes U$ of $Q$ in $SL(2,q)$, assumption $|G|>900\gg^2$ yields
\begin{equation}
\label{eq5dic2015} |H|=(q-1)q>30\,\gg,
\end{equation}
which shows that $H$ satisfies the hypotheses in Lemma \ref{22dic2015}. Therefore, $d=p$ and the quotient curve $\tilde{\cX}=\cX/Q$ is rational. Take a point $P$ from one of the short orbits $\Omega_1$ of $H$. Then $|H_P|=|H||Q_P|=(q-1)p^r$ with a positive integer $r$.
By absurd, $\cX$ has positive $p$-rank and hence (i) of Lemma \ref{22dic2015} occurs.
We prove that Lemma \ref{lemA30dic2015} applies to $G$ where $g$ is the central involution of $G'$. Obviously, $g\in Z(G)$ as $G'$ is a characteristic subgroup of $G$.

It remains to show that the quotient curve $\bar{\cX}=\cX/V$ with $V=\langle g \rangle$ has zero $p$-rank. For this purpose, we need three facts:
\begin{itemize}
\item[(i)] $\bar{G}=G/V$ \textcolor{black}{is} isomorphic to a subgroup of $P\Gamma L(2,q)$ containing $PSL(2,q)$.
\item[(ii)] $\bar{G}$ is a subgroup of $\aut(\bar{\cX})$.
\item[(iii)] Let $\bar{H}=H/V$. If $\bar{P}$ is a point of $\bar{\cX}$ lying under $P\in \Omega_1$ in the cover $\cX|\bar{\cX}$, then $|\bar{H}_{\bar{P}}|=\ha (q-1)p^r$ with a positive integer $r$.
\end{itemize}
By (\ref{eq526dic2015}), Lemma \ref{25dicbis2015} for $W=PSL(2,q)$ applies to $\bar{G}$. Therefore, either $\bar{G}=PSL(2,q)$, or $\bar{G}=PGL(2,q)$. If $\bar\gg=\gg(\bar{\cX)}\geq 2$ then $\gg-1\geq 2(\bar{\gg}-1)$ and hence
$$|\bar{G}|=\ha |G|\geq 900 \gg^2\geq 900 \cdot\ha (2\bar{\gg}-1)^2> 900 \bar{\gg}^2,$$
which yields  $\bar{\gg}=0$ by Lemma \ref{27dicbis2015} applied to $\bar{G}$. If $\bar{\cX}$ is elliptic, (iii) together with Result \ref{res94} yield $p=3$ and $\qa(q-1)=|\bar{H}_P|\in\{6,12\}$ whence $q=9$. On the other hand $|PGL(2,9)|=2|PSL(2,9)|=720$. Hence $|G|\leq 1440<900 \gg^2$ as $\gg\geq 2$. Therefore, $\bar{\cX}$ is not elliptic.
%By (\ref{eq125Fdice2015}) and (\ref{eq526dic2015}), Lemma \ref{25dic2015} applies to $W=PSL(2,q)$, or $W=PGL(2,q)$. Therefore $\bar{\cX}$ has zero $p$-rank.
\end{proof}

\section{Structure of automorphism groups of curves of even genus}
\label{genusparicase}

%If a group $G$ contains a non-trivial normal subgroup of odd order, the {\emph{odd core}} of $G$ is its maximal normal subgroup of odd order. Otherwise, $G$ is a {\emph{odd core-free}} group.}

In this section we prove a classification theorem on the abstract structure of automorphism groups of algebraic curves with even genus.
It is worth mentioning that the list of the possible structures is quite short. As a consequence, the inverse Galois problem restricted on curves with even genus is is not solvable in general.
%Recall that  $O(G)$ stands for the largest normal subgroup of $G$ of odd order.
\begin{theorem}
\label{structure} If  $G$ is a subgroup of the automorphism group of some non-rational algebraic curves with even genus, defined over an algebraically closed field of odd characteristic $p$,  then one of the following cases occurs up to isomorphism:
\begin{itemize}
\item[\rm(i)] $G$ has odd order;
\item[\rm(ii)] $G=O(G)\rtimes S_2$ where $S_2$ is a $2$-group described in Result \ref{res26feb2018};
\item[\rm(iii)] the commutator subgroup of $G/O(G)$ is isomorphic to $SL(2,q)$ with $q\geq 5$;
%, apart from two sporadic cases described in Lemma \ref{lem26agos2015}.
\item[\rm(iv)] $PSL(2,q)\le G/O(G) \le P\Gamma L(2,q)$ with $q\ge 3$;
\item[\rm(v)] $PSL(3,q)\le G/O(G) \le P\Gamma L(3,q)$ with $q\equiv 3 \pmod 4$;
\item[\rm(vi)] $PSU(3,q)\le G/O(G) \le P\Gamma U(3,q)$ with $q\equiv 1 \pmod 4$;
\item[\rm(vii)] $G/O(G)={\rm{Alt}_7}$;
\item[\rm(viii)]  $G/O(G)=M_{11}$;
\item[\rm(ix)] $G/O(G)= GL(2,3)$;
\item[\rm(x)] $G/O(G)$ is the unique perfect group of order 5040 and $(G/O(G))/Z(G)\cong \rm{Alt}_7$;
\item[\rm(xi)] $G/O(G)$ is the group of order $48$ named SmallGroup(48,28) in the GAP-database.
\end{itemize}
\end{theorem}
In Subsection \ref{refinements} refinements of Theorem \ref{structure} \textcolor{black}{are} obtained under the additional \textcolor{black}{hypothesis} that the order of $G$
is large enough compared to genus $\gg$ of the curve; more precisely when $|G|$ is assumed to exceed \textcolor{black}{$900\gg(\cX)^2$}; see Proposition \ref{3apr2016}.

Throughout the section,  $\cX$ stands for an algebraic curve defined over an algebraically closed field $\K$ of odd characteristic $p$, whose genus $\gg$ is a (nonzero) even integer, and $G$ denotes a subgroup  \textcolor{black}{of} $\aut(\cX)$.

For $G$ of even oder, a key property is given in the following lemma.
\begin{lemma}
\label{lem23agos2015} If $G$ has even order, then any $2$-subgroup of $G$ has a cyclic subgroup of index $2$.
\end{lemma}
\begin{proof} Let $U$ be a subgroup of $\aut(\cX)$ of order $2^u\geq 2$. Since $p$ is odd, $U$ is tame. Therefore, the Hurwitz genus formula applied to $U$ gives
$$2\gg-2=2^u(2\bar{\gg}-2)+\sum_{i=1}^m(2^u-\ell_i)$$
where $\bar{\gg}$ is the genus of the quotient curve $\bar{\cX}=\cX/U$ and $\ell_1,\ldots,\ell_m$ are the short orbits of $U$ on $\cX$. Since $2(\gg(\cX)-1)\equiv 2 \pmod 4$ while
$2^u(2\bar{\gg}-2)\equiv 0 \pmod 4$, some $\ell_i$ ($1\le i \le m$) must be either $1$ or $2$.  Therefore, $U$ or a subgroup of $U$ of index $2$ fixes a point of $\cX$ and hence
is cyclic.
\end{proof}
\begin{rem}
\label{2sub} {\em{Lemma \ref{lem23agos2015} shows that Result \ref{res26feb2018} applies to any $2$-subgroup of $\aut(\cX)$}. }
\end{rem}

For solvable $G$, Theorem \ref{structure} is obtained as a corollary of Lemma \ref{lem23agos2015} together with Results \ref{res26feb2018} and \ref{reswong1}.
\begin{proposition}
\label{pro26feb2018} If $G$ is solvable then one of the following cases occurs.
\begin{itemize}
\item[\rm(i)] $G=O(G)\rtimes S_2$ where $S_2$ is a Sylow $2$-subgroup of $G$.
\item[\rm(ii)] $G/O(G)\cong PSL(2,3)$, or $G/O(G)\cong PGL(2,3)$.
\item[\rm(iii)] $G/O(G)\cong GL(2,3)$.
\end{itemize}
\end{proposition}
If  $G$ contains no non-trivial normal subgroup of odd order, then $G$ is an {\emph{odd core-free}} group. We first prove Theorem \ref{structure} in this case.

\subsection{The case $|O(G)|=1$}\label{ssec61}
We show that the hypothesis of $\gg$ to be even gives a heavy restriction on the structure of the minimal normal subgroups of an odd-core free \textcolor{black}{subgroup} of $\aut(\cX)$.

\begin{lemma}\label{MNSOCF} If $G$ is odd core-free, then a minimal normal subgroup $N$ of $G$ is either a $2$-group with $|N|\le 4$, or a non-abelian simple group.
\end{lemma}
\begin{proof}
Since $N$ is characteristically simple, if $N$ is solvable then $N$ is an elementary abelian group, otherwise it is a direct product of pairwise isomorphic non-abelian simple groups; see, for instance, \cite[Kapitel I, Satz 9.13]{huppertI1967}.
In the former case, $N\cong \mathbb Z_2^h$ since $G$ is odd core-free, and $h\le 2$ otherwise $N$ would be a $2$-group without a cyclic subgroup of index $2$. If $N$ is the direct product of pairwise isomorphic simple groups, let $M$ one of these simple subgroups of $N$. From Result \ref{resbs}, $M$ is neither cyclic nor a generalized quaternion group. By the remarks made after Result \ref{res26feb2018}, a Sylow $2$-subgroup of $N$
contains an elementary abelian group of
order $4$, say $V$. Then $M=N$, otherwise $V$ and any involution from another factor in that direct product would generate an elementary abelian subgroup of $N$ of order $8$, which is a $2$-group without a cyclic subgroup of index $2$.
 \end{proof}

Three cases are now treated separately, according as a minimal normal subgroup $N$ has order $2$, $4$, or is a simple group.

\begin{lemma}
\label{lemA23agos2015} If $G$ has a (minimal) normal subgroup $N$ of order $2$ then $G$ has a $2$-complement, unless a Sylow $2$-subgroup of $G$ is either a generalized quaternion group, or a  semidihedral group. In the semidihedral case, $G$ has an index $2$ subgroup $M$ such that either $M$ has a $2$-complement, or a Sylow $2$-subgroup of $M$ is a generalized quaternion group.
If, in addition, $G$ is solvable and has no $2$-complement then $G/O(G)\cong GL(2,3)$ and $G$ has a semidihedral Sylow $2$-subgroup of order $16$.
\end{lemma}
\begin{proof} If a Sylow $2$-subgroup $S_2$ of $G$ is dihedral then Result \ref{resgor} yields that $G$ has a normal $2$-complement. In fact, $NO(G)/O(G)$ is a non-trivial subgroup of $G/O(G)$ contained in the center of $G/O(G)$ whereas any subgroup of $P\Gamma L(2,q)$ containing $PSL(2,q)$, as well as ${\rm{Alt}}_7$, have trivial center.
By the first claim in Result \ref{reswong1}, we are left with the case where $S_2$ is a semidihedral group. Then $S_2$ has a cyclic subgroup $U$ of index $2$ such that $S_2\setminus U$ contains an involution $v$. From Result \ref{resthompson}, either $G$ has a a subgroup of index $2$ containing $U$, or $v$ is conjugate to an involution in $U$. The latter case cannot actually occur in our situation, since the normal subgroup $N$ consists of the unique involution of $U$ together with the identity. Therefore, $G$ has a subgroup of index $2$, and either (i) or (ii) in Result \ref{reswong} holds. In case (i), $G$ has an index $2$ subgroup $M$ with a dihedral Sylow $2$-subgroup, and hence $M$ has a $2$-complement.  In case (ii), $G$ has an index $2$ subgroup $M$ with a generalized quaternion Sylow $2$-subgroup. If $G$ is solvable, the second claim in Result \ref{reswong1} applies.
\end{proof}
\begin{rem}
\label{rem18jan2017}{\emph{Lemma \ref{lemA23agos2015} applies to $SL^{\pm}(2,q)$ and $SU^{\pm}(2,q)$ and shows that $\aut(\cX)$ has no subgroup isomorphic to $SL^{\pm}(2,q)$ for $q\equiv 1 \pmod 4$ and to $SU^{\pm}(2,q)$ for $q\equiv -1 \pmod 4$.}}
\end{rem}
\begin{lemma}
\label{lemB23agos2015}
If $G$ has a minimal normal subgroup $N$ of order $4$, then either $G=O(G)\rtimes S_2$ where $S_2$ is the direct product of a cyclic
group by a group of order $2$, or $G/O(G)\cong PSL(2,3), PGL(2,3)$.
\end{lemma}
\begin{proof} Let $V$ be a Sylow $2$-subgroup of the centralizer $T=C_G(N)$ of $N$ in $G$. Then $N\le V$. Since $N\cong C_2\times C_2$, Result \ref{res26feb2018} shows that $V$ must be abelian. Hence
$V\cong C_{2^m}\times C_2$ with $m\geq 1$. From Result \ref{reswong1}, $T$ has a normal $2$-complement. We investigate the case where $T$ is a proper subgroup of $G$. In this case, $G/T$ is isomorphic to subgroup of ${\rm{Sym}}_3$ and hence it is solvable. Since $T$ is also solvable, this yields that $G$ is itself  \textcolor{black}{solvable}. Since $GL(2,3)$ has no normal subgroup of order $4$, Result \ref{reswong1} shows that either $G/O(G)\cong PSL(2,3)$ or $G/O(G)\cong PGL(2,3)$.
\end{proof}
\begin{lemma}
\label{lemA24agos2015} If $G$ is an odd core-free group and has a non-abelian minimal normal subgroup $N$ then
%$N$ is a simple group and
one the following cases occurs.
\begin{itemize}
\item[\rm(i)] $N\cong PSL(2,q)$ with $q\geq 5$ odd, and a Sylow $2$-subgroup of $N$ is dihedral;
\item[\rm(ii)] $N\cong PSL(3,q)$ with $q\equiv 3 \pmod 4$, and a Sylow $2$-subgroup of $N$ is semidihedral;
\item[\rm(iii)] $N\cong PSU(3,q)$ with $q\equiv 1 \pmod 4$, and a Sylow $2$-subgroup of $N$ is semidihedral;
\item[\rm(iv)] $N\cong {\rm{Alt}_7}$, and a Sylow $2$-subgroup of $N$ is dihedral;
\item[\rm(v)]  $N\cong M_{11}$, the Mathieu group on $11$ letters, and a Sylow $2$-subgroup of $N$ is semidihedral.
\end{itemize}
\end{lemma}
\begin{proof}
By Lemma \ref{MNSOCF} $N$ is simple.
From Result \ref{resalp}, a Sylow $2$-subgroup of $N$ is either dihedral, or semidihedral, or wreathed, or isomorphic to a Sylow $2$-subgroup of $SU(3,2)$. Actually, by Lemma \ref{lem23agos2015} and Result \ref{res26feb2018}, the latter two cases cannot occur in our situation. Now, the claim follows from Results \ref{resgor} and \ref{resabg}.
\end{proof}

Theorem \ref{structure} for odd core-free groups is a corollary of Lemmas \ref{lem26agos2015}, \ref{lem2mar2018}, and \ref{lem29agos2015}.
%For odd core-free groups $G$, Lemma \ref{lemA23agos2015} are refined in the following two lemmas.
\begin{lemma}
\label{lem26agos2015} Let $G$ be an odd core-free group with generalized quaternion Sylow $2$-subgroups. Then its commutator subgroup $G'$ is isomorphic to $SL(2,q)$ with $q\ge 5$, apart from three cases:
\begin{itemize}
\item[\rm(i)] $G$ is a generalized quaternion group,
\item[\rm(ii)]  $G$ is isomorphic to the unique perfect group of order $5040$ and $G/Z(G)\cong {\rm{Alt}_7}$,
\item[\rm(iii)] $G$ is isomorphic to the group of order $48$  named $SmallGroup(48,28)$ in the GAP-database.
\end{itemize}
\end{lemma}
\begin{proof}
%%%modifica 1 ottobre 2016
If $G$ is a $2$-group, Case (i) holds. From now on $|G|$ is assumed to have some odd divisor.

%%fine
From Result \ref{resbs}, $|Z(G)|=2$.
We point out that the factor group $\bar{G}=G/Z(G)$ is also odd core-free. By absurd, let $\bar{N}$ be a non-trivial normal subgroup of $\bar{G}$ of odd order. Then there exists a normal subgroup $N$ of $G$ such that $\varphi(N)=\bar{N}$ where $\varphi$ is the natural homomorphism  $G\to \bar{G}$. Since $G$ is odd core-free, $N$ has even order.
Since the involution $z$ generating $Z(G)$ is the unique involution in $G$, $N$ must contain $z$. Since $|\bar{N}|$ is odd, $|N|$ is twice \textcolor{black}{an} odd number bigger than $1$. By $z\in N$, this yields $N=Z(G)\times O(N)$. Therefore, $O(N)$ is non-trivial.
%%fine modifica
Since $N$ is a normal subgroup of $G$ and $O(N)$ is a characteristic subgroup of $N$, this yields that $O(N)$ is a normal subgroup of $G$, a contradiction as $G$ is odd core-free. Hence $\bar{G}$ is odd core-free, as well. Furthermore, a Sylow $2$-subgroup of $\bar{G}$ is $\bar{S}_2=S_2/Z(S)$ which is a dihedral group as $S_2$ a generalized quaternion group.

{}From Result \ref{resgor}, $\bar{G}$ is either a dihedral group whose order is a power of $2$ in which Case (i) holds, or $\bar{G}\cong\rm{Alt}_7$, or $PSL(2,q)\leq \bar{G}\leq P\Gamma L(2,q)$ with $q$ odd, up to isomorphism. For $q\geq 5$, $PSL(2,q)$ is a perfect group, while the commutator subgroup of $PSL(2,3)$ is an elementary abelian group of order $4$.

If $q=3$ then either $\bar{G}\cong PSL(2,3)$, or $\bar{G}\cong PGL(2,3)$. By direct GAP-aided computation, $G\cong SL(2,3)$ in the former case while if $\bar{G}\cong PGL(2,3)$,  the only possibility for $G$ is the group $SmallGroup(48,28)$ in the GAP database. This gives (iii).

Suppose $q\geq 5$. As $G'$ is a normal subgroup of $G$, and $G$ is odd core-free, $G'$ has even order. Since $z$ is the unique involution of $G$, $z$ is contained in $G'$, that is,
$Z(G)\leq G'$. Thus, $G'/Z(G)= G'Z(G)/Z(G)$. On the other hand, $G'Z(G)/Z(G)$ is the commutator subgroup of $\bar{G}$, see \cite[Kapitel I, Hilfssatz 8.4]{huppertI1967}, and, since $q\geq 5$, either $G'/Z(G)\cong PSL(2,q)$, or  $G'/Z(G)\cong\rm{Alt}_7$.

The commutator subgroup $H$ of $G'$ is a characteristic subgroup of $G'$, and hence a normal subgroup of $G$. Therefore, $H$ has even order, and hence contains $Z(G)$. From this, either $H/Z(G)\cong PSL(2,q)$, or $H/Z(G)\cong\rm{Alt}_7$. Since $H\le G'$, this yields $H=G'$ showing that $G'$ is a perfect group, and hence $G'$ is perfect non-split extension
of either $PSL(2,q)$ or $\rm{Alt}_7$. Now, the first claim in Lemma \ref{lem26agos2015} follows from Result \ref{resschur}.
Finally, ${\rm{Alt_7}}$ has a unique non-split central extension by a cyclic group of order $2$ which is the group of order $5040$ given in (ii).
\end{proof}
%\begin{rem} {\emph{The proof of Lemma \ref{lem26agos2015} can be shortened using \cite[Proposition 4]{conder} stated without proof.}} \end{rem}
\begin{lemma}
\label{lem2mar2018} Let $G$ be an odd core-free group with a semidihedral Sylow $2$-subgroup $S_2$. If $G$ has a (minimal) normal subgroup of order $2$ then either $G=S_2$, or its commutator group $G'$ is isomorphic to $SL(2,q)$.
\end{lemma}
\begin{proof} If $G$ is solvable the assertion follows from Lemma \ref{lemA23agos2015}.
Assume that $G$ has \textcolor{black}{even} order, and take a minimal normal subgroup $N$ of $G$ such that $|N|=2$. Since $S_2$ is semidihedral, the quotient group $S_2/N$ is dihedral. Therefore, the quotient group $\bar{G}=G/N$ has a dihedral Sylow $2$-subgroup, and Result \ref{resgor} applies. First the case where $\bar{G}$ is isomorphic to a subgroup of $P\Gamma L(2,\textcolor{black}{q})$ containing $PSL(2,\textcolor{black}{q})$ is investigated. In this case, the commutator group $\bar{G}'$ of $\bar{G}$ is isomorphic to $PSL(2,\textcolor{black}{q})$ with $\textcolor{black}{q}\geq 5$.
From Lemma \ref{lemA23agos2015}, $G$ has a (normal) subgroup $M$ of index $2$, and we may assume that its commutator group $M'$ is isomorphic to $SL(2,q)$.  Then $N\geq M'$, and the factor group $\bar{M}=M/N$ is a subgroup of $\bar{G}$. In particular, the commutator subgroup $\bar{M}'$ is contained in $\bar{G}'$.  Since $\bar{M}$ is normal in $\bar{G}$ and $\bar{M}'$ is a characteristic subgroup of $\bar{M}$, we have that $\bar{M}'$ is normal in $\bar{G}$. Therefore, $\bar{M}'$ is a normal subgroup of $\bar{G}'$. As $\textcolor{black}{q}\geq 5$, $\bar{G'}$ is simple, and hence $\bar{M}'=\bar{G}'$.
From \cite[Kapitel I, Hilfssatz 8.4]{huppertI1967},  $\bar{G}'=G'/N$ and $\bar{M}'=M'/N$ whence $G'=M'\cong SL(2,q)$. Finally,  since $[G:M]=2$ together with $N\leq M\leq G$ yield $[\bar{G}:\bar{M}]=2$, the case $\bar{G}\cong {\rm{Alt}}_7$  does not occur.
\end{proof}
\begin{lemma}
\label{lemC28dic2015} If $G$ is an odd core-free group with commutator subgroup $G'\cong SL(2,q)$ then the centralizer of $G'$ in $G$ has order $2$.
\end{lemma}
\begin{proof} Since $G'$ is a normal subgroup of $G$, the centralizer $C_G(G')$ of $G'$ in $G$ is a normal subgroup of $G$. As $G$ is odd core-free, $C_G(G')$ has even order. Assume by contradiction that  $|C_G(G')|>2$. Therefore $|C_G(G')|$ must be divisible by four, otherwise $|C_G(G')|$ is twice an odd number and hence $C_G(G')=Z(G')\times O(C_G(G'))$ where  $O(C_G(G'))$ is a characteristic subgroup of $C_G(G')$ of odd order contradicting the hypothesis that $G$ is odd core-free. Therefore, $C_G(G')$ has a Sylow $2$-subgroup $R$ of order at least $4$.

Take a Sylow $2$-subgroup $V$ of $G'$, and let $U=VR$. Then $|Z(U)|\geq 4$. Since $V$ is a generalized quaternion group, we also have $|Z(V)|=2$. Therefore, $U\gvertneqq V$. From $|Z(U)|\geq 4$, $U$ is neither a generalized quaternion, nor a dihedral nor a semidihedral group. As we pointed out after Result \ref{res26feb2018}, a modular maximal-cyclic group contains no generalized quaternion group.
Therefore $U\gvertneqq V$ cannot actually occur, and the assertion follows by absurd.
\end{proof}

\begin{lemma}
\label{lem29agos2015} If $G$ is an odd core-free group with a non-abelian simple minimal normal subgroup, then one of the following cases occurs, up to isomorphisms,
where $q$ is an odd prime power $d^k$:
\begin{itemize}
\item[\rm(i)] $PSL(2,q)\le G \le P\Gamma L(2,q)$ with $q\ge 5$;
\item[\rm(ii)] $PSL(3,q)\le G \le P\Gamma L(3,q)$ with $q\equiv 3 \pmod 4$;
\item[\rm(iii)] $PSU(3,q)\le G \le P\Gamma U(3,q)$ with $q\equiv 1 \pmod 4$;
\item[\rm(iv)] $G={\rm{Alt}_7}$;
\item[\rm(v)]  $G=M_{11}$.
\end{itemize}
\end{lemma}
\begin{proof} Let $N$ be a non-abelian simple minimal normal subgroup of $G$. We prove that the centralizer $C_G(N)$ of $N$ in $G$ is trivial. Obviously, $C_G(N)\cap N$ is trivial. If $C_G(N)$ has an involution $u$ then $u\not\in N$ and hence the $2$-subgroup of $G$ generated by $u$ and a Sylow $2$-subgroup of $N$ has rank at least $3$, a contradiction. Therefore $C_G(N)$ has odd order. Then $C_G(N)$ is trivial, since $C_G(N)$ is a normal subgroup of $G$, and $G$ is odd core-free. Therefore, $G$ is isomorphic to an automorphism group of $N$. By Lemma \ref{lemA24agos2015} all possibilities for $N$ are determined. The cases $N\cong PSL(2,q)$ with $q\geq 5$, $N\cong PSL(3,q)$ and $N\cong PSU(3,q)$ give (i), (ii) and (iii), respectively.
If $N={\rm{Alt}_7}$ then either $G=N$, or $G\cong{\rm{Sym}_7}$. The latter case cannot actually occur as a Sylow $2$-subgroup of $\rm{Sym}_7$ has $2$-rank $3$. Hence (iv) holds. Finally, $M_{11}$ is isomorphic to its automorphism group, and hence (v) holds.
\end{proof}
\begin{rem} {\em{The subgroups of $P\Gamma L(2,q)$ containing $PSL(2,q)$ whose Sylow $2$-subgroups have $2$-rank $2$ are $PSL(2,q),PGL(2,q)$ and, when $q=d^h$ with an odd prime $d$ and $h\ge 2$, the semidirect product of $PSL(2,q)$ or $PGL(2,q)$ with a cyclic group whose order is an odd divisor of $h$. Analog results are valid for $PSL(3,q)$ and $PSU(3,q)$. These groups are all the candidates for $G$ in (i),(ii),(iii) of Lemma \ref{lem29agos2015}. It is not true that $d$ must be equal to $p$. For instance, the Hermitian curve of equation $X^{q+1}+Y^{q+1}+Z^{q+1}=0$ defined over $\mathbb{F}_q$ has genus $\ha q(q-1)$ and if either $p=7$ or $\sqrt{-7}\in \mathbb{F}_q$ it has an automorphism group isomorphic to $PSL(2,7)$.
}}
\end{rem}
\begin{rem}
\label{remA5jan2016}{\em{The Hermitian curve of equation $X^6+Y^6+1=0$ defined over $\mathbb{F}_5$ has genus $10$ and an automorphism group isomorphic to $\rm{Alt}_7$. Over an algebraically closed field of characteristic $3$, the modular curve $X(11)$ has genus $26$, and  its automorphism group is isomorphic to $M_{11}$; see \cite{AA,R}.}}
\end{rem}

\subsection{The case $|O(G)|>1$}\label{ssec62}

To prove Theorem \ref{structure} for any automorphism groups the following lemma is useful.
\begin{lemma}
\label{le23agosto2015} Let $N$ be any subgroup of $G$ of odd order. If a Sylow $2$-subgroup of $G$ fixes no point of $\cX$ then the quotient curve $\cY=\cX/N$ has also even genus.
\end{lemma}
\begin{proof}
 If the cover $\cX|\cY$ is unramified then the claim follows from the Hurwitz genus formula applied to $N$.
For a point $\bar{P}\in \cY$ where $\cX|\cY$ ramifies, take a point $P\in \cX$ lying over $\bar{P}$. Since $|N|$ is  odd, all ramification groups of $N_P$ at $P$ have odd order. By %\cite[Theorem 11.70]{hirschfeld-korchmaros-torres2008}
(\ref{eq1bis}), $d_P$ is even. Let $\theta$ be the $G$-orbit containing $P$. Then, $|G_P||\theta|=|G|$. Take a Sylow $2$-subgroup $S$ of $G$ containing a Sylow $2$-subgroup $S_P$ of $G_P$. Then $S\neq S_P$, as $S$ does not fix $P$. Therefore, $|S|$ does not divide $|G_P|$ showing that $|\theta|$ must be even. This yields that $d_P|\theta|$ is divisible by four. Since $N$ is a normal subgroup of $G$, the different $D(\cX|\cY)$ is also divisible by four, that is,  $ D(\cX|\cY)/2$ is even.  Therefore, the Hurwitz genus formula applied to $N$ yields $\gg(\cX)-1=|V|(\gg(\bar{\cX})-1)+ (D(\cX|\cY)/2)$ whence  $\gg(\cY)-1$ is odd.
\end{proof}
\begin{rem}
\label{remle23agosto2015}
 {\em{If a Sylow $2$-subgroup of $G$ is not cyclic, then (iii) of Result \ref{res74} shows that the hypothesis of Lemma \ref{le23agosto2015} is satisfied.}}
\end{rem}

\begin{lemma}\label{razioOCNF}
If $|O(G)|>1$ and $G/O(G)$ is a subgroup of $PGL(2,\K)$, then Theorem \ref{structure} holds for $G$.
\end{lemma}
\begin{proof}  By Proposition \ref{pro26feb2018}, $G/O(G)$ may be supposed to be non-solvable. Then Result \ref{resdickson} yields that either $G/O(G)\cong PSL(2,q)$, or $G\cong PGL(2,q)$ with $q\geq 5$.
\end{proof}

We are now in a position to prove Theorem \ref{structure} for the case $|O(G)|>1$.
If $G$ is any subgroup of
 $\aut(\cX)$ then the factor group $G/O(G)$ is an odd core-free automorphism group of the quotient curve $\bar{\cX}=\cX/O(G)$. If a Sylow $2$-subgroup fixes a point of $\cX$ then $S_2$ is cyclic, and hence $G=O(G)\rtimes S_2$. Dismissing this case allows us to use Lemma \ref{le23agosto2015}. Doing so, either $\bar{\cX}$ is rational and $\bar{G}=G/O(G)$ is a subgroup of $PGL(2,\K)$, and Lemma \ref{razioOCNF} applies, or $\gg(\bar{\cX})\geq 2$ and Lemmas \ref{lemA23agos2015}, \ref{lemB23agos2015}, \ref{lem26agos2015}, \ref{lem2mar2018} and \ref{lem29agos2015} apply.

%%%%%%%%%%%%%%%

\subsection{Refinements of Theorem \ref{structure} for large $G$}\label{refinements}

For large automorphism groups compared to the genera of the curves, the list in Theorem \ref{structure} shortens. For the odd core-free case this is shown in the following lemma.
%Lemmas \ref{lem29agos2015} and \ref{lem26agos2015} are obtained.
\begin{lemma}
\label{4dic2015} If $G$ be an odd core-free group of automorphisms of a curve $\cX$ of even genus $\gg(\cX)\geq 2$ such that
\begin{equation}
\label{eq4jan2016}
{\mbox{$|G|>900 \gg(\cX)^2$},}
\end{equation}
%
%\begin{equation}
%\label{eq4jan2016}
%{\mbox{$|G|>900 \gg(\cX)^2$}}
%\end{equation}
then $\cX$ has zero $p$-rank, and one of the following three cases occurs for $G$, up to group isomorphisms
\begin{itemize}
\item[\rm(i)] $PSL(2,q)$, $PGL(2,q)$, $q=p^k\ge 5$;
%\item[\rm(ii)] $PSL(3,q)\le G \le PGL(3,q)$ with $q=p^k$, $k$ odd, $q\equiv 3 \pmod 4$, and $\cX$ has zero $p$-rank;
\item[\rm(ii)] $PSU(3,q)$, $q=p^k$ and $q\equiv 1 \pmod 4$, $PGU(3,q)$, $q=p^k$ and $q\equiv 5 \pmod{12}$;
\item[\rm(iii)] $SL(2,q)$, $SL^{(2)}(2,q)$, $SU^{\pm}(2,q),\,q\equiv 1 \pmod 4$, $SL^{\pm}(2,q),\,q\equiv -1 \pmod 4$,  $q=p^k\geq 5$ .
\end{itemize}
\end{lemma}
\begin{proof} Theorem \ref{structure} with $O(G)=\{1\}$ applies. Obviously, $|G|$ is even, otherwise $G$ is not odd-core free, because it is solvable by Result \ref{resft}. If $G$ is solvable, Theorem \ref{structure} shows that $G$ is a $2$-group apart from
four sporadic cases, namely when $G\cong PSL(2,3),PGL(2,3),GL(2,3)$, SmallGroup($48,28$) whose orders are $12,24$, $48$, and $48$ respectively.
If $G$ is a $2$-group then Lemma  \ref{lem23agos2015} yields that $G$ has a cyclic \textcolor{black}{subgroup} $N$ of index $2$. From (i) of Result \ref{res60.79.108} applied to $N$, $|G|=2|N|\leq 8(\gg(\cX)+1)$ whence $|G|<900 \gg(\cX)^2$. Therefore, $G$ is assumed to be non-solvable.

The sporadic cases $G\cong \rm{Alt}_7$, $G\cong M_{11}$ and $G/Z(G)\cong \rm{Alt}_7$ \textcolor{black}{do} not occur. In fact, for $G\cong \rm{Alt}_7$, (\ref{eq4jan2016}) yields $\gg(\cX)^2<4$ which contradicts $\gg(\cX)\geq 2$.
For $G\cong M_{11}$, we have $|G|=7920$. From (\ref{eq4jan2016}), it follows $\gg(\cX)^2<9$ whence $\gg(\cX)=2$. In this case the hyperelliptic involution $\omega$ does not belong to $G$, since $G$ is simple; hence, $G\langle\omega\rangle/\langle\omega\rangle\cong M_{11}$, which contradicts the last claim in Result \ref{res94}. For $G/Z(G)\cong \rm{Alt}_7$, we have $|G|=5040$, and (\ref{eq4jan2016}) yields $\gg(\cX)=2$ which shows that the quotient curve $\cX/Z(G)$ is rational by Remark \ref{2sub}. But this contradicts Result \ref{resdickson} as no subgroup of $PGL(2,\K)$ is isomorphic to $\rm{Alt}_7$.

Furthermore, the case
$PSL(3,q) \leq G \leq P\Gamma L(3,q)$ is ruled out by Lemma \ref{25dicbis2015} and Remark \ref{rem7nov2016}.

{}Now, from Theorem \ref{structure} with $|O(G)|=1$, either $G'\cong SL(2,q)$, or one of the Cases (iv) or (vi) of Theorem \ref{structure} hold. In the latter two cases, $\cX$ has zero $p$-rank and $q$ is a power of $p$ from Lemma
\ref{27dicbis2015}. Also, by \textcolor{black}{Lemma} \ref{25dicbis2015} and Remark \ref{rem7nov2016}, either (i) or (ii) of Lemma \ref{4dic2015} holds.

Hence $G'=SL(2,q)$ can be assumed. In this case a Sylow $2$-subgroup of $G$ is either a generalized quaternion group or a semidihedral group; in fact, the subgroup $N=Z(G')$ generated by the central involution in $G'=SL(2,q)$ is normal in $G$ since $G'$ is a characteristic subgroup of $G$, and hence Lemma \ref{lemA23agos2015} applies. From Lemmas \ref{lemC28dic2015} and \ref{26dic2015}, $\cX$ has zero $p$-rank.

Also, since $N$ is normal in $G$, the central involution in $G'$ commutes with each element in $G$ and hence $Z(G')\subseteq Z(G)$. Actually, $Z(G')=Z(G)$  by Lemma \ref{lemC28dic2015}. From Lemma \ref{26dic2015} $d=p$, that is, $q$ is a power of $p$.

To prove that Case (iii) must occur, consider the quotient group $\bar{G}=G/Z(G)$ regarded as an automorphism group of the quotient curve $\bar{\cX}=\cX/Z(G)$.
 Since the Sylow $2$-subgroups of $G$ are either semidihedral groups or generalized quaternion groups, those of $\bar{G}$ are dihedral groups.  Furthermore, as $G'\cong SL(2,q)$, we have $G'/Z(G)\cong PSL(2,q)$. Since $G'/Z(G)=G'Z(G)/Z(G)$ and, by \cite[Kapitel I, Hilfssatz 8.4]{huppertI1967}, $\bar{G}'\cong G'Z(G)/Z(G)$, this yields $\bar{G}'\cong PSL(2,q)$.

First we prove that $\bar{G}$ is also odd core-free. By absurd, let $\bar{N}$ be a non-trivial odd order \textcolor{black}{normal} subgroup of $\bar{G}$. The subgroup $N$ of $G$ containing $Z(G)$ such that $N/Z(G)=\bar{N}$ has order  \textcolor{black}{twice} an odd number. Hence, its maximal normal subgroup of odd order is non-trivial. Since such a maximal subgroup is a characteristic subgroup of $N$, it must be a normal subgroup of $G$, a contradiction as $G$ is odd core-free.

Next we prove that $\bar{G}$ is isomorphic to a subgroup of $P\Gamma L(2,q)$ containing $PSL(2,q)$. Since $\bar{G}$ is odd core-free and its Sylow $2$-subgroups are dihedral, Result \ref{resgor} applies to $\bar{G}$ showing that either $\bar{G}\cong {\rm{Alt}}_7$, or, up to an isomorphism, $PSL(2,n)\leq \bar{G}\leq P\Gamma L(2,n)$ with an odd prime power $n\geq 5$. As we have already shown, $\bar{G'}\cong PSL(2,q)$ is the commutator group of $\bar{G}$. Therefore $n=q$, since $PSL(2,n)$ is the commutator subgroup of any subgroup of $P\Gamma L(2,n)$ containing $PSL(2,n)$,
The case $\bar{G}\cong {\rm{Alt}}_7$ does not occur. In fact,  from $2|{\rm{Alt}}_7|=5040> 900 \gg(\cX)^2$, we would have $\gg(\cX)=2$ which contradicts Result \ref{res94}.

To finish the proof, three cases are distinguished according to the value $\gg(\bar{\cX})$.

If $\gg(\bar{\cX})\geq 2$ then the Hurwitz genus formula together with (\ref{eq4jan2016}) yield $|\bar{G}|>900 \gg(\bar{\cX})^2$. From Lemma \ref{25dicbis2015} and Remark \ref{rem7nov2016} applied to $\bar{G}$, either $\bar{G}\cong PSL(2,q)$ or $\bar{G}\cong PGL(2,q)$. In the former case, $|G|=2|\bar{G}|=2|PSL(2,q)|=|SL(2,q)|=|G'|$ whence $G=G'$ follows. This gives the first possibility in (iii).
In the latter case,  $G$ is isomorphic to a central extension of $PGL(2,q)$ by a group of order $2$. If this is a split extension then $G$ contains a subgroup $T\cong PGL(2,q)$ such that $G=Z(G)\times T$.  Then a Sylow $2$-subgroup of $G$ is the direct product of a dihedral group and a cyclic group of order $2$. But such a direct product contains  neither a generalized quaternion group nor a semidihedral group. Therefore,  Result \ref{resgoha} applies, and either the second, or the third, or the fourth case in (iii) must occur.

If $\gg(\bar{\cX})=0$, Result \ref{resdickson} gives the same possibilities for $\bar{G}$, namely $\bar{G}=PSL(2,q)$ or $\bar{G}=PGL(2,q)$, and the same conclusion can be made.

If $\gg(\bar{\cX})=1$ then the Hurwitz genus formula shows that the set $\theta$ of all points fixed by $Z(G)$ has size $2\gg(\cX)-2$. Further,  $\bar{G}$ leaves $\theta$ invariant. For a point $\bar{P}\in \theta$, the stabilizer $\bar{G}_P$ of $P$ has order at most $12$ by Result \ref{res94}.  Therefore, $|\bar{G}|\leq 12
|\theta|$ whence $|\bar{G}|\leq 24 (\gg(\cX)-1)$. Thus $|G|\leq 48(\gg(\cX)-1)$. Comparing with (\ref{eq4jan2016}) shows that this is impossible.
\end{proof}
\begin{rem}
\label{rem5jan2016} {\em{The Hermitian curve of equation $Y^q+Y=X^{q+1}$ with $q\equiv {1 \pmod 4}$ is an example for (ii). If  $q\geq 125$, the Roquette curve $\cX$ of equation $Y^2=X^q-X$ provides an example for (iii) with $G\cong SL(2,q)$ where $q\equiv 1 \pmod 4$.}}
\end{rem}
The general case $|O(G)|\geq 1$ is described in the following proposition.
\begin{proposition}
\label{3apr2016} Let $G$ be a non-solvable group of automorphisms of a curve $\cX$ of even genus $\gg(\cX)\geq 2$ such that (\ref{eq4jan2016}) holds.
Then either $G/O(G)$ has a non-abelian simple minimal normal subgroup, and one of the following two cases occurs, up to group isomorphisms
\begin{itemize}
\item[\rm(A)] $G/O(G)\cong PSL(2,q)$, $G/O(G)\cong PGL(2,q)$, $q=p^k\ge 5$;
\item[\rm(B)] $G/O(G)\cong PSU(3,q)$, $q=p^k$ and $q\equiv 1 \pmod 4$, $G/O(G)\cong PGU(3,q)$, $q=p^k$ and $q\equiv 5 \pmod{12}$;
\end{itemize}
or
\begin{itemize}
\item[\rm(C)]  $G/O(G)\cong SL(2,q), SL^{(2)}(2,q)$, $G/O(G)\cong SU^{\pm}(2,q),\,q\equiv 1 \pmod 4$, $G/O(G)\cong SL^{\pm}(2,q),\,q\equiv -1 \pmod 4$,  where, in all cases, $q=p^k\geq 5$.
\end{itemize}
\end{proposition}
\begin{proof} $G/O(G)$ is an odd core-free automorphism group of the quotient curve $\bar{\cX}=\cX/O(G)$. If $\gg(\bar{\cX})\geq 2$, the assertion follows from Lemma \ref{4dic2015}. Furthermore, from Lemma \ref{le23agosto2015} and Remark \ref{remle23agosto2015}, $\bar{\cX}$ is not elliptic. If $\bar{\cX}$ is rational then $G/O(G)$ is a subgroup of $PGL(2,q)$ with $q=p^k$. Since $G/O(G)$ is supposed to be non-solvable, Result \ref{resdickson} yields that case (A) occurs with only two possible exceptions for $p\neq 5,$ namely
$G/O(G)\cong PSL(2,5)$ and  $G/O(G)\cong PGL(2,5)$. We rule out both exceptions. Both $PSL(2,5)$ and $PGL(2,5)$ contain a maximal, solvable non-abelian subgroup of index $6$. As $O(G)$ is solvable by Result \ref{resft}, $G$ also has a solvable non-abelian subgroup $N$ of index $6$. This together with (\ref{eq4jan2016}) give $|N|>150 \gg(\cX)^2$. As $N$ is solvable, Theorem \ref{princA}  yields that $N$ fixes a point $P\in \cX$. Since $G$ does not fix $P$, and $N$ is a maximal subgroup of $G$, the $G$-orbit $o$ of $P$ has length $6$. By Theorem \ref{princA}, $\cX$ has zero $p$-rank. If $N$ has some elements of order $p$, then a $p$-subgroup of $N$ fixes $P$ and acts transitively on the remaining $5$ points in $o$. But this yields $p=5$, a contradiction. Otherwise, $N$ is a prime to $p$ subgroup, and the classical Hurwitz bound yields $|N|<84(\gg-1)$ whence $|G|<504 (\gg-1)$ contradicting (\ref{eq4jan2016}).
\end{proof}
Finally, a useful bound on the size of $O(G)$ is given in the lemma below.
\begin{lemma}
\label{lem10Emar2018}  Let $G$ be a non-solvable group of automorphisms of a curve $\cX$ of even genus $\gg\geq 2$ such that (\ref{eq4jan2016}) holds. Let $\bar{\cX}=\cX/O(G)$. If $\gg(\bar{\cX)}>0$ then $900|O(G)|<|G/O(G)|$.
\end{lemma}
\begin{proof} From Lemma \ref{le23agosto2015} and Remark \ref{remle23agosto2015},  $\gg(\bar{\cX})\geq 2$. The Hurwitz genus formula applied to $O(G)$ gives $\gg(\cX)-1\geq |O(G)|$.
This together with (\ref{eq4jan2016}) show  $|G/O(G)||O(G\textcolor{black}{)}|>900 \gg(\cX)^2>900|O(G)|^2$ whence the claim follows.
\end{proof}

\section{Large automorphism groups of curves with even genus}\label{sectc}
An important step toward our second main result is the following theorem.
\begin{theorem}\label{princB}
Let $G$ be a non-solvable automorphism group of an algebraic curve $\cX$ of even  genus $\gg(\cX) \geq 2$ such that
$$
|G|>900 \gg(\cX)^2.
$$
Then $\cX$ has zero $p$-rank. Furthermore, $O(G)$ is a cyclic  prime to $p$ subgroup, and either $O(G)=Z(G)$ or $[Z(G):O(G)]=2$ according as Cases (A),(B), or (C) in Proposition \ref{3apr2016} holds.
\end{theorem}

The proof of Theorem \ref{princB} requires a careful analysis of the Cases (A),(B),(C) in Proposition \ref{3apr2016}. The basic idea is to find an appropriate upper bound on $|O(G)|$ depending on $q$ so that the linear bound given in Lemma \ref{22dic2015} may be applied to show that $\cX$ has zero $p$-rank. Once this is done, the relationship between $O(G)$ and $Z(G)$ can thoroughly been worked out.
%Regrading cases (A) and (B) we give a detailed proof for $G/O(G)\cong PSL(2,q)$ and $G/O(G)\cong PSU(3,q)$, as our proof can easily been adapted for the other two groups $PGL(2,q)$ and $PGU(3,q)$ up to a change of constants.
\subsection{Proof of Theorem \ref{princB} in Case (A) of Proposition \ref{3apr2016}}
\label{sspsl}
Let $L=G'O(G)$ where $G'$ is the commutator subgroup of $G$.
%%%modifica 11 aprile 2018
Since $G'O(G)/O(G)\cong (G/O(G))'$ by \cite[Kapitel I, Hilfssatz 8.4]{huppertI1967}, we have that $\bar{L}=L/O(G)$ is the commutator subgroup $\bar{G}'$ of $\bar{G}$.
%%fine modifica
In particular $\bar{L}=\bar{G}'\cong PSL(2,q)$, and either $\bar{G}'=\bar{G}$, or $\bar{G}'$ is the unique normal subgroup of $\bar{G}$ and $\bar{G}\cong PGL(2,q)$. Furthermore, either $L=G$ or the index of $L$ in $G$ is equal to $2$.
%We assume that $G/O(G)\cong PSL(2,q)$ with an odd power $q\geq 5$ of $p$, as the other possibility $G/O(G)\cong PGL(2,q)$ can be treated analogously.
%From (\ref{eq4jan2016}) and $\gg\geq 2$, we have $q>13$.
\begin{lemma}
\label{lem1mar2018} If Case (A) in Proposition \ref{3apr2016} holds then
\begin{equation}
\label{eq8jan2016}
\frac{(q-1)q}{|O(G)|(q+1)}>\frac{25}{4}.
\end{equation}
In particular,
\begin{equation}
\label{eqA8jan2016} |O(G)|<\frac{4}{25}q
\end{equation}
and $q>11$.
\end{lemma}
\begin{proof}
Consider $\bar{L}=L/O(G)$ as an automorphism group of the quotient curve $\bar{\cX}=\cX/O(G)$.  Since $\bar{L}\cong PSL(2,q)$ its Sylow $2$-subgroups are dihedral and hence not cyclic. Result \ref{res74} shows that $L$ fixes no point of $\cX$. From Lemma \ref{le23agosto2015}, $\gg(\bar{\cX})$ is even. Therefore, either $\bar{\cX}$ is rational, or $\gg(\bar{\cX})\geq 2$.
In the latter case, since $\bar{L}$ contains an abelian subgroup of odd order $q$, see Result \ref{resdickson}, we have $4\gg(\bar{\cX})+3\geq q$ by (i) of Result \ref{res60.79.108}. Thus
$4\gg(\bar{\cX})+4\geq q+1$.

If $\bar{\cX}$ is not rational, the Hurwitz genus formula applied to $O(G)$ gives $\gg(\cX)-1\geq |O(G)|(\gg(\bar{\cX})-1)$. From (\ref{eq4jan2016}),
$$(q-1)q>900 \frac{\gg(\cX)^2}{|O(G)|(q+1)}.$$
Since $\gg(\cX)>\gg(\cX)-1$,
$$
\frac{(q-1)q}{|O(G)|(q+1)}>900 \left(\frac{\gg(\bar{\cX})-1}{q+1}\right)^2>900 \left(\frac{\gg(\bar{\cX})-1}{4(\gg(\bar{\cX})+1)}\right)^2>\frac{900}{144}=\frac{25}{4},
$$
whence (\ref{eq8jan2016}) follows.  A direct computation shows that this yields (\ref{eqA8jan2016}) and $q>11$.

If $\bar{\cX}$ is rational, then $\bar{L}$ acts on $\bar{\cX}$ as $PSL(2,q)$ on the projective line $PG(1,\K)$. Take any point $\bar{P}\in \bar{\cX}$. As $PSL(2,q)$ is non-solvable, $\bar{L}$ does not fix $\bar{P}$ by (iii) of Result \ref{res74}. Hence, the $\bar{L}$-orbit of $\bar{P}$ has length at least $q$. In fact, $|\bar{L}_{\bar{P}}|\le \ha q(q-1)$ as the order of the largest proper subgroup of $PSL(2,q)$ is $\ha q(q-1)$ by Result \ref{resdickson}.  Now take a point $P\in \cX$ where the cover $\cX|\bar{\cX}$ ramifies. Then the $G$-orbit $o$ of $P$ has length at least $qu$ where
$u=|O(G)|/|O(G)_P|$. The Hurwitz genus formula applied to $O(G)$ yields
$$2\gg-2\geq -2|O(G)|+|o|(|O(G)_P|-1)\geq -2|O(G)|+|o|\textstyle\frac{2}{3}|O(G)_P|\geq |O(G)|(\textstyle\frac{2}{3}q-2).$$
Hence  $\gg(\cX)>\thi(q-3)|O(G)|$. From this and $(q+1)q(q-1)|O(G)|>900\gg^2$,
$$|O(G)|<\frac{9}{900}\frac{(q+1)(q-1)}{(q-3)^2}q<
%\frac{\cdot 6\cdot 4}{4\cdot 100} q=
\frac{3}{50}q$$
Therefore, both (\ref{eq8jan2016}) and (\ref{eqA8jan2016}) hold for any $\bar{\cX}$. Also, $q>11$ must hold. \end{proof}
\begin{lemma}
\label{lemA1mar2018} If Case (A) in Proposition \ref{3apr2016} holds then the quotient curve $\bar{\cX}=\cX/O(G)$ has zero $p$-rank.
\end{lemma}
\begin{proof}
We may assume that  $\bar{\cX}$ is not rational. As we have already observed in the proof of Lemma \ref{lem1mar2018},
 $\gg(\bar{\cX})$ is even by Lemma \ref{le23agosto2015}.

Since $\bar{L}$ contains a subgroup $\bar Q$ of odd order $q\ge 13$, see Result \ref{resdickson}, if $\bar X$ has positive $p$-rank then  $\gg(\bar{\cX})-1 \ge \frac{p-2}{p}q\ge \frac{13}{3}$ by  Result \ref{resnaga} applied to $\bar Q$. Thus  we can assume  $\gg(\bar{\cX})\ge 6$
 and hence
\begin{equation}
\label{15may2016}
\gg(\cX)> \gg(\cX)-1\geq |O(G)|(\gg(\bar{\cX})-1)\ge \textstyle\frac{5}{6}|O(G)|\gg(\bar{\cX}).
\end{equation} As $|\bar L|\ge \frac{|G|}{2|O(G)|}$ holds,
taking into account (\ref{eq4jan2016}) we then have
$$
|\bar L|\ge \frac{450 \gg(\cX)^2}{|O(G)|} >450 \cdot \frac{25}{36}|O(G)|\gg(\bar{\cX})^2>900 \gg(\bar{\cX})^2.
$$
%Since, by (\ref{eq4jan2016}),
%$|G|\leq |L|=\ha\,q(q-1)(q+1)|O(G)|>900\gg(\cX)^2$, (\ref{15may2016})  yields
%$$|\bar{L}|=\ha\,q(q-1)(q+1)>900\gg(\bar{\cX})^2.$$
{}From Lemma \ref{4dic2015} applied to $\bar{L}$,  the $p$-rank of $\bar{\cX}$ is equal to zero.
\end{proof}
\begin{lemma}
\label{lemB1mar2018} If Case (A) in Proposition \ref{3apr2016} holds then $\cX$ has zero $p$-rank.
\end{lemma}
\begin{proof}
{}Lemma \ref{lemA1mar2018} together with Result \ref{lem29dic2015} applied to $\bar{\cX}$ show that a Sylow $p$-subgroup $\bar{Q}$ of $\bar{L}$, as well as its normalizer in $\bar{L}$, fixes a unique point $\bar{P}\in \bar{\cX}$. Here $|\bar Q|=q$, and the normalizer of $\bar{Q}$ in $\bar{L}$ is the semidirect product $\bar{Q}\rtimes \bar{T}$ with a cyclic group $\bar{T}$ of order $\ha (q-1)$ by Result \ref{resdickson}. Since this normalizer is a maximal subgroup of $\bar{G}$, see again Result \ref{resdickson}, it is the stabilizer of $\bar{P}$ in $\bar{L}$.
Let $o$ be the $O(G)$-orbit lying over $\bar{P}$ in the cover $\cX|\bar{\cX}$. The counter-image of $\bar{Q}\rtimes \bar{T}$ in the homomorphism $L\to \bar{L}$ is a subgroup $N$ of order $\ha(q-1)q|O(G)|$ where $N$  is the subgroup of $G$ which preserves $o$. For a point $P\in o$, the stabilizer $H$ of $P$ in $N$ has order $\ha(q-1)q|O(G)_P|$. Therefore, the (unique) Sylow $p$-subgroup $Q_1$ of $H$ has order $q_1=qp^a$ and a complement of $Q_1$ in $H$ is a cyclic subgroup $C_t$ of order $t=\ha(q-1)b$ with $p\nmid b$ by (iii) of Result \ref{res74}. Hence $H=Q_1\rtimes C_t$, and $|o|p^ab=|O(G)|$. Now, since by \eqref{eq4jan2016} $q(q-1)(q+1)|O(G)|>900\gg(\cX)^2$,
%$q(q-1)p^ab(q+1)|o|>900\gg^2$, (\ref{eq8jan2016}) gives
$$|H|^2>(p^ab)^2\frac{\qa(q-1)q}{|O(G)|(q+1)}900\gg(\cX)^2.$$
This together with \eqref{eq8jan2016} give $|H|^2>900 \gg(\cX)^2$ whence
$$|H|=q_1t> 30 \gg(\cX).$$
From Lemma \ref{22dic2015} and Remark \ref{rem3jan2016}, either $\cX$ has zero $p$-rank, or $Q_1$ has just one short orbit $\theta$ other than its fixed point $P$. Observe that $\theta$ is contained in $o$. In fact, each $Q_1$-orbit in $o$ is short as $|o|\leq |O(G)|<q<q_1$ by (\ref{eqA8jan2016}). %Since $\bar{G}$ From \cite[Chapter II, Satz8.28]{huppert1967},
Therefore $o=\theta\cup\{P\}$, and hence $|o|=1+p^h$ for some $h\ge 0$. Since $|o|$ divides $|O(G)|$ this implies that  $|O(G)|$ is even, a contradiction. Thus $\cX$ has zero $p$-rank.
\end{proof}
To complete our investigation of Case (A) in Proposition \ref{3apr2016}, we only need the following result.
\begin{lemma}
\label{lemC1mar2018} If Case (A) in Proposition \ref{3apr2016} holds then $O(G)$ is a prime to $p$ cyclic group, and $Z(G)=O(G)$.
\end{lemma}
\begin{proof}
For every $h\in G$ the map $\varphi_h$ taking $u\in O(G)$ to $h^{-1}uh$ is an automorphism of $O(G)$, and hence the map $\Phi:\,h\mapsto \varphi_h$ is a homomorphism from $G$ into $\aut(O(G))$.  Take a Sylow $p$-subgroup $Q$ of $G$. From Lemma \ref{lemB1mar2018} and Result \ref{lem29dic2015}, $Q$ fixes a unique point $P\in \cX$ but no non-trivial element of $Q$ fixes a point other than $P$. If $g$ is any element in $O(G)$, then (\ref{eqA8jan2016}) yields that the $Q$-orbit of $g$ in the action of $\Phi$ has length smaller than $|Q|$. Hence, $g$ commutes with a non-trivial element of $Q$. Therefore, $g$ fixes $P$, and hence
$O(G)$ does, as well. Repeating this argument with another Sylow $p$-subgroup of $G$, it turns out that $O(G)$ fixes a point $P'\in \cX$ other than $P$. From (i) of Result \ref{lem29dic2015}, $O(G)$ contains no element of order $p$. Therefore, from (iii) of Result \ref{res74}, $O(G)$ is a cyclic prime to $p$ group. In particular, $\ker(\Phi)\ge O(G)$. Since $\bar{G}$ is either simple, or $\bar{L}\cong PSL(2,q)$ is the only non-trivial normal subgroup of $\bar{G}$, two cases arise namely $\ker(\Phi)=O(G)$ and $\ker(\Phi)\ge L$.

In the former case, $\Phi$ induces a faithful action of $\bar{L}$ on $O(G)$. If $\Lambda$ is a non-trivial orbit then $|\Lambda|<q$ by (\ref{eqA8jan2016}). On the other hand, since $q>11$, Result \ref{resdickson} yields $|\Lambda|\geq q+1$,  a contradiction.

In the latter case, $\ker(\Phi)\geq L$ and hence $O(G)\subseteq Z(L)$; actually, $O(G)=Z(L)$ since $Z(\bar{L})$ is trivial. As we have shown, $O(G)$ fixes a point $P\in \cX$. Since $\gamma(\cX)=0$, (ii) of Result \ref{lem29dic2015} shows that $G_P$ coincides with the normalizer $N_G(Q)$ of $Q$ in $G$. From Result \ref{reszass}, a complement $C$ of $Q$ contains $O(G)$. If $L\lneqq G$ then $C$ contains an element of $G\setminus L$. Therefore, some element in $G\setminus L$ centralizes $O(G)$. This shows that $Z(G)=O(G)$, since $Z(\bar{G})$ is trivial.
\end{proof}
\subsection{Proof of Theorem \ref{princB} in Case (B) in Proposition \ref{3apr2016}}
The quotient curve $\bar{\cX}=\cX/O(G)$ is not rational as $\bar{G}=G/O(G)$ is an automorphism group of the quotient curve $\bar{\cX}=\cX/O(G)$ but  $PGL(2,\K)$ has no subgroup isomorphic to $\bar{G}\cong PSU(3,q)$ by Result \ref{resdickson}. Also, $\bar{\cX}$ is neither elliptic by Lemma \ref{le23agosto2015}.
Then $\gg(\bar{\cX})>2$ by  Result \ref{res94}.

In particular, $\gg(\bar{\cX})>2$.
From $|O(G)|\geq 3$ and $\gg(\cX)>\gg(\cX)-1\geq |O(G)|(\gg(\bar{\cX})-1)$, we have by \eqref{eq4jan2016}
%\begin{equation}
%{\label{21jan2016}
%(q^2-1)q^3(q^3+1)\ge
$$|\bar G|=\frac{|G|}{|O(G)}>\frac{900}{|O(G)|}\gg(\cX)^2> 900|O(G)|(\gg(\bar{\cX})-1)^2\ge  900 \frac{|O(G)|}{3}\gg(\bar{\cX})^2>900\gg(\bar{\cX})^2$$
%\end{equation}
which together with Lemma \ref{27dicbis2015} show that $\bar{\cX}$ has zero $p$-rank.

Let $L=G'O(G)$ where $G'$ is the commutator subgroup of $G$. By \cite[Kapitel I, Hilfssatz 8.4]{huppertI1967}, $\bar{L}=L/O(G)$ is the commutator subgroup $\bar{G}'$ of $\bar{G}$.
In particular $\bar{L}=\bar{G}'\cong PSU(3,q)$, and either $\bar{G}'=\bar{G}$, or $\bar{G}'$ is the unique proper normal subgroup of $\bar{G}$ and $\bar{G}\cong PGU(3,q)$.
First we show how the arguments used in Section \ref{sspsl} can be adapted as far as
\begin{equation}
\label{eq18jan2016A} \frac{(q^2-1)q^3}{|O(G)|(q^3+1)}>\mu^2
\end{equation}
where, as in Result \ref{resmitchell}, $\mu={\rm{gcd}}(3,q+1)$.
\begin{lemma}
\label{3Cmar2018} If (\ref{eq18jan2016A}) is assumed then Theorem \ref{princB} holds.
\end{lemma}
\begin{proof}
Since the $p$-rank of $\bar{\cX}$ is equal to zero, (i) of Result \ref{lem29dic2015} shows that a Sylow $p$-subgroup $\bar{Q}$ of $\bar{L}$ fixes a unique point $\bar{P}\in\bar{\cX}$.
The normalizer of $\bar{Q}$ in $\bar{L}$ is the semidirect product $\bar{Q}\rtimes \bar{T}$ where $\bar{Q}$ is a non-abelian $p$-group of order $q^3$ while $\bar{T}$ is a cyclic group of order $(q^2-1)/\mu$. From Result \ref{resmitchell}, this normalizer is a maximal subgroup of $\bar{L}$. Hence it coincides with  the stabilizer of $\bar{P}$ in $\bar{L}$. Let $o$ be the $O(G)$-orbit lying over $\bar{P}$ in the cover $\cX|\bar{\cX}$. The counter-image of $\bar{Q}\rtimes \bar{T}$ in the homomorphism $L\to \bar{L}$ is a subgroup $N$ of order $(q^2-1)q^3|O(G)|/\mu$. For a point $P\in o$, the stabilizer of $P$ in $N$ is a subgroup $H$ of order $(q^2-1)q^3|O(G)_P|/\mu.$ Therefore, the (unique) Sylow $p$-subgroup $Q_1$ of $H$ has order $q_1=q^3p^a$ and a complement of $Q_1$ in $H$ is a cyclic subgroup $C_t$ of order
$t=b(q^2-1)/\mu$ with $p\nmid b$. Hence $H=Q_1\rtimes C_t$, and $|o|p^ab=|O(L)|$. From (\ref{eq4jan2016}),
$$(q^2-1)q^3>\frac{900\gg(\cX)^2}{(q^3+1)|O(G)|}.$$
Therefore,
$$|H|^2=(q_1t)^2>(p^ab)^2\frac{(q^2-1)q^3}{\mu^2|O(G)|(q^3+1)}900\gg(\cX)^2$$
This together with  (\ref{eq18jan2016A}) give
$$|H|>30 \gg(\cX).$$
From Lemma \ref{22dic2015} and Remark \ref{rem3jan2016}, either $\cX$ has zero $p$-rank, or $Q_1$ has just one short orbit $\theta$ other than its fixed point $P$. Observe that $\theta$ is contained in $o$. In fact, each $Q_1$-orbit in $o$ is short as $|o|\leq |O(G)|<(q^2-1)/\mu^2<q^3<q_1$ by (\ref{eq18jan2016A}). Therefore $o=\theta\cup\{P\}$, and hence $|o|=1+p^h$ for some $h\ge 0$. Since $|o|$ divides $|O(G)|$ this implies that  $|O(G)|$ is even, a contradiction. Thus $\cX$ has zero $p$-rank.

%%%modifica 11 aprile 2018
 Finally, $O(G)=Z(G)$ and $O(G)$ is a cyclic prime to $p$  group. We point out that this can be shown by the arguments used in the final part of the proof of Lemma \ref{lemC1mar2018}. It is enough to deal with the case  where $\Phi$ induces a faithful action of $\bar{L}\cong PSU(3,q)$ on $O(G)$. If $\Lambda$ is a non-trivial orbit then $|\Lambda|<|O(G)|< (q^2-1)/\mu^2$ by (\ref{eq18jan2016A}). On the other hand, $|\Lambda|$ is at least as large as the index of the largest subgroup of $\PSU(3,q)$, and, from Result \ref{resmitchell}, $|\Lambda|\geq q^3+1$ for $q\neq 5$ while $|\Lambda|\geq 60$ for $q=5$. In both cases, this gives a contradiction.
%%modifica
\end{proof}
Next we show that (\ref{eq18jan2016A}) is not heavily restrictive. Throughout this section, let  $\bar{\cX}=\cX/O(G)$.
\subsubsection{The case of ramified cover $\cX|\bar{\cX}$}
\begin{lemma}
\label{lemD1mar2018}
If the cover $\cX|\bar{\cX}$ ramifies then  (\ref{eq18jan2016A}) holds.
\end{lemma}
\begin{proof}
Take a ramification point $\bar{P}\in \bar{\cX}$. As $PSU(3,q)$ is non-solvable, $\bar{G}$ does not fix $\bar{P}$ by (iii) of Result \ref{res74}.
Let $\bar{o}$ be the $\bar{G}$-orbit of $\bar{P}$. Note that $q\neq 3$ by
$q\equiv 1  \pmod 4$. From Result \ref{resmitchell}, either $|\bar{o}|\geq q^3+1$, or $|\bar{o}|\ge 60$ and $q=5$. Assume $q\neq 5$. Then the cover $\cX|\bar{\cX}$ has at least $q^3+1$ ramification points with the same ramification index. If the $O(G)$-orbit lying over $\bar{P}$ has length $\ell$ then $|O(G)|-\ell\geq \frac{2}{3}|O(G)|$. Recall that $g(\bar{\cX})>2$ . Therefore,
%since $q\neq 3$ by
%$q\equiv 1  \pmod 4$,
if $q\neq 5$ then
\begin{equation}
\label{eq8mar2018}
2\gg(\cX)-2> (q^3+1)(|O(G)|-\ell)\geq |O(G)|\textstyle\frac{2}{3}(q^3+1),
\end{equation}
 whence
$$\gg(\cX)>\gg(\cX)-1>\textstyle\frac{4}{9}(q^3+1)|O(G)|.$$ This together with (\ref{eq4jan2016})  give
$$
900((q^3+1)|O(G)| \textstyle\frac{4}{9})^2<900\gg(\cX)^2<|O(G)|(q^3+1)(q^2-1)q^3
$$
and hence (\ref{eq18jan2016A}) holds.
Similarly, if $q=5$ we have
$$
\gg(\cX)>\gg(\cX)-1>20 |O(G)|,
$$
and
$$
900(20|O(G)|)^2<900\gg(\cX)^2<|O(G)|\cdot |PGU(3,5)|
$$
gives (\ref{eq18jan2016A}) for $q=5$.
%Finally, we show that case $q=5$ cannot actually occur. For $q=5$, (\ref{eq4jan2016}) reads $|O(G)|>420\gg(\cX)^2$. From %Theorem \ref{princA}, $\cX$ has zero $p$-rank and $O(G)$ fixes a point $P\in \cX$. If $p$ divides $|O(G)|$ then Result %\ref{res74} yields that $P$ is the unique fixed point of $O(G)$. Hence $G$ also fixes $P$, a contradiction with (iii) of Result %\ref{lem29dic2015}. Otherwise $O(G)$ is tame, and
%the Hurwitz bound $84(\gg(\cX)-1)$ contradicts $|O(G)|>420\gg(\cX)^2$.
\end{proof}
By Lemmas \ref{3Cmar2018} and \ref{lemD1mar2018}, Theorem \ref{princB} holds for ramified covers $\cX|\bar{\cX}$. We will show below that the unramified case cannot actually occur. Unfortunately, this case appears to be more involved and our proof requires deeper results from Group theory.
\subsubsection{The case of unramified cover $\cX|\bar{\cX}$}

It has alredy been noticed that $\bar{\cX}$ has zero $p$-rank and that $\gg(\bar{\cX})>2$. Also, by Lemma \ref{le23agosto2015} the genus of $\bar{\cX}$ is even.

\begin{lemma}\label{lemmacarsemp}
 If  the cover $\cX|\bar{\cX}$  is unramified then
$O(G)$ is not charactersitically simple.
 \end{lemma}
\begin{proof}
Assume on the contrary that $O(G)$ is characteristically simple. From Result \ref{resthompson} $O(G)$ is solvable. Hence $O(G)$ is the direct product of $r\geq 1$ cyclic groups of odd prime order $v$; see \cite[Kapitel I, Satz 9.12 ]{huppertI1967}.
We show that $v\neq p$.  Let $\bar{Q}$ be a Sylow $p$-subgroup of $\bar{G}$.
From (i) of Result \ref{lem29dic2015}, there is a point $\bar{P}\in \bar{\cX}$ which is fixed by $\bar{Q}$.
Assume on the contrary that $v=p$. Then $|O(G)|=p^r$. Let $Q$ be a Sylow $p$-subgroup of $G$ containing $O(G)$ such that $\bar{Q}=Q/O(G)$.
Let $o$ be the $O(G)$-orbit lying over $\bar{P}$ in the cover $\cX|\bar{\cX}$. Since $|o|=p^r$,
for each $P\in O$ the stabilizer of $P$ in $Q$ has order $q^3=|\bar{Q}|$. On the other hand, no non-trivial element of $Q$ fixes a point in $\mathcal X\setminus o$. In fact, if $h\in Q$ fixes a point $P'$ in $\mathcal X\setminus o$, then $h\notin O(G)$ since $\cX|\bar{\cX}$ is unramified. Let $\bar{h}$ be the image of $h$ under the homomorphism $G\to \bar{G}$. Then $\bar{h}$ is not the trivial automorphism of $\bar{\cX}$, and fixes two distinct points, namely $\bar{P}$ and the point $\bar{P}'$ lying under $P'$ in the cover $\cX|\bar{\cX}$. Since $\bar{h}$ is a $p$-element, the $p$-rank of $\bar{\cX}$ is not zero by (i) of Result \ref{lem29dic2015}, a contradiction.
Note that $\bar{\mathcal X}=\cX/O(G)$ is an unramified $p$-cover of $\bar{\cX}/Q$, and hence the $p$-rank of $\mathcal X/Q$ is zero.
 From the Deuring-Shafarevich formula applied to $Q$,
$$\gamma(\cX)-1=-|Q|+|O(G)|(q^3-1)=-p^rq^3+p^r(q^3-1)=-p^r,$$
a contradiction. Thus, $v\neq p$.

Let $\Phi$ denote the homomorphism from $G$ to the automorphism group of $O(G)$ arising from conjugacy. Since $O(G)$ is abelian, $\ker(\Phi)$ contains $O(G)$.
Two cases are distinguished according as $\ker(\Phi)=O(G)$ or $\ker(\Phi)$ contains $O(G)$ properly.

In the former case, $\Phi$ gives rise to a faithful representation of $\bar{G}$ by a matrix group of the vector space $V=\mathbb F_v^r$. As we have shown at the beginning of the proof, no non-trivial normal subgroup of $G$ is contained in $O(G)$. Therefore, $\Phi$ is irreducible. Let $PG(r-1,v)$ be the projective space arising from $V$.
%Let $\bar{T}$ be the subgroup of $\bar{G}$ of index $\leq 3$ which is isomorphic to $PSU(3,q)$.
Since $\bar{L}\cong \PSU(3,q)$ is simple, $\bar{L}$ also has an irreducible, faithful representation as a projective group of $\PG(r-1,v)$.  Since $v\neq p$, this yields $r-1 \ge q(q-1)$ by Result \ref{reslse}. Hence $|O(G)|\ge v^{q(q-1)+1}$. But this contradicts Lemma \ref{lem10Emar2018}.

In the latter case, $\ker(\Phi)/O(G)$ is a proper normal subgroup of $\bar{G}$. Since $\bar{G}\cong PSU(3,q)$, or
$\bar{G}\cong PGU(3,q)$, this yields that $\ker(\Phi)$ contains $L$, that is,
%
%
 % $\ker(\Phi)/O(G)\geq \bar{G}'\cong PSU(3,q)$. By \cite[Hilfssatz 8.4, Chapter I]{huppertI1967}, $\ker(\Phi)\geq G'$, that is,
 $O(G)\subseteq Z(L)$; actually $O(G)=Z(L)$ holds since the center of $\bar{L}$ is trivial. Two cases arise according as $O(G)\leq G''$ or $O(G) \nleqq G''$.

If $G''$ contains $O(G)$, then $Z(G')=O(G)$ and $G'/O(G)=\bar{L}\cong PSU(3,q)$. Therefore, $G'$ is a perfect central extension of $PSU(3,q)$. From Result \ref{resgriess}, $G'=SU(3,q)$. In particular,  $|Z(G')|=3$ whence $|O(G)|=3$. This shows that (\ref{eq18jan2016A}) holds. From Lemma \ref{3Cmar2018}, $\cX$ has zero $p$-rank.
As  any $p$-subgroup of $G$ is contained in $G'$, we have that  a Sylow p-subgrop of $G$, say $S_p$ is contained in $G'$.
Since
$\cX$ has zero $p$-rank, $S_p$ fixes a point $P$ in $\cX$ and no non-trivial element of $S_p$ fixes another point; see Result \ref{lem29dic2015}.
Each $h\in O(G)$ commutes with every element in $S_p$, and hence fixes $P$, which contradicts the hypothesis that the cover $\cX|\bar{\cX}$ is unramified.

If $G''$ does not contain $O(G)$ then they intersect trivially as $O(G)$ contains no proper normal subgroup of $G$. Therefore, $G''\times O(G)$ is a normal subgroup of $G$. Since $PSU(3,q)\le G/O(G)\le PGU(3,q)$ up to an isomorphism, this implies $G''\cong PSU(3,q)$. In particular, any $p$-subgroup of $G$ is contained in $G''$.
Take a point $\bar{P}\in \bar{\cX}$. Since $\bar{\cX}$ has zero $p$-rank, (i) of Result \ref{lem29dic2015} yields that $\bar{P}$ is fixed by a Sylow $p$-subgroup $\bar{Q}$ of $\bar{G}$. The $O(G)$-orbit $o$ lying over $\bar{P}$ is left invariant by a subgroup $Q$ of $G$ such that $Q/O(G)=\bar{Q}$. Since $|O(G)|$ is prime to $|\bar{Q}|=q^3$, Result \ref{reszass} shows that $Q= O(G) \rtimes R$ where $|R|=|\bar{Q}|$ with $R\le G''$, and that $R$ fixes a point $P\in o$. Since $O(G)$ is contained in the centralizer of $G'$ in $G$, this yields that $R$ fixes $o$ pointwise. From the Deuring-Shafarevich formula applied to $R$,
\begin{equation}
\label{eq11mar2018}
\gamma(\cX)\geq-q^3+|O(G)|(q^3-1)+1>\thi (q^3+1)|O(G)|.
\end{equation}
This together with $\gg(\cX)\ge \gamma(\cX)$ yields
 (\ref{eq18jan2016A}). From Lemma \ref{3Cmar2018}, $\gamma(\cX)=0$, but this contradicts (\ref{eq11mar2018}).

\end{proof}

%2) IL CASO GENEREALE NON SI PRESENTA, PROCEDENDO PER INDUZIONE

%AGGIUSTARE QUI
%\begin{lemma}
%\label{lem10mar2018}  If  the cover $\cX|\bar{\cX}$ with $\bar{\cX}=\cX/%O(G)$ is unramified then $\cX$ has zero $p$-rank.
%\end{lemma}
%\begin{proof}
%By way of contradiction, $\cX$ is assumed to be a minimal counter-example with respect to its genus, that is, any curve $\cY$ with $2\le \gg(Y)<\gg(\cX)$ satisfying (\ref{eq4jan2016}) such that
\begin{lemma}
\label{lem10Fmar2018}  The cover $\cX|\bar{\cX}$ is ramified.
% then Proposition \ref{princB} holds with $G'\cong SU(3,q)$ and $Z(G')=O(G)$.
\end{lemma}
\begin{proof}
Assume that  $\cX$ is a counter-example with minimal genus.
Since by Lemma \ref{lemmacarsemp} $O(G)$ is not characteristically simple,
 $G$ has a normal subgroup $N$ properly contained in $O(G)$. Consider the quotient curve $\tilde{\cX}=\cX/N$. From $N< O(G)$, we have $4\le \gg(\bar{\cX})<\gg(\tilde{\cX})$.
%By Lemma \ref{le23agosto2015} applied to $N$, this yields $\gg(\tilde{\cX})%\geq 4$.
From the Hurwitz genus formula applied to $N$, $$\gg(\cX)>\gg(\cX)-1=|N|(\gg(\tilde{\cX})-1)\ge \textstyle\frac{4}{5}|N|\gg(\tilde{\cX}).$$
 Since $\tilde{G}=G/N$ is a subgroup of $\aut(\tilde{\cX})$, and $|N|\geq 3$,
$$
|{\tilde{G}}|> \frac{900\cdot 16}{25} |N^2|\gg({\tilde{\cX}})^2\ge \frac{146}{25}\cdot 900 \gg({\tilde{\cX}})^2. $$
%$$|\tilde{G}|\ge (q^2-1)q^3(q^3+1)\frac{|O(G)|}{|N|}>900\frac{\gg^2}{|N|}>900 \textstyle\frac{4}{9}|N|\gg(\tilde{\cX})^2>900 %\gg(\tilde{\cX})^2.$$
Note that $\tilde{G}$ is a non-solvable group and that the quotient group $O(G)/N$ is the largest normal subgroup $O(\tilde{G})$ of $\tilde{G}$ of odd order. Then $\bar{\cX}=\tilde{\cX}/O(\tilde{G})$ and the cover $\tilde{\cX}|\bar{\cX}$ is unramified. Therefore, $\tilde{\cX}$ with respect to $\tilde{G}$ satisfies  Condition (B) in Proposition \ref{3apr2016}.
As $\cX$ is assumed to be a counter-example with minimal genus, we obtain a contradiction.
\end{proof}

\subsection{Proof of Theorem \ref{princB} in Case (C) in Proposition \ref{3apr2016}}
We use the idea worked out in Section \ref{sspsl}. Note that in this case
%Since $\bar{G}=G/O(G)$ is isomorphic to a subgroup of $GL(2,q)$ containing $SL(2,q)$,
the commutator subgroup $\bar{G}'$ contains $SL(2,q)$. Let $L$ be the subgroup of $G$ such that $\bar{L}=SL(2,q)$.
% where $G'$ is the commutator subgroup of $G$. By \cite[Hilfssatz 8.4, Chapter I]{huppertI1967},  $\bar{G}'=\bar{L}$ with $%\bar{L}=L/O(G)$. Furthermore, either $\bar{G}'=\bar{G}$, or $\bar{G}'$ is the unique normal subgroup of $\bar{G}$ and %%%modifica 5 aprile 2018
%$[\bar{G}:\bar{L}]=2$, as in case $\bar{G}'\neq \bar{G}$ we have either $\bar{G}\cong SL^{(2)}(2,q)$, or $\bar{G}\cong %SU^{\pm}(2,q),\,q\equiv 1 \pmod 4$, or $\bar{G}\cong SL^{\pm}(2,q),\,q\equiv -1 \pmod 4$.
\begin{lemma}
\label{lemF3mar2018} In Case (C) of Proposition \ref{3apr2016}, $\cX$ has zero $p$-rank.
\end{lemma}
\begin{proof}
As $\bar{L}\cong SL(2,q)$ is a subgroup of $\aut(\bar{\cX})$, Result \ref{resdickson} yields that $\bar{\cX}$ is not rational. Furthermore, since a Sylow $2$-subgroup of $SL(2,q)$ is not cyclic, Lemma \ref{le23agosto2015} and Remark \ref{2sub} yield that $\bar{\cX}$ is neither elliptic. Therefore, $\gg(\bar{\cX})\geq 2$. If equality holds then Result \ref{resnaga} applied to $\gg(\bar{\cX})=2$ yields $\gamma(\bar{\cX})=0$, with just one exception, namely when $p=3$, and a Sylow $p$-subgroup of $\bar{L}$ has order $3$. This exception does not actually occur here as $q\geq 5$ in Case (C). Therefore, $\bar{\cX}$ has zero $p$-rank. We show that this holds true for $\gg(\bar{\cX})>2$. From the Hurwitz genus formula applied to $O(G)$,
$\gg(\cX)-1\geq |O(G)|(\gg(\bar{\cX})-1)$. This and (\ref{eq4jan2016}) give
$$|\bar{G}||O(G)|>900 (\gg(\bar{\cX})-1)^2|O(G)|^2\geq 900\gg(\bar{\cX})^2 \textstyle\frac{4}{9}|O(G)|^2\geq 900 \gg(\bar{\cX})^2|O(G)|$$
whence $|\bar{G}|>900 \gg(\bar{\cX})^2$. From Lemma \ref{4dic2015}, $\bar{\cX}$ has zero $p$-rank.

Furthermore, since $|\bar{G}|\le 2(q+1)q(q-1)$ and $SL(2,q)$ has an abelian subgroup of order $q+1$,
\begin{equation}
\label{eq8apr2016}
\frac{2(q-1)q}{|O(G)|(q+1)}\geq \frac{|G|}{(|O(G)|(q+1))^2}>900 \left(\frac{\gg({\cX})-1}{|O(G)|(q+1)}\right)^2\ge 900 \left(\frac{\gg(\bar{\cX})-1}{4(\gg(\bar{\cX})+1)}\right)^2\ge \frac{900}{64}=\frac{225}{16}.
\end{equation}
This yields
\begin{equation}
\label{eqA8jan2016A} |O(G)|< \frac{32}{225}q.
\end{equation}
Also, $q\ge 23$ holds.
%If $\gg(\bar{\cX})>2$ then the constant in (\ref{eq8apr2016}) can be improved to $225/16$ whence $q>11$. Let $%%\gg(\bar{\cX})=2$. From (i) of Result \ref{res60.79.108}, the maximum order of an abelian subgroup of $\bar{G}$ is at most% $12$. Since $SL(2,q)$ has an abelian group of order $2q$ this is only possible for $q=5$. On the other hand,  $q=5$ is not %possible by (\ref{eq8apr2016}).

From (i) of Result \ref{lem29dic2015},  a Sylow $p$-subgroup $\bar{Q}$ of $\bar{L}\cong SL(2,q)$ fixes a unique point $\bar{P}\in \bar{\cX}$. Here $|\bar{Q}|=q$, and the normalizer of $\bar{Q}$ in $\bar{L}$ is the semidirect product $\bar{Q}\rtimes \bar{T}$ with a cyclic group $\bar{T}$ of order $q-1$. Since this normalizer is a maximal subgroup of $\bar{L}$, it is the stabilizer of $\bar{P}$ in $\bar{L}$ otherwise $\bar{G}$ would fix $\bar{P}$ contradicting (iii) of Result \ref{res74}.

Let $o$ be the $O(G)$-orbit lying over $\bar{P}$ in the cover $\cX|\bar{\cX}$. The counter-image of $\bar{Q}\rtimes \bar{T}$ in the homomorphism $G\to \bar{G}$ is a subgroup $N$ of order $(q-1)q|O(G)|$ where $N$  is the subgroup of $G$ which preserves $o$. For a point $P\in o$, the stabilizer $H$ of $P$ in $N$ has order $(q-1)q|O(G)_P|$. Therefore, the (unique) Sylow $p$-subgroup $Q_1$ of $H$ has order $q_1=qp^a$ and a complement of $Q_1$ in $H$ is a cyclic subgroup $C_t$ of order $t=(q-1)b$ with $p\nmid b$. Hence $H=Q_1\rtimes C_t$, and $|o|p^ab=|O(G)|$. Now, (\ref{eq4jan2016}) yields
$q(q-1)p^ab(q+1)|o|>450\gg(\cX)^2$ whence
$$|H|^2=(q_1t)^2>\frac{(q-1)qp^ab}{(q+1)|o|}\, 450\gg^2={\qa} (p^ab)^2\frac{2(q-1)q}{|O(G)|(q+1)}\,900\gg(\cX)^2.$$
This together with (\ref{eq8apr2016}) give
$$|H|=q_1t>30 \gg(\cX).$$
From Lemma \ref{22dic2015} and Remark \ref{rem3jan2016}, either $\cX$ has zero $p$-rank, or $Q_1$ has just one short orbit $\theta$ other than its fixed point $P$. Observe that $\theta$ is contained in $o$. In fact, each $Q_1$-orbit in $o$ is short as $|o|\leq |O(G)|<q<q_1$ by (\ref{eqA8jan2016A}). %Since $\bar{G}$ From \cite[Chapter II, Satz8.28]{huppert1967},
Therefore $o=\theta\cup\{P\}$, and hence $|o|=1+p^h$ for some $h\ge 0$. Since $|o|$ divides $|O(G)|$ this implies that  $|O(G)|$ is even, a contradiction. Thus $\cX$ has zero $p$-rank.
\end{proof}
\begin{lemma}
\label{lemG3mar2018} In Case (C) of Proposition \ref{3apr2016}, $Z(G)$ is a cyclic prime to $p$ group and contains $O(G)$ as an index $2$ subgroup.
\end{lemma}
\begin{proof}
We adapt the arguments used in Section \ref{sspsl}. For this purpose, let $M$ be the normal subgroup of $G$ containing $O(G)$ such that $M/O(G)=Z(G/O(G))$. Since $\bar{G}=G/O(G)$ is isomorphic to either $SL(2,q)$, or $SL^{(2)}(2,q)$, or $SL^{\pm}(2,q)$, or
$SU^{\pm}(2,q)$, we have $|Z(\bar{G})|=2$. Thus $|M|=2|O(G)|$. Since $\bar{L}\cong SL(2,q)$, we have that $\bar{L}$ contains $Z(\bar{G})$. Hence $M\le L$, and
$\hat{L}=L/M\cong (L/O(G))/(M/O(G))\cong PSL(2,q)$. From Lemma \ref{4dic2015} applied to $\cX/M$, we have that  $\hat{G}=G/M$ is isomorphic to either $PSL(2,q)$ or $PGL(2,q)$ according as $\hat{G}=\hat{L}$ or $[\hat{G}:\hat{L}]=2$.

For every $h\in G$ the map $\varphi_h$ taking $u\in M$ to $h^{-1}uh$ is an automorphism of $M.$ Thus the map $\Phi:\,h\mapsto \varphi_h$ is a homomorphism from $G$ into $\aut(M)$. Now, after replacing  $O(G)$ by $M$ and Lemma \ref{lemC1mar2018} by Lemma
\ref{lemF3mar2018}, the analog argument used in the proof of Lemma \ref{lemC1mar2018} shows that $M$ is a cyclic prime to $p$ group. In particular, $\ker(\Phi)\ge M$. Since either $\hat{G}=\hat{L}$, or $\hat{L}$ is the only non-trivial normal subgroup of $\hat{G}$, two cases arise namely $\ker(\Phi)=M$ and $\ker(\Phi)\ge L$.

In the former case, $\Phi$ induces a faithful action of $\hat{L}$ on $M$. If $\Lambda$ is a non-trivial orbit then $|\Lambda|<2|O(G)|$, and hence $|\Lambda|<q$ by (\ref{eqA8jan2016A}). On the other hand, $|\Lambda|\le q+1$ by the analog argument in the proof of Lemma \ref{lemC1mar2018}; a contradiction.

In the latter case, $\ker(\Phi)\geq L$ and hence $Z(L)=M$. Let $Q$ be a Sylow $p$-subgroup of $G$. Then $Q\le L$ as either $G=L$ or $[G:L]=2$. From Lemma \ref{lemF3mar2018} and (i) of Result \ref{lem29dic2015}, $Q$ fixes a unique point $P$. Since $Z(L)=M$, $M$ fixes $P$. From Lemma \ref{lemF3mar2018} and (ii) of Result \ref{lem29dic2015}, $G_P$ coincides with the normalizer $N_G(Q)$ of $Q$ in $G$. From Result \ref{reszass}, a complement $C$ of $Q$ contains $M$. If $L\lneqq G$ then $C$ contains an element of $G\setminus L$. Therefore, some element in $G\setminus L$ centralizes $O(G)$. This shows that $Z(G)=M$.
\end{proof}

\section{Proof of  Theorem \ref{princC}}\label{sec8}
From Theorem \ref{princB}, $\gamma(\cX)=0.$ If $O(G)=\{1\}$, Theorem \ref{princC} is a corollary of Lemma \ref{4dic2015}, taking into account that for the groups $G$ listed in Lemma \ref{4dic2015} $p(G)$ is $PGL(2,q)$, $PSL(2,q)$ or $SL(2,q)$.

Therefore, $O(G)$ is assumed to be non-trivial.
From (i) of Result \ref{lem29dic2015}, each Sylow $p$-subgroup fixes a point of $\cX$. Therefore, the set $\Omega$ consisting of the fixed points of the Sylow $p$-subgroups of $G$ is non-empty. There is a bijection between $\Omega$ and the set of all Sylow $p$-subgroups of $G$. Since the Sylow $p$-subgroups of $G$ are pairwise conjugate in $G$, $\Omega$ is a $G$-orbit. Therefore, $G$ has a transitive permutation representation $\Phi$ on $\Omega$. Let $\tilde{G}=G/\ker(\Phi)$. Then $Z(G)\le \ker(\Phi)$ as $Z(G)$ fixes $\Omega$ pointwise. Also, from Theorem \ref{princB}, $Z(G)$ is a prime to $p$ subgroup, and hence Sylow $p$-subgroups of $G$ are isomorphic to those of $G/Z(G)$. Note that $\ker(\Phi)$ is a proper subgroup of $G$ as otherwise $G$ would have just one Sylow $p$-subgroup.

If Case (A) or Case (B) in Proposition \ref{3apr2016} holds, then $\bar{G}=G/O(G)$ has trivial center and $Z(G)=O(G)$ by Theorem \ref{princB}.
 Therefore, $Z(G)=\ker{\Phi}$ since $\ker(\Phi)/O(G)$ is a normal subgroup of $\bar{G}$.
 Then
$\tilde{G}=\bar{G}$ acts (by conjugation) on $\Omega$ faithfully in the same way as either $PSL(2,q)$, or $PGL(2,q)$, or $PSU(3,q)$, or $PGU(3,q)$, on the set of their Sylow $p$-subgroups, that is, in their usual doubly transitive permutation representation. In particular, $q=|\Omega|-1$ for $PSL(2,q)$ and $PGL(2,q)$ while $q^3=|\Omega|-1$ for $PSU(3,q)$ and $PGU(3,q)$.
%In particular, any Sylow $p$-subgroup fixes a point and acts as regular permutation group on the remaining $|\Omega|-1$ points.
%Also, any $p$-Sylow subgroup of $\tilde{G}$ acts regularly on the set of all the remaining Sylow $p$-subgroups of $\tilde{G}$.

If case (C) in Proposition \ref{3apr2016} occurs, then either $G/Z(G)\cong PSL(2,q)$, or $G/Z(G)\cong PGL(2,q)$, and $[Z(G):O(G)]=2$ by Theorem \ref{princB}. In the former case, $G/Z(G)$ is simple, in the latter one $G/Z(G)$ has a unique proper normal subgroup, and this subgroup is isomorphic to $PSL(2,q)$. Therefore, $\ker(\Phi)=Z(G)$, and $\tilde{G}$  acts on $\Omega$ faithfully as either $PSL(2,q)$, or $PGL(2,q)$, on the set of their Sylow $p$-subgroups, that is, in their usual doubly transitive permutation representation.
%In particular, any Sylow $p$-subgroup fixes a point and acts as regular permutation group on the remaining $|\Omega|-1$ points.

Let $S_p$ be a Sylow $p$-subgroup of $G$, and look at the action of $S_p$ on $\Omega$. Obviously $S_p$ fixes the point $P\in \Omega$ represented by $S_p$ itself. From (i) of Result \ref{lem29dic2015}, no non-trivial element of $S_p$ fixes another point, that is, $S_p$ acts semiregularly on the points of $\Omega$ other than $P$. Actually, $S_p$ is regular on $\Omega\setminus \{P\}$, as $|S_p|=|\Omega|-1$.
%the largest $p$-power divisor of $|G|$ being $q$ for $PSL(2,q),\,PGL(2,q),\, SL(2,q),\,SL(2,q)^{\pm}$ and $q^3$ for $PSU(3,q)$ and $PGU(3,q)$.
Furthermore, from (i) of Result \ref{res74}, $S_p$ is a normal subgroup of the stabilizer $G_P$ of $P$ in $G$.

Therefore, the pair $(\Omega,G)$ is a group space satisfying the hypothesis of Result \ref{reshering} where the normal closure $S$ of $Q$ is the subgroup $p(G)$ of $G$ generated by all Sylow $p$-subgroups of $G$. In Result \ref{reshering}, there are several possibilities for $S$ and hence for our $p(G)$, but those consistent with Proposition \ref{3apr2016} are only four, namely $PSL(2,q)$, $SL(2,q)$, $PSU(3,q)$ with $q\equiv 1 \pmod 4$ and $SU(3,q)$ with $q\equiv 5 \pmod {12}$. This completes the proof of Theorem \ref{princC}.

\section{Examples for Theorem \ref{princC}}
\label{examples}
In Remark \ref{rem5jan2016} we have already pointed out that Hermitian curves and Roquette curves provide example for the cases $p(G)=PSU(3,q), \,PSL(2,q),\,SL(2,q)$.

We show that the  GK \textcolor{black}{curves} \cite{GK} have quotient curves that provide more examples for Theorem \ref{princC}.

Let $q$ be a power of $p$ such that $q\equiv 1 \pmod 4$.

%The Hermitian curve $\cY$ has even genus $\gg(\cY)=\ha q(q-1)$. Up to an isomorphism, $\aut(\cY)=PGU(3,q)$ and $p(\aut(\cY))=PSU(3,q)$ as in the second case in Theorem \ref{princC}.

%%
 %This gives an example of type (IIa) and (IIb) in terms of Theorem \ref{princD} and it has a subgroup $M\cong PSU(3,q)$ as in the first case in Theorem \ref{princC}.

%Take $\cY$ with its homogeneous equation $X^{q+1}+Y^{q+1}+Z^{q+1}=0$.
%Let $T=\langle t \rangle$ be the subgroup of $PGU(3,q)$ generated by $t:%(X,Y,Z)\to (wX,wY,Z)$ where $w\in \K$ has order $m$ for some divisor $m$ of $q+1$. Here $t$ (and every non-trivial element of $T$) fixes exactly $q+1$ points of $\cY$, those lying on the line of equation $Z=0$ at infinity. From the Hurwitz genus formula, the quotient curve $\cX_m=\cY/T$ has even genus $\gg(\cX_m)=\ha((q-1-m)(q+1)/m+2)$. Then $G=N_{PGU(3,q)}(T)/T$ is a subgroup of $\aut(\cX_m)$. If $m$ is even then $G\cong PGL(2,q)\times C_{(q+1)/m}$ while for odd $m$, $G\cong SU(2,q)^{\pm}\times C_{(q+1)/2m}$. An index $2$ subgroup of $G$ is isomorphic to $PSL(2,q)\times C_{(q+1)/2}$ for $m$ even and to $SL
%(2,q)\times C_{(q+1)/2m}$ for $m$ odd. In all cases, $|G|>900 \gg(\cX)^2$ for %$m$ big enough, namely $m>225$, and examples   for $p(G)\cong PSL(2,q), PGL(2,q), SL(2.q)$  in Theorem \ref{princC}
%Here $G_m$ has a subgroup $M$ of first or second types in Theorem \ref{princC}.
%are obtained.

From \cite{GK}, the GK curve $\cZ$ has genus $\gg(\cZ)=\ha\,(q^3+1)(q^2-2)+1$ which is even for $q\equiv 1 \pmod 4$, and $\aut(\cZ)$ has order $(q^3+1)q^3(q^2-1)(q^2-q+1)$. More precisely, $\aut(\cZ)\cong PSU(3,q)\times C_{q^2-q+1}$ for $q\not\equiv -1 \pmod 3$ while $\aut(\cZ)$ has an index $3$ subgroup $G\cong SU(3,q)\times C_{(q^2-q+1)/3}$ for $q\equiv -1 \pmod 3$; see \cite{GK}. In particular, $|\aut(\cZ)|\approx 8\gg(\cX)^2$. Furthermore, $\aut(\cZ)$ has exactly two short orbits. One of them has size $q^3+1$ and the kernel of the permutation representation of $\aut(\cZ)$ on it is the cyclic subgroup $C_{q^2-q+1}$ of $\aut(\cZ)$ while no non-trivial element of $C_{q^2-q+1}$
fixes a point in the other short orbit. The quotient curves of $\cZ$ have been under investigations; see \cite{fg,gqz}. Here we limit ourselves to a special case, also discussed in \cite{GK}.

For a divisor $u$ of $q^2-q+1$, the group $C_{q^2-q+1}$ contains a subgroup $C_u$ of  order $u$. Let $\cX_u=\cZ/C_u$ be the quotient
curve of $\cX$ with respect to $C_u$. Since $C_u$ is tame, the Hurwitz genus
formula gives
$$(q^3+1)(q^2-2)=2\gg(\cY)-2=u(2\gg(\cX_d)-2)+(d-1)(q^3+1),$$ whence
$$\gg(\cX_d)=\frac{1}{2} \left(\frac{(q^3+1)(q^2-u-1)}{u}+2\right).$$
In particular, $\gg(\cX_u)$ is even. Furthermore, since $C_u$ is a normal subgroup of $\aut(\cX)$, $\aut(\cX)/C_u$ is a subgroup $G_u$ of $\aut(\cX_u)$ such that
$$|G_u|=\frac{q^3(q^3+1)(q^2-1)(q^2-q+1)}{u}.$$
Comparison $|G_u|$ with $\gg(\cX_d)$ shows that if $u\geq 450$ then
$|\aut(\cX_u)|\geq |G_u|>900\gg(\cX_u)^2$.

If $q\not\equiv -1 \pmod 3$ then both $(PSU(3,q)\times C_u)/C_u\cong PSU(3,q)$ and $ C_{q^2-q+1}/C_u\cong C_{(q^2-q+1)/u}$ are subgroups of $G_u$, and their intersection is trivial. Therefore, they generate a group isomorphic to $PSU(3,q)\times C_{(q^2-q+1)/u}$ whose order is equal to that of $G_u$. Thus $G_d\cong PSU(3,q)\times C_{(q^2-q+1)/u}$. This gives an example, not isomorphic to the Hermitian curve, for the case $p(G)\cong PSU(3,q)$ with $q\equiv 1 \pmod 3$ in Theorem \ref{princC}.

If $q\equiv -1\pmod 3$, take the index $3$ subgroup $G\cong SU(3,q)\times C_{(q^2-q+1)/3}$ together with a subgroup $C_u$ of $C_{(q^2-q+1)/3}$. Then $u$ is not divisible by $3$ and we may replace $G_u=\aut(\cX)/C_u$ with $G/C_u$ in the above argument. This time we obtain $G/C_u\cong SU(3,q)\times C_{(q^2-q+1)/3u}$ which gives  an example for the case $p(G)\cong SU(3,q)$ with $q\equiv -1 \pmod 3$ in Theorem \ref{princC}.

\vspace{0,5cm}\noindent {\em Authors' addresses}:

\vspace{0.2 cm} \noindent Massimo GIULIETTI \\
Dipartimento di Matematica e Informatica
\\ Universit\`a degli Studi di Perugia \\ Via Vanvitelli, 1 \\
06123 Perugia
(Italy).\\
 E--mail: {\tt massimo.giulietti@unipg.it}

\vspace{0.2cm}\noindent G\'abor KORCHM\'AROS\\ Dipartimento di
Matematica\\ Universit\`a della Basilicata\\ Contrada Macchia
Romana\\ 85100 Potenza (Italy).\\E--mail: {\tt
gabor.korchmaros@unibas.it }

    \end{document}